\documentclass{article}

\usepackage{graphicx}

\usepackage{hyperref}

\usepackage{amsfonts, amsmath, amsthm}

\usepackage{skak}
\usepackage{tabularx}
\usepackage{color}

\usepackage[margin=3cm]{geometry}

\usepackage{tikz}
\usetikzlibrary{calc, arrows, decorations.markings, decorations.pathmorphing, positioning, decorations.pathreplacing}
\makeatletter
\tikzset{nomorepostaction/.code={\let\tikz@postactions\pgfutil@empty}}
\makeatother

\newtheorem{thm}{Theorem}[section]
\newtheorem{cor}[thm]{Corollary}

\newtheorem{prop}[thm]{Proposition}

\newcommand{\To}{\longrightarrow}
\newcommand{\C}{\mathbb{C}}
\newcommand{\Z}{\mathbb{Z}}

\newcommand{\R}{\mathbb{R}}

\newcommand{\CP}{\mathbb{CP}}

\begin{document}

\title{Contact topology and holomorphic invariants via elementary combinatorics}

\author{Daniel V. Mathews}

\date{}

\maketitle

\begin{abstract}
In recent times a great amount of progress has been achieved in symplectic and contact geometry, leading to the development of powerful invariants of 3-manifolds such as Heegaard Floer homology and embedded contact homology. These invariants are based on holomorphic curves and moduli spaces, but in the simplest cases, some of their structure reduces to some elementary combinatorics and algebra which may be of interest in its own right. In this note, which is essentially a light-hearted exposition of some previous work of the author, we give a brief introduction to some of the ideas of contact topology and holomorphic curves, discuss some of these elementary results, and indicate how they arise from holomorphic invariants.
\end{abstract}

\tableofcontents

\section{Introduction}

In recent years a great amount of progress has been achieved in understanding the structures of symplectic and contact geometry. Powerful invariants of manifolds and their contact and symplectic structures have been defined, such as Heegaard Floer homology and contact homology. These theories are based on generalised Cauchy-Riemann equations, holomorphic curves and their moduli spaces. 

In a series of papers, the author has developed some of the structure that arises in some of the most simple cases of the holomorphic invariants known as sutured Floer homology and embedded contact homology \cite{Me09Paper, Me10_Sutured_TQFT, Me11_torsion_tori, Me12_itsy_bitsy, Mathews_Schoenfeld12_string}. This note is essentially a light-hearted exposition of some of that subject matter.

The structures that we discuss here are algebraic and combinatorial, and entirely elementary. They may be of interest for their own sake, as well as for their relation to other subjects, including quantum information theory, representation theory, topological quantum field theory, and especially contact geometry. That all this structure can arise from the most simple cases of sutured Floer homology and embedded contact homology, testifies to the power of these invariants.

It is our desire here to convey some of these results to a broad mathematical audience. Consequently, this note assumes no knowledge of symplectic or contact geometry or holomorphic curves. Since the subject matter is in a certain sense a ``combinatorialization of contact geometry'', we will give an exposition, as we proceed, of how these elementary results relate to contact geometry --- and this may serve as an unorthodox introduction to some of the ideas of contact geometry. 

The reader who is only interested in combinatorics or algebra can easily skip the sections on symplectic and contact geometry and holomorphic curves. The reader without such background, but who wishes to know where the subject matter comes from, can hopefully gain here some idea of the types of considerations involved, whether in symplectic or contact geometry or holomorphic curves, and is encouraged to follow the references for further details.

\section{Symplectic and contact geometry}

We begin with a little --- very, very little --- about what symplectic and contact geometry are, and where some holomorphic invariants come from. For a proper introduction to symplectic geometry we refer to \cite{McDuffSalamon_Introduction} and to \cite{Geiges_Introduction} or \cite{Et02} for contact geometry.  We will give essentially no proofs in this section, just assert some basic facts. 

None of this is needed for the combinatorics and algebra that we will shortly discuss, but it will be useful as we proceed to make connections with contact geometry and holomorphic invariants. The reader who is solely interested in combinatorial and algebraic aspects can safely skip this section.

\subsection{Symplectic geometry}

Symplectic geometry is the mathematical structure of Hamiltonian mechanics. A symplectic manifold is a pair
\[
(M, \omega),
\]
where $M$ is a smooth manifold and $\omega$ is a closed non-degenerate differential $2$-form on $M$. This implies that $M$ is even-dimensional.

The key property that allows this structure to produce mechanics is that any smooth function $H: M \To \R$ (called a \emph{Hamiltonian}) has a naturally associated vector field $X_H$ on $M$. Namely, from $H$ we obtain a differential $1$-form $dH$, and then the non-degeneracy of $\omega$ gives the vector field $X_H$ via the equation
\[
\omega( X_H, \cdot ) = dH.
\]
(Different authors have different sign conventions in this equation, but this is the idea.)

The fact that $\omega$ is closed implies that flowing along $X_H$ leaves $\omega$ unchanged:
\[
L_{X_H} \omega = i_{X_H} d\omega + d i_{X_H} \omega = i_{X_H} 0 + d(dH) = 0.
\]
The physical interpretation is that $M$ is the phase space (``space of states'') of the universe, $H$ gives the energy of any state, and $X_H$ is the time evolution of the universe.

The simplest example of a symplectic manifold is $\R^{2n}$ with the symplectic form
\[
\omega = \sum_{j=1}^n dx_j \wedge dy_j.
\]
Here each $y_j$ can be considered a ``position coordinate'' and each $x_j$ a corresponding ``conjugate momentum''.

\subsection{Symplectic vs. complex geometry}

The fact that symplectic geometry only arises in even numbers of dimensions suggests a similarity to complex geometry. Indeed this is even the etymological root: Weyl introduced the word ``symplectic'' as a Greek version of the word ``complex'' \cite{Weyl_Classical}. The Latin \emph{complexus} and the Greek \emph{symplektikos} both mean ``braided together''. 

There is a difference of course. In a complex manifold every point has a neighbourhood homeomorphic to $\C^n$. In particular, for every direction there is also ``$i$ times'' that direction. This is very different from having a closed non-degenerate $2$-form.

The relationship between symplectic and complex geometry has been exploited to great effect in the last 25 years or so, with the use of \emph{holomorphic curves}, starting with the work of Gromov in 1985 \cite{Grom}, leading to great advances in the understanding of symplectic geometry.

In particular, any symplectic manifold $(M, \omega)$ has an \emph{almost complex structure}. An almost complex structure is a map which is like ``multiplication by $i$ at every point of $M$''. That is, $J$ is a map
\[
J \; : \; TM \To TM
\]
which takes every fibre $T_p M \to T_p M$ and satisfies $J^2 = -1$. 

\begin{center}
\def\svgwidth{120pt}
\begingroup%
  \makeatletter%
  \providecommand\color[2][]{%
    \errmessage{(Inkscape) Color is used for the text in Inkscape, but the package 'color.sty' is not loaded}%
    \renewcommand\color[2][]{}%
  }%
  \providecommand\transparent[1]{%
    \errmessage{(Inkscape) Transparency is used (non-zero) for the text in Inkscape, but the package 'transparent.sty' is not loaded}%
    \renewcommand\transparent[1]{}%
  }%
  \providecommand\rotatebox[2]{#2}%
  \ifx\svgwidth\undefined%
    \setlength{\unitlength}{219.48855534bp}%
    \ifx\svgscale\undefined%
      \relax%
    \else%
      \setlength{\unitlength}{\unitlength * \real{\svgscale}}%
    \fi%
  \else%
    \setlength{\unitlength}{\svgwidth}%
  \fi%
  \global\let\svgwidth\undefined%
  \global\let\svgscale\undefined%
  \makeatother%
  \begin{picture}(1,0.57415684)%
    \put(0,0){\includegraphics[width=\unitlength]{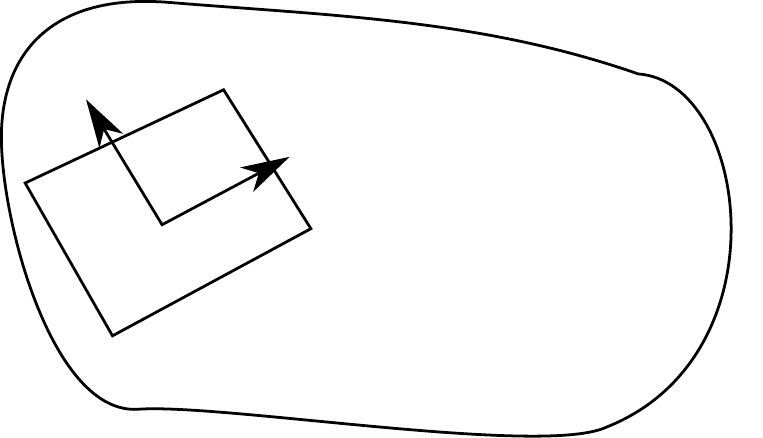}}%
    \put(0.74377454,0.25326604){\color[rgb]{0,0,0}\makebox(0,0)[lb]{\smash{$M$}}}%
    \put(0.3740841,0.37302493){\color[rgb]{0,0,0}\makebox(0,0)[lb]{\smash{$v$}}}%
    \put(0.07989354,0.45893893){\color[rgb]{0,0,0}\makebox(0,0)[lb]{\smash{$Jv$}}}%
    \put(0.15799714,0.24545563){\color[rgb]{0,0,0}\makebox(0,0)[lb]{\smash{$p$}}}%
    \put(0.17361784,0.07362762){\color[rgb]{0,0,0}\makebox(0,0)[lb]{\smash{$T_p M$}}}%
  \end{picture}%
\endgroup%

\end{center}

An almost complex structure is often required to be \emph{compatible} with the symplectic form. Roughly this means that $J$ and $\omega$ behave like $i$ and $\sum_j dx_j \wedge dy_j$ in $\C^n$ (where the $z_j = x_j + i y_j$ are the coordinates). (Precisely, it means that $\omega( v, w ) = \omega (Jv, Jw)$ for all tangent vectors $v,w$; and $\omega( v, Jv) > 0$ for all tangent vectors $v \neq 0$.)

It turns out that any symplectic manifold has a compatible almost complex structure. Moreover, all almost complex structures on $(M, \omega)$ are homotopic. (The space of almost complex structures on $M$ is the space of sections of a fibre bundle over $M$ with contractible fibres.) However, an almost complex structure by no means implies the existence of a complex structure. An almost complex structure only requires a pointwise condition. A complex structure requires local charts to $\C^n$ with holomorphic transition functions, which is a far more onerous condition. To drop the ``almost'' requires an additional condition, the vanishing of the Nijenhuis tensor.

For proper discussion of these issues we refer to \cite{McDuffSalamon_Introduction} or \cite{McDuff_Salamon_J-holomorphic}.

\subsection{Holomorphic curves in symplectic manifolds}

Gromov in \cite{Grom} had the idea of considering \emph{holomorphic curves in symplectic manifolds}. This has led to many developments, including the development of \emph{Floer homology} theories. 

The point of this note is not really to explain any such homology theories, or the detail of why they work, or how they work, or how to compute them. The point is to discuss some of the structure obtained from them in some very simple cases, which is elementary and of independent interest.

Nonetheless, for the interested reader, we can summarise some of the ideas involved, very roughly and avoiding all details, as follows. Readers not interested in holomorphic curves should skip to the next section.
\begin{itemize}
\item
Start with a symplectic manifold $(M, \omega)$ you want to understand.
\item
Introduce a compatible almost complex structure $J$ on $(M, \omega)$. As we just noted, one always exists, and any choice is homotopic to any other.
\item
Take a Riemann surface $(\Sigma, i)$ and consider maps
\[
u: (\Sigma, i) \To (M, J)
\]
which are \emph{holomorphic},\footnote{It is common to say ``pseudo-holomorphic'' or ``$J$-holomorphic'' to indicate that the target has an almost complex structure, rather than a complex structure. We prefer simply to say holomorphic here and there should be no ambiguity.} meaning that the following diagram commutes.
\[
\begin{tikzpicture}
\draw (0,0) node {$T\Sigma$};
\draw (3,0) node {$TM$};
\draw (0,-2) node {$T\Sigma$};
\draw (3,-2) node {$TM$};
\draw [->] (0.5,0) -- (2.5,0)
	node [above, align=center, midway] {$Du$};
\draw [->] (0.5,-2) -- (2.5,-2)
	node [above, align=center, midway] {$Du$};
\draw [->] (0,-0.4) -- (0,-1.6)
	node [left, align=center, midway] {$i$};
\draw [->] (3,-0.4) -- (3,-1.6) node [right, align=center, midway] {$J$};
\end{tikzpicture}
\]
I.e., the \emph{Cauchy--Riemann equations} hold: $Du \circ i = J \circ Du$.

\begin{center}
\def\svgwidth{300pt}
\begingroup%
  \makeatletter%
  \providecommand\color[2][]{%
    \errmessage{(Inkscape) Color is used for the text in Inkscape, but the package 'color.sty' is not loaded}%
    \renewcommand\color[2][]{}%
  }%
  \providecommand\transparent[1]{%
    \errmessage{(Inkscape) Transparency is used (non-zero) for the text in Inkscape, but the package 'transparent.sty' is not loaded}%
    \renewcommand\transparent[1]{}%
  }%
  \providecommand\rotatebox[2]{#2}%
  \ifx\svgwidth\undefined%
    \setlength{\unitlength}{500.12733106bp}%
    \ifx\svgscale\undefined%
      \relax%
    \else%
      \setlength{\unitlength}{\unitlength * \real{\svgscale}}%
    \fi%
  \else%
    \setlength{\unitlength}{\svgwidth}%
  \fi%
  \global\let\svgwidth\undefined%
  \global\let\svgscale\undefined%
  \makeatother%
  \begin{picture}(1,0.32779624)%
    \put(0,0){\includegraphics[width=\unitlength]{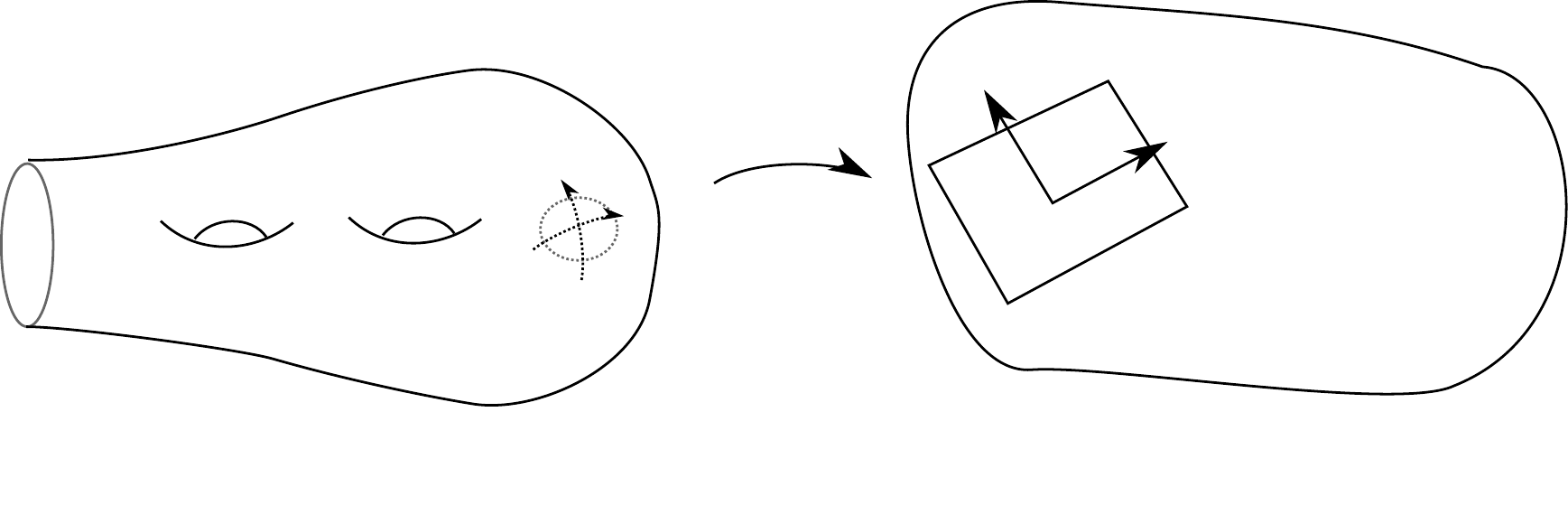}}%
    \put(0.75603608,0.01101319){\color[rgb]{0,0,0}\makebox(0,0)[lb]{\smash{$M$}}}%
    \put(0.74232531,0.23952641){\color[rgb]{0,0,0}\makebox(0,0)[lb]{\smash{$v$}}}%
    \put(0.61321527,0.27723109){\color[rgb]{0,0,0}\makebox(0,0)[lb]{\smash{$Jv$}}}%
    \put(0.34585479,0.09327796){\color[rgb]{0,0,0}\makebox(0,0)[lb]{\smash{$\mathbb{C}$}}}%
    \put(0.47039451,0.25437977){\color[rgb]{0,0,0}\makebox(0,0)[lb]{\smash{$u$}}}%
    \put(0.04307475,0.01215575){\color[rgb]{0,0,0}\makebox(0,0)[lb]{\smash{$\Sigma$}}}%
  \end{picture}%
\endgroup%

\end{center}

\item
If we consider holomorphic curves with appropriate constraints such as marked points, fixed degree and boundary conditions, and provided sufficient transversality conditions are satisfied, then the space of holomorphic curves will be a \emph{finite-dimensional} orbifold-like space called a \emph{moduli space}.

(This is a vast generalisation of elementary facts such as the following. The space of holomorphic maps $\CP^1 \To \CP^1$ of degree $1$ fixing $3$ points is a point. The space of holomorphic maps $\CP^1 \To \CP^1$ of degree $1$ fixing $2$ points is an annulus.)

\item
More generally, under these favourable conditions, there is an index theory for holomorphic curves. This is a version of the Riemann--Roch theorem and a special case of the Atiyah--Singer index theorem; the Cauchy--Riemann equations give a $\bar{\partial}$ operator and its index is computed in terms of topological data. 
The dimension of the moduli space is given in terms of the symplectic topology of the constraints. Such moduli spaces intricately encode topological data about the manifold. 

\item
A moduli space has a natural \emph{compactification} which is the subject of the \emph{Gromov compactness theorem}. The compactified moduli space has a stratified boundary, and the strata are moduli spaces of ``degenerate'' holomorphic curves such as nodal surfaces.

\begin{center}
\def\svgwidth{300pt}
\begingroup%
  \makeatletter%
  \providecommand\color[2][]{%
    \errmessage{(Inkscape) Color is used for the text in Inkscape, but the package 'color.sty' is not loaded}%
    \renewcommand\color[2][]{}%
  }%
  \providecommand\transparent[1]{%
    \errmessage{(Inkscape) Transparency is used (non-zero) for the text in Inkscape, but the package 'transparent.sty' is not loaded}%
    \renewcommand\transparent[1]{}%
  }%
  \providecommand\rotatebox[2]{#2}%
  \ifx\svgwidth\undefined%
    \setlength{\unitlength}{488.16840907bp}%
    \ifx\svgscale\undefined%
      \relax%
    \else%
      \setlength{\unitlength}{\unitlength * \real{\svgscale}}%
    \fi%
  \else%
    \setlength{\unitlength}{\svgwidth}%
  \fi%
  \global\let\svgwidth\undefined%
  \global\let\svgscale\undefined%
  \makeatother%
  \begin{picture}(1,0.26570928)%
    \put(0,0){\includegraphics[width=\unitlength]{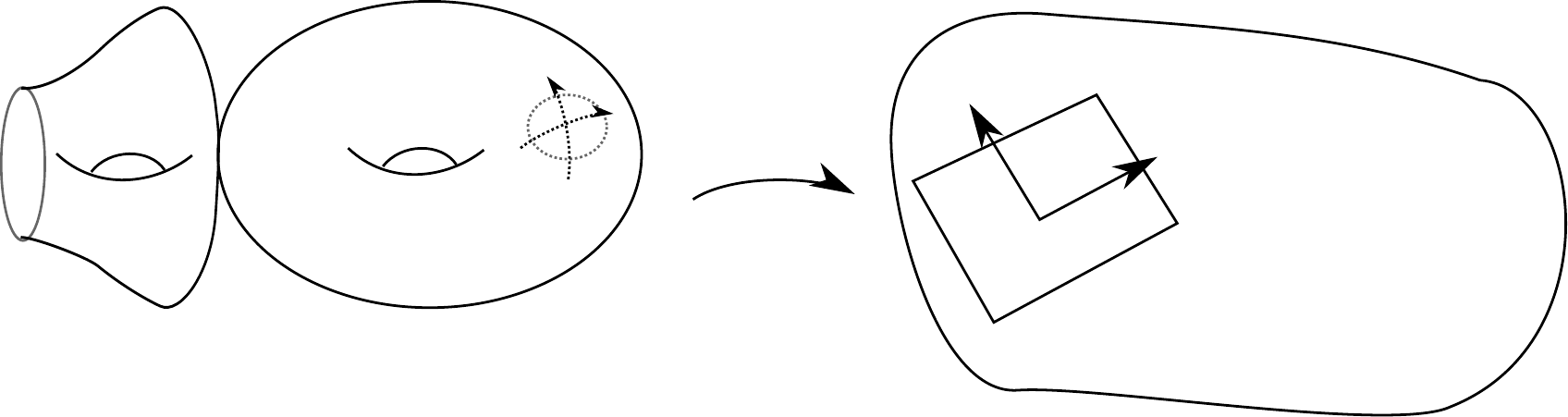}}%
    \put(0.82263404,0.04012754){\color[rgb]{0,0,0}\makebox(0,0)[lb]{\smash{$M$}}}%
    \put(0.73601292,0.16771815){\color[rgb]{0,0,0}\makebox(0,0)[lb]{\smash{$v$}}}%
    \put(0.60374,0.2063465){\color[rgb]{0,0,0}\makebox(0,0)[lb]{\smash{$Jv$}}}%
    \put(0.45742049,0.18293538){\color[rgb]{0,0,0}\makebox(0,0)[lb]{\smash{$u$}}}%
  \end{picture}%
\endgroup%

\end{center}

\item
Some information can be extracted from a moduli space by exploiting the codimension-1 parts of its boundary, and defining \emph{homology theories} based on it. Roughly, and although there are several different approaches, one takes a chain complex generated by sets of boundary conditions for holomorphic curves. One then defines a differential $\partial$ counting holomorphic curves between two sets of boundary conditions, at positive and negative ends respectively. Under favourable conditions, the boundary structure of the moduli space and the index theory of solutions to the Cauchy-Riemann equations give $\partial^2 = 0$. 

This is analogous to how the singular homology of a smooth manifold can be obtained from a Morse function $f$ via a chain complex (the \emph{Morse complex}) generated by critical points of $f$, and a differential counting gradient trajectories between critical points. (Critical points are ``boundary conditions for gradient trajectories'').

\begin{center}
\def\svgwidth{100pt}
\begingroup%
  \makeatletter%
  \providecommand\color[2][]{%
    \errmessage{(Inkscape) Color is used for the text in Inkscape, but the package 'color.sty' is not loaded}%
    \renewcommand\color[2][]{}%
  }%
  \providecommand\transparent[1]{%
    \errmessage{(Inkscape) Transparency is used (non-zero) for the text in Inkscape, but the package 'transparent.sty' is not loaded}%
    \renewcommand\transparent[1]{}%
  }%
  \providecommand\rotatebox[2]{#2}%
  \ifx\svgwidth\undefined%
    \setlength{\unitlength}{312.06857262bp}%
    \ifx\svgscale\undefined%
      \relax%
    \else%
      \setlength{\unitlength}{\unitlength * \real{\svgscale}}%
    \fi%
  \else%
    \setlength{\unitlength}{\svgwidth}%
  \fi%
  \global\let\svgwidth\undefined%
  \global\let\svgscale\undefined%
  \makeatother%
  \begin{picture}(1,0.94147405)%
    \put(0,0){\includegraphics[width=\unitlength]{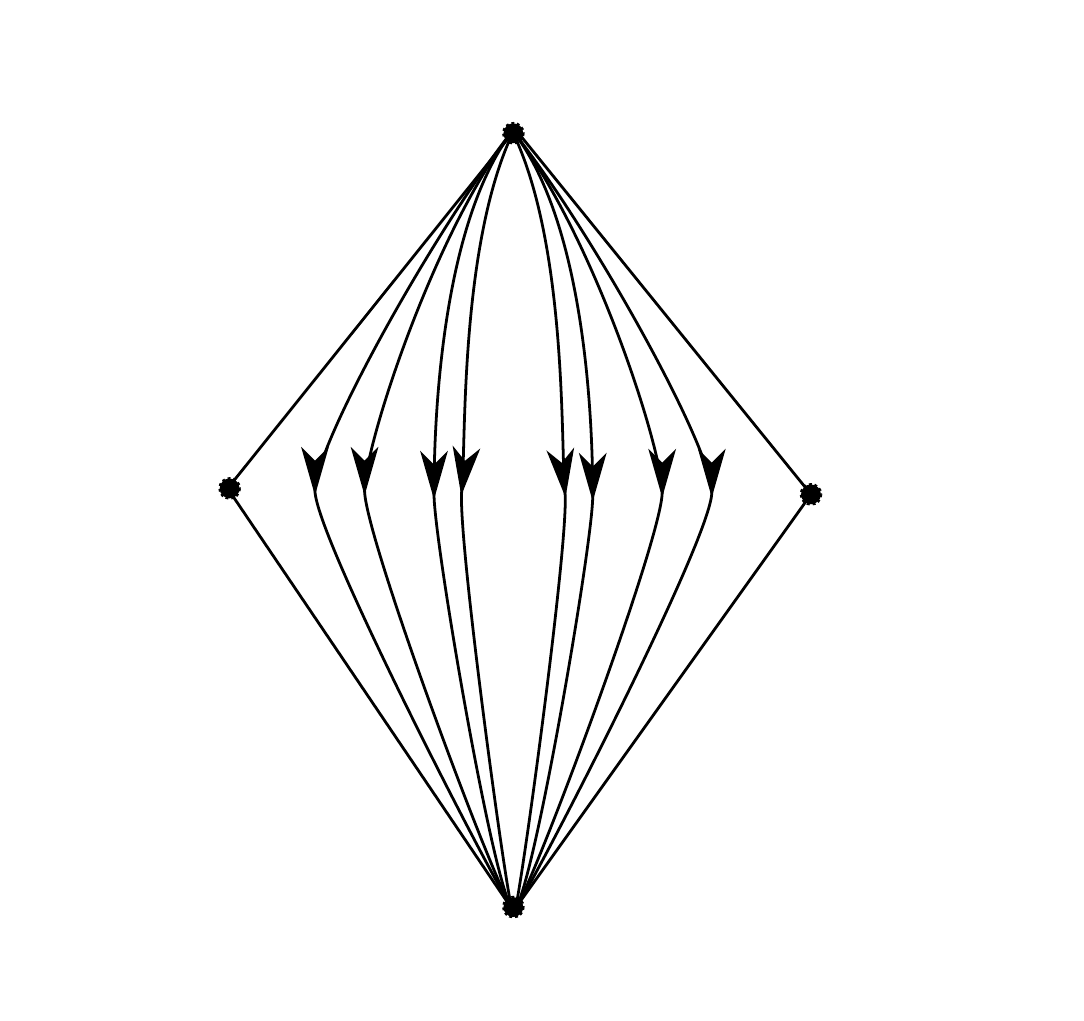}}%
    \put(0.42983228,0.86792193){\color[rgb]{0,0,0}\makebox(0,0)[lb]{\smash{Index $i$}}}%
    \put(0.79239,0.47789775){\color[rgb]{0,0,0}\makebox(0,0)[lb]{\smash{$i-1$}}}%
    \put(-0.00230716,0.49620877){\color[rgb]{0,0,0}\makebox(0,0)[lb]{\smash{$i-1$}}}%
    \put(0.41335242,0.00913634){\color[rgb]{0,0,0}\makebox(0,0)[lb]{\smash{$i-2$}}}%
  \end{picture}%
\endgroup%

\end{center}

\item
Depending on the details of how this chain complex is constructed, one can obtain various flavours of \emph{Floer homology}, \emph{contact homology} or \emph{symplectic field theory}. An intricate algebraic structure of generating functions can be used to keep track of detail about counts of holomorphic curves.

\item
It often turns out that the resulting homology is independent of the choice of almost complex structure $J$, and sometimes even of the underlying symplectic or contact structure. It is possible to obtain \emph{smooth} manifold or \emph{knot} invariants. 

\item
\emph{Heegaard Floer homology theories} are a very powerful variant of the above ideas, introduced by Ozsv{\'a}th and Szab{\'o} \cite{OS04Prop, OS04Closed, OS06}. From a $3$-manifold $M$, we take its \emph{Heegaard decomposition} consisting of a surface $\Sigma$ with curves $\alpha_1, \ldots, \alpha_g$ bounding discs on one side and $\beta_1, \ldots, \beta_g$ on the other. We then (as in \cite{Lip}) consider holomorphic curves in the symplectic manifold $\Sigma \times I \times \R$ with asymptotics prescribed by the $\alpha_i$ and $\beta_j$. 

It turns out that one can set up a chain complex so that the resulting homology is independent of any symplectic or almost complex structure chosen, giving a \emph{smooth manifold invariant} of $M$. Different versions, such as the \emph{hat}, \emph{infinity}, \emph{minus}, \emph{knot}, \emph{bordered} \cite{LOT08} and \emph{sutured} \cite{Ju06} Heegaard Floer homologies vary between the types of manifolds they apply to, the complexity of the theory, and the ease of obtaining information.

One indication of the power of Heegaard Floer homology is that it detects the genus of a knot \cite{OS04Knot, OS04Genus}; this can be computed combinatorially \cite{MOSa}.
\end{itemize}

We shall say no more about holomorphic curves until the final section, when we briefly return to sutured Floer homology to see how it is related to the elementary structures forming the main subject of this note.

\subsection{Contact geometry}

Contact geometry is the odd-dimensional sibling of symplectic geometry. Its roots can be traced back to Lie's work on systems of differential equations, and arguably back to Christiaan Huygens \cite{Geiges01, Geiges05}.

One way to define a symplectic structure is a $2$-form $\omega$ satisfying
\[
d\omega = 0 \quad \text{and} \quad \omega^n \text{ is nowhere zero, i.e. a volume form.}
\]

Analogously, a \emph{contact form} $\alpha$ on a $(2n+1)$-dimensional manifold $M$ is a $1$-form satisfying
\[
\alpha \wedge (d\alpha)^n \text{ is nowhere zero, i.e. a volume form.}
\]
The kernel of $\alpha$ is a hyperplane field $\xi$ on $M$ and it is this plane field that is called a \emph{contact structure}. 

Contact manifolds arise in considering wavefronts in optics. They also arise naturally as submanifolds of symplectic manifolds. In fact, so many geometric problems can be interpreted in terms of contact geometry, that Arnold once famously said that ``contact geometry is all geometry''.\cite{Arnold_Symplectic_Contact}

The condition $\alpha \wedge (d\alpha)^n \neq 0$ has a geometric interpretation from Frobenius' theorem in differential geometry. It is that $\xi$ is \emph{totally non-integrable}. 

We are interested in $3$-dimensional contact manifolds. Non-integrability then means that there is no 2-dimensional surface immersed in $M$ which is tangent to $\xi$. There are 1-dimensional curves tangent to $\xi$, but not 2-dimensional surfaces, not even locally.

The simplest example of a $3$-dimensional contact manifold is $\R^3$ with contact form
\[
\alpha = dz - y dx.
\]
The corresponding contact structure $\xi = \ker \alpha$ is shown below.

\begin{center}
\includegraphics[scale=1]{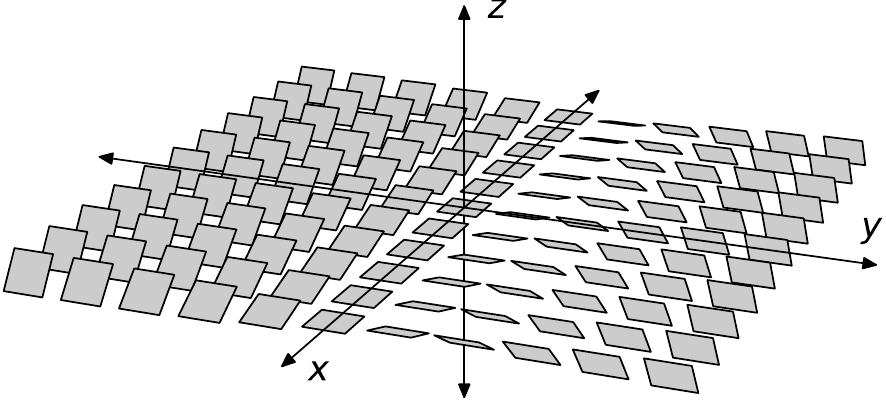}
\end{center}

Any contact manifold with a contact form $(M, \alpha)$ can be used to obtain a symplectic manifold, called its \emph{symplectization}:
\[
\left( M \times \R, \; d \left( e^t \alpha \right) \right),
\]
where $t$ is the coordinate on $\R$. It is then possible to consider holomorphic curves in this symplectization, with asymptotics prescribed by contact geometry. Using the analytic techniques mentioned above, we can obtain holomogy theories such as \emph{contact homology} \cite{Bo} and \emph{embedded contact homology} \cite{Hutchings02}.

Having said a little about contact and symplectic topology, I now propose to drop the subject entirely and talk about some seemingly unrelated algebra and combinatorics. We will later see how this becomes, in a certain sense, a combinatorial version of contact topology, and how it is related to holomorphic invariants.

\section{Quantum pawn dynamics}

The theory of quantum pawn dynamics, or QPD, is a strange theory of pawns on chessboards.

\subsection{Pawns on 1-dimensional chessboards}

This theory has no space, no time, and no proper chessboards. It does have pawns though. The pawns move along a finite 1-dimensional chessboard.

\begin{center}
\begin{tabularx}{0.4\textwidth}{|>{\centering}X|>{\centering}X|>{\centering}X|>{\centering}X|>{\centering}X|>{\centering}X|}
\hline
\WhitePawnOnWhite & & \WhitePawnOnWhite & \WhitePawnOnWhite & & \\
\hline
\end{tabularx}
\end{center}

The pawns move from left to right. As they are pawns, they can only move ahead one space at a time, and only into an unoccupied square. There is no capturing, no en passant, and no double first moves. Two pawns cannot occupy the same square. So from the situation above, the pawns could eventually arrive at

\begin{center}
\begin{tabularx}{0.4\textwidth}{|X|X|X|X|X|X|}
\hline
& \WhitePawnOnWhite & \WhitePawnOnWhite & & \WhitePawnOnWhite & \\
\hline
\end{tabularx}
\end{center}
in two moves, but could never get to
\begin{center}
\begin{tabularx}{0.4\textwidth}{|X|X|X|X|X|X|}
\hline
\WhitePawnOnWhite & \WhitePawnOnWhite & & & \WhitePawnOnWhite & \\
\hline
\end{tabularx}
\end{center}
as the middle pawn would have to move backwards.

There is nothing special about the number of pawns or the size of the chess-``board''. We can have a chess-``board'' of any length $n$, and we can have any number of pawns $n_p$ where $0 \leq n_p \leq n$. Any setup of the board is a state of the QPD universe.

\subsection{Quantum ``inner product''}

We now have pawns, and they move --- hence, dynamics. We now make the dynamics quantum by declaring an ``inner product'' $\langle \cdot | \cdot \rangle$. In quantum mechanics the inner product is supposed to give you a probability amplitude for getting from one state of the universe to another.

The QPD universe is more binary than that. Given one setup of pawns $w_0$, you can either get to another setup of pawns $w_1$ or you can't. And in fact, we do not need a number system more complicated than $\Z_2$. So we do not have probabilities so much as \emph{possibilities}. 

We declare that
\[
\langle w_0 | w_1 \rangle = \left\{ \begin{array}{ll}
	1 & \text{if it is possible for pawns to move from $w_0$ to $w_1$} \\
		&	\text{\quad (in some number of moves, possibly $0$);} \\
	0 & \text{if not.}
	\end{array}
	\right.
\]

So our examples above translate to
\[
\left\langle \quad
\begin{tabularx}{0.4\textwidth}{|>{\centering}X|>{\centering}X|>{\centering}X|>{\centering}X|>{\centering}X|>{\centering}X|}
\hline
\WhitePawnOnWhite & & \WhitePawnOnWhite & \WhitePawnOnWhite & & \\
\hline
\end{tabularx}
\quad | \quad
\begin{tabularx}{0.4\textwidth}{|X|X|X|X|X|X|}
\hline
& \WhitePawnOnWhite & \WhitePawnOnWhite & & \WhitePawnOnWhite & \\
\hline
\end{tabularx}
\quad \right\rangle
\quad = \quad 1
\]
and
\[
\left\langle \quad
\begin{tabularx}{0.4\textwidth}{|X|X|X|X|X|X|}
\hline
\WhitePawnOnWhite & & \WhitePawnOnWhite & \WhitePawnOnWhite & & \\
\hline
\end{tabularx}
\quad | \quad
\begin{tabularx}{0.4\textwidth}{|X|X|X|X|X|X|}
\hline
\WhitePawnOnWhite & \WhitePawnOnWhite & & & \WhitePawnOnWhite & \\
\hline
\end{tabularx}
\quad \right\rangle
\quad = \quad 0
\]

Moreover, this is \emph{quantum} pawn dynamics, so we can have \emph{superpositions} of states --- entangled chessboards! We consider that we can \emph{add} chessboards. But, as we said, we will not consider numbers more complicated than $1$; we take coefficients in $\Z_2$. The space of states is the $\Z_2$-vector space freely generated by chessboard setups. We declare that the ``inner product'' is bilinear, e.g.:
\[
\left\langle \quad
\begin{tabularx}{0.4\textwidth}{|X|X|X|X|X|X|}
\hline
\WhitePawnOnWhite & & \WhitePawnOnWhite & \WhitePawnOnWhite & & \\
\hline
\end{tabularx}
\quad | \quad
\begin{array}{c}
\begin{tabularx}{0.4\textwidth}{|X|X|X|X|X|X|}
\hline
& \WhitePawnOnWhite & \WhitePawnOnWhite & & \WhitePawnOnWhite & \\
\hline
\end{tabularx}\\
\quad + \quad \\
\begin{tabularx}{0.4\textwidth}{|X|X|X|X|X|X|}
\hline
\WhitePawnOnWhite & \WhitePawnOnWhite & & & \WhitePawnOnWhite & \\
\hline
\end{tabularx}
\end{array}
\quad \right\rangle
\quad = 1+0 = 1.
\]

Importantly, note that the ``inner product'' is \emph{not symmetric}. But it is asymmetric in a very interesting way, as we'll see. It is a ``booleanized'' \emph{partial order} on the setups of the chessboard. In fact, the configurations of a chessboard form a \emph{complete lattice} in a natural way.

\subsection{Dirac sea and anti-pawns}

Actually, there is some strong symmetry in QPD. We can think of chessboard with no pawns as a thriving sea of anti-pawns. We can think of any square not occupied by a pawn, as containing an anti-pawn. An anti-pawn is just an ``absence of pawn''. This is very similar to the idea of the ``Dirac sea'' of matter and anti-matter. We draw pawns as white and anti-pawns as black.

\[
\begin{tabularx}{0.4\textwidth}{|X|X|X|X|X|X|}
\hline
\WhitePawnOnWhite & & \WhitePawnOnWhite & \WhitePawnOnWhite & & \\
\hline
\end{tabularx}
=
\begin{tabularx}{0.4\textwidth}{|X|X|X|X|X|X|}
\hline
\WhitePawnOnWhite & \BlackPawnOnWhite & \WhitePawnOnWhite & \WhitePawnOnWhite & \BlackPawnOnWhite & \BlackPawnOnWhite \\
\hline 
\end{tabularx}
\]

Note when a pawn moves right, an anti-pawn moves left. So we can get from one setup of the chessboard to another, iff all pawns move right, iff all anti-pawns move left.

So white and black are not exactly like the opposing sides of chess, but they do move in opposite directions. Each is the absence of the other.

In each of the examples above, we imagine we have 6 pawn-particles: 3 pawns, and 3 anti-pawns. With $n$ as the number of squares on the chessboard / number of pawn-particles, and $n_p$ the number of pawns, we let $n_q = n - n_p$ be the number of anti-pawns.

\subsection{The initial creation and annihilation of pawns}

Being a quantum field theory of pawns, QPD will have to allow us to create and annihilate pawns and anti-pawns.

We'll define the \emph{initial pawn creation operator} $a_{p,0}^*$ to take a chessboard setup, and adjoin a new \emph{initial} square to the chessboard --- i.e. at the left hand side. This new square/particle has a pawn on it (is a pawn-particle).
\[
a_{p,0}^* \quad
\begin{tabularx}{0.4\textwidth}{|X|X|X|X|X|X|}
\hline
\WhitePawnOnWhite & \BlackPawnOnWhite & \WhitePawnOnWhite & \WhitePawnOnWhite & \BlackPawnOnWhite & \BlackPawnOnWhite \\
\hline 
\end{tabularx}
=
\begin{tabularx}{0.48\textwidth}{|X|X|X|X|X|X|X|}
\hline
\WhitePawnOnWhite & \WhitePawnOnWhite & \BlackPawnOnWhite & \WhitePawnOnWhite & \WhitePawnOnWhite & \BlackPawnOnWhite & \BlackPawnOnWhite \\
\hline 
\end{tabularx}
\]
We'll also define the \emph{initial pawn annihilation operator} $a_{p,0}$ to delete an initial pawn. That is, it will annihilate the leftmost square from the chessboard --- the chessboard shrinks and a particle disappears --- provided that square contains a pawn. What does the QPD universe do if the initial square does not contain a pawn, but an anti-pawn? Why, it returns an error of course: error 404 universe not found mod 2 is 0.

So, for example:
\begin{align*}
a_{p,0}
\quad
\begin{tabularx}{0.4\textwidth}{|X|X|X|X|X|X|}
\hline
\WhitePawnOnWhite & \BlackPawnOnWhite & \WhitePawnOnWhite & \WhitePawnOnWhite & \BlackPawnOnWhite & \BlackPawnOnWhite \\
\hline 
\end{tabularx}
&=
\begin{tabularx}{0.33\textwidth}{|X|X|X|X|X|}
\hline
\BlackPawnOnWhite & \WhitePawnOnWhite & \WhitePawnOnWhite & \BlackPawnOnWhite & \BlackPawnOnWhite \\
\hline 
\end{tabularx} \\
a_{p,0} \quad
\begin{tabularx}{0.33\textwidth}{|X|X|X|X|X|}
\hline
\BlackPawnOnWhite & \WhitePawnOnWhite & \WhitePawnOnWhite & \BlackPawnOnWhite & \BlackPawnOnWhite \\
\hline 
\end{tabularx}
&= 0.
\end{align*}

The vacuum $\emptyset$ is the state of the QPD universe with no chessboard. (This is different from $0$.) The universe is a lonely, empty place without a chessboard. But you can make a bang and start your universal chessboard by creating an initial pawn.
\[
a_{p,0}^* \quad \emptyset = 
\begin{tabularx}{0.07\textwidth}{|X|}
\hline
\WhitePawnOnWhite \\
\hline 
\end{tabularx}
\]

Everything we have said for pawns applies also to anti-pawns. So there is an \emph{initial anti-pawn creation operator} $a_{q,0}^\dagger$, which creates a new leftmost square with an anti-pawn. There is also an \emph{initial anti-pawn annihilation operator} $a_{q,0}$, which annihilates a leftmost anti-pawn, or else returns error $0$.

Using initial pawn and anti-pawn creation operators you can build any chessboard setup out of nothing.
\[
a_{p,0}^* a_{q,0}^\dagger a_{p,0}^* a_{p,0}^* a_{q,0}^\dagger a_{q,0}^\dagger \quad \emptyset \quad = \quad 
\begin{tabularx}{0.4\textwidth}{|X|X|X|X|X|X|}
\hline
\WhitePawnOnWhite & \BlackPawnOnWhite & \WhitePawnOnWhite & \WhitePawnOnWhite & \BlackPawnOnWhite & \BlackPawnOnWhite \\
\hline 
\end{tabularx}
\]

The $*$ and $\dagger$ on creation operators refer to \emph{adjoints}, which we discuss next. In the world of lattices and partial orders these are known as \emph{Galois connections}. (See \cite{Me10_Sutured_TQFT} for further details.)

\subsection{Adjoints}

It's standard notation in physics to write annihilation operators with an $a$ and creation operators as $a^*$ or $a^\dagger$. And these operators are supposed to be related: they are \emph{adjoint} with respect to the inner product.
\[
\langle a x | y \rangle = \langle x | a^* y \rangle,
\quad
\langle x | ay \rangle = \langle a^* x | y \rangle.
\]
In our case, the ``inner product'' isn't really an inner product because it's not symmetric. However it is bilinear and nondegenerate (check!). An operator $f$ on chessboards can have an adjoint --- but, in fact, \emph{two} adjoints. We'll write one of these as $f^*$ and one of them as $f^\dagger$, as shown below.
\[
\langle f x | y \rangle = \langle x | f^* y \rangle,
\quad
\langle x | fy \rangle = \langle f^\dagger x | y \rangle.
\]
Note this means that $f^{* \dagger} = f^{\dagger *} = f$. 

In general $f^{**} \neq f$. Some interesting things happen as we repeatedly take adjoints.

\subsection{Initial creation and annihilation are adjoint}

We can now see that our initial pawn creation $a_{p,0}^*$ is indeed the $*$-adjoint of the initial pawn annihilation $a_{p,0}$.
\[
\langle a_{p,0} x  | y \rangle = \langle x | a_{p,0}^* y \rangle
\]
Let's see why. Here $x$ and $y$ are chessboards. Note that if the leftmost square of $x$ is empty (contains an anti-pawn), so that $x = \begin{tabularx}{0.13\textwidth}{|X|X}
\hline
\BlackPawnOnWhite & $x_1 \cdots$ \\
\hline 
\end{tabularx}$, then
\[
\begin{array}{rcl}
\left\langle a_{p,0} x \quad | \quad y \right\rangle &=& \left\langle a_{p,0} \quad
\begin{tabularx}{0.13\textwidth}{|X|X}
\hline
\BlackPawnOnWhite & $x_1 \cdots$  \\
\hline 
\end{tabularx}
\quad | \quad y \right \rangle 
=
\langle 0 \quad | \quad y \rangle = 0 \\
\langle x \quad | \quad a_{p,0}^* y \rangle &=& \left\langle \quad
\begin{tabularx}{0.13\textwidth}{|X|X}
\hline
\BlackPawnOnWhite & $x_1 \cdots$ \\
\hline 
\end{tabularx}
\quad | \quad
\begin{tabularx}{0.13\textwidth}{|X|X}
\hline
\WhitePawnOnWhite & $y \cdots$ \\
\hline 
\end{tabularx}
\quad
\right\rangle
= 0.
\end{array}
\]
The first expression equals zero because we tried to annihilate an initial pawn where there was none and the universe returned an error. The second expression equals zero because a pawn would have to move backwards to get to the leftmost square.

On the other hand, if the leftmost square of $x$ contains a pawn, so that $x = \begin{tabularx}{0.13\textwidth}{|X|X}
\hline
\WhitePawnOnWhite & $x_1 \cdots$ \\
\hline 
\end{tabularx}$, then
\[
\begin{array}{rcl}
\left\langle a_{p,0} x \quad | \quad y \right\rangle &=& \left\langle a_{p,0} \quad
\begin{tabularx}{0.13\textwidth}{|X|X}
\hline
\WhitePawnOnWhite & $x_1 \cdots$  \\
\hline 
\end{tabularx}
\quad | \quad y \right\rangle 
=
\left\langle x_1 \quad | \quad y \right\rangle \\
\left\langle x \quad | \quad a_{p,0}^* y \right\rangle &=&
\left\langle \quad 
\begin{tabularx}{0.13\textwidth}{|X|X}
\hline
\WhitePawnOnWhite & $x_1 \cdots$ \\
\hline 
\end{tabularx} 
\quad | \quad
\begin{tabularx}{0.13\textwidth}{|X|X}
\hline
\WhitePawnOnWhite & $y \cdots$ \\
\hline 
\end{tabularx}
\quad
\right\rangle
=
\left\langle x_1 \quad | \quad y \right\rangle
\end{array}
\]
So we've shown that $a_{p,0}^*$ is indeed the $*$ adjoint of $a_{p,0}$.

To (try to) put it succinctly: $a_{p,0}^* y$ begins with a pawn, so $\langle x | a_{p,0}^* y \rangle$ is only nonzero if $x$ also begins with a pawn; in which case $a_{p,0} x$ removes that pawn so $\langle a_{p,0} x | y \rangle$ gives the same result.

It's not difficult to check by a similar argument that $a_{q,0}^\dagger$, the initial anti-pawn creation operator, is indeed the $\dagger$-adjoint of the initial anti-pawn annihilation operator $a_{q,0}$.

\subsection{Adjoints and adjoints and adjoints}

We have seen one adjoint $a_{p,0}^*$, and by definition $a_{p,0}^{* \dagger} = a_{p,0}$; but what happens if we take the $*$ adjoint twice? In other words, what operator $f$ satisfies
\[
\langle a_{p,0}^* x | y \rangle = \langle x | f y \rangle?
\]
Note $a_{p,0}^* x$ just puts a pawn at the left end of $x$, and we compare $a_{p,0}^* x$ to $y$. Clearly the newly-added first pawn can't move left. The real question to compute $\langle a_{p,0}^* x | y \rangle$ is what happens to all the other pawns. In fact, if you \emph{eliminate} the first pawn from the left in both $a_{p,0}^* x$ and $y$, you'll reduce to the problem of looking at the other pawns.

\[
\left\langle a_{p,0}^* \quad
\begin{tabularx}{0.33\textwidth}{|X|X|X|X|X|}
\hline
& \WhitePawnOnWhite & \WhitePawnOnWhite & & \\
\hline
\end{tabularx}
\quad | \quad 
\begin{tabularx}{0.4\textwidth}{|X|X|X|X|X|X|}
\hline
& \WhitePawnOnWhite & \WhitePawnOnWhite & & \WhitePawnOnWhite & \\
\hline
\end{tabularx}
\quad \right\rangle
\]
\[
= 
\left\langle \quad
\begin{tabularx}{0.4\textwidth}{|X|X|X|X|X|X|}
\hline
\textcolor{red}{\WhitePawnOnWhite} & & \WhitePawnOnWhite & \WhitePawnOnWhite & & \\
\hline
\end{tabularx}
\quad | \quad
\begin{tabularx}{0.4\textwidth}{|X|X|X|X|X|X|}
\hline
& \textcolor{red}{\WhitePawnOnWhite} & \WhitePawnOnWhite & & \WhitePawnOnWhite & \\
\hline
\end{tabularx}
\quad \right\rangle
\]
\[
= \left\langle \quad
\begin{tabularx}{0.33\textwidth}{|X|X|X|X|X|}
\hline
& \WhitePawnOnWhite & \WhitePawnOnWhite & & \\
\hline
\end{tabularx}
\quad | \quad
\begin{tabularx}{0.33\textwidth}{|X|X|X|X|X|}
\hline
& \WhitePawnOnWhite & & \WhitePawnOnWhite & \\
\hline
\end{tabularx}
\quad \right\rangle
\quad = \quad 1
\]
We conclude that $a_{p,0}^{* *}$ is the operator which \emph{annihilates the first pawn on the chessboard}, from left to right. We'll write this as $a_{p,1}$.

But why stop there? What is $a_{p,1}^*$? This is the operator $f$ which satisfies
\[
\langle a_{p,1} x | y \rangle = \langle x | fy \rangle.
\]
Given chessboards $x$ and $y$, you delete the first pawn from $x$, and compare $a_{p,1} x$ and $y$. You now want to come up with a chessboard $fy$, so that comparing the deleted chessboard $a_{p,1} x$ to $y$ always gives the same result as comparing $x$ to $fy$. Well, you certainly don't want to disturb the relative positions of the pawns. You just want to quietly slip a first pawn into $y$ in such a way that the first pawn of $x$ can easily move there. And the place to slip in that pawn is to \emph{double the first pawn in $y$}.
\[
\left\langle \quad a_{p,1} \quad
\begin{tabularx}{0.4\textwidth}{|X|X|X|X|X|X|X|}
\hline
& \WhitePawnOnWhite & & \WhitePawnOnWhite & & & \WhitePawnOnWhite \\
\hline
\end{tabularx}
\quad | \quad
\begin{tabularx}{0.33\textwidth}{|X|X|X|X|X|X|}
\hline
& & & \WhitePawnOnWhite & & \WhitePawnOnWhite \\
\hline
\end{tabularx}
\quad \right\rangle
\]
\[
= \left\langle \quad
\begin{tabularx}{0.33\textwidth}{|X|X|X|X|X|X|}
\hline
& & \WhitePawnOnWhite & & & \WhitePawnOnWhite \\
\hline
\end{tabularx}
\quad | \quad
\begin{tabularx}{0.33\textwidth}{|X|X|X|X|X|X|}
\hline
& & & \WhitePawnOnWhite & & \WhitePawnOnWhite \\
\hline
\end{tabularx}
\quad \right\rangle
\]
\[
= \left\langle \quad
\begin{tabularx}{0.4\textwidth}{|X|X|X|X|X|X|X|}
\hline
& \textcolor{green}{\WhitePawnOnWhite}  & & \WhitePawnOnWhite & & & \WhitePawnOnWhite \\
\hline
\end{tabularx}
\quad | \quad
\begin{tabularx}{0.4\textwidth}{|X|X|X|X|X|X|X|}
\hline
& & & \textcolor{green}{\WhitePawnOnWhite} & \WhitePawnOnWhite & & \WhitePawnOnWhite \\
\hline
\end{tabularx}
\quad \right\rangle
\]

\subsection{Round-up of adjoints}

We can continue in this fashion, computing more and more adjoints. Given what's above, it might not be too surprising to discover that the iterated adjoints of $a_{p,0}$ can be expressed as
\[
\begin{array}{cccccccccccccccccc}
a_{p,0} & \rightarrow & a_{p,0}^* & \rightarrow & a_{p,1} & \rightarrow & a_{p,1}^* & \rightarrow & a_{p,2} & \rightarrow & \cdots & a_{p,n_p} & a_{p,n_p}^* & \rightarrow & a_{p,\Omega} & \rightarrow & a_{p,\Omega}^*
\end{array}
\]
where the arrows represented the operation of taking the $*$-adjoint. The operators $a_{p,i}$ for $1 \leq i \leq n_p$ delete the $i$'th pawn. The operators $a_{p,i}^*$ double the $i$'th pawn. And at the end we obtain \emph{final} creation and annihilation operators, which we've denoted $a_{p,\Omega}^*$ and $a_{p,\Omega}$. These are to the right end of a chessboard what initial creation and annihilation operators are to the left end.

We can do just the same for the anti-pawn creation and annihilation operators. Drawing a similar diagram, with arrows representing the $*$ adjoint (the inverse of the $\dagger$ adjoint), we obtain something similar (although in a different direction).
\[
\begin{array}{cccccccccccccccccc}
a_{q,\Omega}^\dagger & \rightarrow & a_{q,\Omega} & \rightarrow & a_{q,n_q}^\dagger & \rightarrow & a_{q,n_q} & \rightarrow & \cdots & a_{q,2} & \rightarrow & a_{q,1}^\dagger & \rightarrow & a_{q,1} & \rightarrow a_{q,0}^\dagger & \rightarrow & a_{q,0}
\end{array}
\]
It follows that
\[
a_{p,0}^{*^{2n_p + 2}} = a_{p,\Omega} \quad \text{and}
\quad
a_{q,0}^{\dagger^{2n_q + 2}} = a_{q,\Omega}.
\]
But there's no reason to stop there; you can just keep taking adjoints. The description of these adjoints in terms of pawns, however, becomes more complicated.

It turns out that these adjoints are \emph{periodic}.
\begin{thm}
On a chessboard with $n$ squares, 
\[
a_{p,0}^{*^{2n+2}} = a_{p,0}, \quad a_{q,0}^{*^{2n+2}} = a_{q,0}.
\]
\end{thm}

We might say that ``$*^{2n+2} = 1$''. Our proof of this involves a little more geometry, by considering the combinatorics of chords on discs, as discussed in the next section.

The various creation and annihilation operators in different positions actually satisfy the relations of a \emph{simplicial set}: see \cite{Me10_Sutured_TQFT} for details.

\subsection{Chessboards and words}

By now ``chessboard notation'' is becoming somewhat unwieldy. In fact, we can denote any chessboard by a sequence of $p$'s and $q$'s, where a $p$ is a pawn and $q$ an anti-pawn. 
\[
qpqq
\quad
\leftrightarrow
\quad
\begin{tabularx}{0.3\textwidth}{|X|X|X|X|}
\hline
\BlackPawnOnWhite & \WhitePawnOnWhite & \BlackPawnOnWhite & \BlackPawnOnWhite \\
\hline 
\end{tabularx}
\]

\[
\langle pqppqq | qppqpq \rangle =
\]
\[
\left\langle \quad
\begin{tabularx}{0.4\textwidth}{|>{\centering}X|>{\centering}X|>{\centering}X|>{\centering}X|>{\centering}X|>{\centering}X|}
\hline
\WhitePawnOnWhite & & \WhitePawnOnWhite & \WhitePawnOnWhite & & \\
\hline
\end{tabularx}
\quad | \quad
\begin{tabularx}{0.4\textwidth}{|X|X|X|X|X|X|}
\hline
& \WhitePawnOnWhite & \WhitePawnOnWhite & & \WhitePawnOnWhite & \\
\hline
\end{tabularx}
\quad \right\rangle
\quad = \quad 1.
\]

\section{Chords on discs}

\subsection{Curves on a disc}

We now consider curves on discs and cylinders.

We consider a disc with $2n+2$ points marked on the boundary, which we number clockwise by integers modulo $2n+2$. The point numbered $0$ is considered a basepoint. We consider \emph{chord diagrams}, which are collections of non-intersecting curves joining the points up in pairs. We only consider them up to isotopy relative to the boundary.

Note that each chord in a chord diagram must connect points with opposite parity. We can shade the complementary regions of a chord diagram, so that each boundary interval $(2i, 2i+1)$ is shaded black, and each interval $(2i-1,2i)$ is shaded white.

\[
\begin{tikzpicture}[
scale=0.9, 
suture/.style={thick, draw=red},
boundary/.style={ultra thick},
vertex/.style={draw=red, fill=red}]

\coordinate [label = above:{$0$}] (12) at (90:2);
\coordinate [label = above right:{$1$}] (1) at (60:2);
\coordinate [label = above right:{$2$}] (2) at (30:2);
\coordinate [label = right:{$3$}] (3) at (0:2);
\coordinate [label = below right:{$4$}] (4) at (-30:2);
\coordinate [label = below right:{$5$}] (5) at (-60:2);
\coordinate [label = below:{$6$}] (6) at (-90:2);
\coordinate [label = below left:{$-5$}] (7) at (-120:2);
\coordinate [label = below left:{$-4$}] (8) at (-150:2);
\coordinate [label = left:{$-3$}] (9) at (-180:2);
\coordinate [label = above left:{$-2$}] (10) at (-210:2);
\coordinate [label = above left:{$-1$}] (11) at (-240:2);

\filldraw[fill=black!10!white, draw=none] (11) to [bend left=90] (10) arc (150:120:2) -- cycle;
\filldraw[fill=black!10!white, draw=none] (12) arc (90:60:2) to [bend right=90] (2) arc (30:0:2) to [bend right=15] (8) arc (210:180:2) to [bend right=45] (12);
\filldraw[fill=black!10!white, draw=none] (4) to [bend right=90] (5) arc (-60:-30:2);
\filldraw[fill=black!10!white, draw=none] (6) to [bend right=90] (7) arc (-120:-90:2);
\draw[suture] (9) to [bend right=45] (12);
\draw[suture] (11) to [bend left=90] (10);
\draw[suture] (1) to [bend right=90] (2);
\draw[suture] (3) to [bend right=15] (8);
\draw[suture] (4) to [bend right=90] (5);
\draw[suture] (6) to [bend right=90] (7);

\draw [boundary] (0,0) circle (2 cm);
\foreach \point in {1, 2, 3, 4, 5, 6, 7, 8, 9, 10, 11, 12}
\fill [vertex] (\point) circle (2pt);

\end{tikzpicture}
\]

Now we will declare that the effect of the \emph{creation operator} $a_{p,0}^*$ on a chord diagram is to insert a new chord between the points currently numbered $0$ and $1$; the new points are numbered $-1$ and $0$ and so the points on the left need to be renumbered. The existing chords remain in place but are pushed around the disc as shown.
\[
\begin{tikzpicture}[
scale=0.9, 
suture/.style={thick, draw=red},
boundary/.style={ultra thick},
vertex/.style={draw=red, fill=red}]

\draw (-4,0) node {$a_{p,0}^*$};

\draw [->] (72:2.5) -- (66:2.3);
\draw [->] (72:2.5) -- (78:2.3);

\coordinate [label = above:{$0$}] (0) at (90:2);
\coordinate [label = above right:{$1$}] (1) at (54:2);
\coordinate [label = above right:{$2$}] (2) at (18:2);
\coordinate [label = right:{$3$}] (3) at (-18:2);
\coordinate [label = below right:{$4$}] (4) at (-54:2);
\coordinate [label = below right:{$5$}] (5) at (-90:2);
\coordinate [label = below:{$-4$}] (6) at (-126:2);
\coordinate [label = below left:{$-3$}] (7) at (-162:2);
\coordinate [label = below left:{$-2$}] (8) at (162:2);
\coordinate [label = left:{$-1$}] (9) at (126:2);

\filldraw[fill=black!10!white, draw=none] (0) to [bend right=90] (1) arc (54:90:2);
\filldraw[fill=black!10!white, draw=none] (9) to [bend right=30] (2) arc (18:-18:2) -- (8) arc (162:126:2);
\filldraw[fill=black!10!white, draw=none] (4) to [bend right=90] (5) arc (-90:-54:2);
\filldraw[fill=black!10!white, draw=none] (6) to [bend right=90] (7) arc (-162:-126:2);
\draw[suture] (9) to [bend right=30] (2);
\draw[suture] (0) to [bend right=90] (1);
\draw[suture] (3) -- (8);
\draw[suture] (4) to [bend right=90] (5);
\draw[suture] (6) to [bend right=90] (7);

\draw [boundary] (0,0) circle (2 cm);
\foreach \point in {0, 1, 2, 3, 4, 5, 6, 7, 8, 9}
\fill [vertex] (\point) circle (2pt);

\draw (4,0) node {$=$};

\coordinate [label = above:{$0$}] (12b) at ($ (90:2) + (8,0) $);
\coordinate [label = above right:{$1$}] (1b) at ($ (60:2) + (8,0) $);
\coordinate [label = above right:{$2$}] (2b) at ($ (30:2) + (8,0) $);
\coordinate [label = right:{$3$}] (3b) at ($ (0:2) + (8,0) $);
\coordinate [label = below right:{$4$}] (4b) at ($ (-30:2) + (8,0) $);
\coordinate [label = below right:{$5$}] (5b) at ($ (-60:2) + (8,0) $);
\coordinate [label = below:{$6$}] (6b) at ($ (-90:2) + (8,0) $);
\coordinate [label = below left:{$-5$}] (7b) at ($ (-120:2) + (8,0) $);
\coordinate [label = below left:{$-4$}] (8b) at ($ (-150:2) + (8,0) $);
\coordinate [label = left:{$-3$}] (9b) at ($ (-180:2) + (8,0) $);
\coordinate [label = above left:{$-2$}] (10b) at ($ (-210:2) + (8,0) $);
\coordinate [label = above left:{$-1$}] (11b) at ($ (-240:2) + (8,0) $);

\filldraw[fill=black!10!white, draw=none] (12b) arc (90:60:2) to [bend left=45] (10b) arc (150:120:2) to [bend right=90] (12b);
\filldraw[fill=black!10!white, draw=none] (2b) arc (30:0:2) to [bend right=15] (8b) arc (210:180:2) to [bend right=15] (2b);
\filldraw[fill=black!10!white, draw=none] (4b) to [bend right=90] (5b) arc (-60:-30:2);
\filldraw[fill=black!10!white, draw=none] (6b) to [bend right=90] (7b) arc (-120:-90:2);
\draw[suture] (9b) to [bend right=15] (2b);
\draw[suture] (1b) to [bend left=45] (10b);
\draw[suture] (11b) to [bend right=90] (12b);
\draw[suture] (3b) to [bend right=15] (8b);
\draw[suture] (4b) to [bend right=90] (5b);
\draw[suture] (6b) to [bend right=90] (7b);

\draw [boundary] (8,0) circle (2 cm);
\foreach \point in {1b, 2b, 3b, 4b, 5b, 6b, 7b, 8b, 9b, 10b, 11b, 12b}
\fill [vertex] (\point) circle (2pt);

\end{tikzpicture}
\]

More generally, we will declare that the creation operator $a_{p,i}^*$  inserts a new chord between the points currently labelled $-2i$ and $-2i+1$; the new points are labelled $-2i-1$ and $-2i$, and points are renumbered accordingly,

For example
\[
\begin{tikzpicture}[
scale=0.9, 
suture/.style={thick, draw=red},
boundary/.style={ultra thick},
vertex/.style={draw=red, fill=red}]

\draw (-4,0) node {$a_{p,1}^*$};

\draw [->] (144:2.5) -- (138:2.3);
\draw [->] (144:2.5) -- (150:2.3);

\coordinate [label = above:{$0$}] (0) at (90:2);
\coordinate [label = above right:{$1$}] (1) at (54:2);
\coordinate [label = above right:{$2$}] (2) at (18:2);
\coordinate [label = right:{$3$}] (3) at (-18:2);
\coordinate [label = below right:{$4$}] (4) at (-54:2);
\coordinate [label = below right:{$5$}] (5) at (-90:2);
\coordinate [label = below:{$-4$}] (6) at (-126:2);
\coordinate [label = below left:{$-3$}] (7) at (-162:2);
\coordinate [label = below left:{$-2$}] (8) at (162:2);
\coordinate [label = left:{$-1$}] (9) at (126:2);

\filldraw[fill=black!10!white, draw=none] (0) to [bend right=90] (1) arc (54:90:2);
\filldraw[fill=black!10!white, draw=none] (9) to [bend right=30] (2) arc (18:-18:2) -- (8) arc (162:126:2);
\filldraw[fill=black!10!white, draw=none] (4) to [bend right=90] (5) arc (-90:-54:2);
\filldraw[fill=black!10!white, draw=none] (6) to [bend right=90] (7) arc (-162:-126:2);
\draw[suture] (9) to [bend right=30] (2);
\draw[suture] (0) to [bend right=90] (1);
\draw[suture] (3) -- (8);
\draw[suture] (4) to [bend right=90] (5);
\draw[suture] (6) to [bend right=90] (7);

\draw [boundary] (0,0) circle (2 cm);
\foreach \point in {0, 1, 2, 3, 4, 5, 6, 7, 8, 9}
\fill [vertex] (\point) circle (2pt);

\draw (4,0) node {$=$};

\coordinate [label = above:{$0$}] (12b) at ($ (90:2) + (8,0) $);
\coordinate [label = above right:{$1$}] (1b) at ($ (60:2) + (8,0) $);
\coordinate [label = above right:{$2$}] (2b) at ($ (30:2) + (8,0) $);
\coordinate [label = right:{$3$}] (3b) at ($ (0:2) + (8,0) $);
\coordinate [label = below right:{$4$}] (4b) at ($ (-30:2) + (8,0) $);
\coordinate [label = below right:{$5$}] (5b) at ($ (-60:2) + (8,0) $);
\coordinate [label = below:{$6$}] (6b) at ($ (-90:2) + (8,0) $);
\coordinate [label = below left:{$-5$}] (7b) at ($ (-120:2) + (8,0) $);
\coordinate [label = below left:{$-4$}] (8b) at ($ (-150:2) + (8,0) $);
\coordinate [label = left:{$-3$}] (9b) at ($ (-180:2) + (8,0) $);
\coordinate [label = above left:{$-2$}] (10b) at ($ (-210:2) + (8,0) $);
\coordinate [label = above left:{$-1$}] (11b) at ($ (-240:2) + (8,0) $);

\filldraw[fill=black!10!white, draw=none] (12b) arc (90:60:2) to [bend left=90] (12b);
\filldraw[fill=black!10!white, draw=none] (2b) arc (30:0:2) to [bend right=15] (8b) arc (210:180:2) to [bend right=90] (10b) arc (150:120:2) to [bend right=45] (2b);
\filldraw[fill=black!10!white, draw=none] (4b) to [bend right=90] (5b) arc (-60:-30:2);
\filldraw[fill=black!10!white, draw=none] (6b) to [bend right=90] (7b) arc (-120:-90:2);
\draw[suture] (9b) to [bend right=90] (10b);
\draw[suture] (1b) to [bend left=90] (12b);
\draw[suture] (11b) to [bend right=45] (2b);
\draw[suture] (3b) to [bend right=15] (8b);
\draw[suture] (4b) to [bend right=90] (5b);
\draw[suture] (6b) to [bend right=90] (7b);

\draw [boundary] (8,0) circle (2 cm);
\foreach \point in {1b, 2b, 3b, 4b, 5b, 6b, 7b, 8b, 9b, 10b, 11b, 12b}
\fill [vertex] (\point) circle (2pt);

\end{tikzpicture}
\]

We will declare the effect of the \emph{annihilation operator} $a_{p,0}$ on a chord diagram to ``close off'' the points $0$ and $1$ by joining them with a chord, and pushing this into the disc. Points numbered $3, 4, \ldots$ retain their numbering; the points labelled $-1, -2, \ldots$ have their label increased by $2$. We obtain a diagram with two fewer points on the boundary. It might have a closed curve as well as chords.

For example,
\[
\begin{tikzpicture}[
scale=0.9, 
suture/.style={thick, draw=red},
boundary/.style={ultra thick},
vertex/.style={draw=red, fill=red}]

\draw (-4,0) node {$a_{p,0}$};

\coordinate [label = above:{$0$}] (12) at (90:2);
\coordinate [label = above right:{$1$}] (1) at (60:2);
\coordinate [label = above right:{$2$}] (2) at (30:2);
\coordinate [label = right:{$3$}] (3) at (0:2);
\coordinate [label = below right:{$4$}] (4) at (-30:2);
\coordinate [label = below right:{$5$}] (5) at (-60:2);
\coordinate [label = below:{$6$}] (6) at (-90:2);
\coordinate [label = below left:{$-5$}] (7) at (-120:2);
\coordinate [label = below left:{$-4$}] (8) at (-150:2);
\coordinate [label = left:{$-3$}] (9) at (-180:2);
\coordinate [label = above left:{$-2$}] (10) at (-210:2);
\coordinate [label = above left:{$-1$}] (11) at (-240:2);

\filldraw[fill=black!10!white, draw=none] (12) arc (90:60:2) to [bend left=90] (12);
\filldraw[fill=black!10!white, draw=none] (2) arc (30:0:2) to [bend right=15] (8) arc (210:180:2) to [bend right=90] (10) arc (150:120:2) to [bend right=45] (2);
\filldraw[fill=black!10!white, draw=none] (4) to [bend right=90] (5) arc (-60:-30:2);
\filldraw[fill=black!10!white, draw=none] (6) to [bend right=90] (7) arc (-120:-90:2);
\draw[suture] (9) to [bend right=90] (10);
\draw[suture] (1) to [bend left=90] (12);
\draw[suture] (11) to [bend right=45] (2);
\draw[suture] (3) to [bend right=15] (8);
\draw[suture] (4) to [bend right=90] (5);
\draw[suture] (6) to [bend right=90] (7);
\draw[red, ultra thick, dotted] (12) to [bend left=90] (1);

\draw [boundary] (0,0) circle (2 cm);
\foreach \point in {1, 2, 3, 4, 5, 6, 7, 8, 9, 10, 11, 12}
\fill [vertex] (\point) circle (2pt);

\draw (4,0) node {$=$};

\coordinate [label = above:{$0$}] (0b) at ($ (90:2) + (8,0) $);
\coordinate [label = above right:{$1$}] (1b) at ($ (54:2) + (8,0) $);
\coordinate [label = above right:{$2$}] (2b) at ($ (18:2) + (8,0) $);
\coordinate [label = right:{$3$}] (3b) at ($ (-18:2) + (8,0) $);
\coordinate [label = below right:{$4$}] (4b) at ($ (-54:2) + (8,0) $);
\coordinate [label = below right:{$5$}] (5b) at ($ (-90:2) + (8,0) $);
\coordinate [label = below:{$-4$}] (6b) at ($ (-126:2) + (8,0) $);
\coordinate [label = below left:{$-3$}] (7b) at ($ (-162:2) + (8,0) $);
\coordinate [label = below left:{$-2$}] (8b) at ($ (162:2) + (8,0) $);
\coordinate [label = left:{$-1$}] (9b) at ($ (126:2) + (8,0) $);

\filldraw[fill=black!10!white, draw=none] (2b) arc (18:-18:2) -- (8b) arc (162:126:2) to [bend right=90] (0b) arc (90:54:2) to [bend right=90] (2b);
\filldraw[fill=black!10!white, draw=none] (4b) to [bend right=90] (5b) arc (-90:-54:2);
\filldraw[fill=black!10!white, draw=none] (6b) to [bend right=90] (7b) arc (-162:-126:2);
\filldraw[fill=black!10!white, draw=red] ($ (36:1.8) + (8,0) $) circle (0.1);
\draw[suture] (0b) to [bend left=90] (9b);
\draw[suture] (1b) to [bend right=90] (2b);
\draw[suture] (3b) -- (8b);
\draw[suture] (4b) to [bend right=90] (5b);
\draw[suture] (6b) to [bend right=90] (7b);

\draw [boundary] (8,0) circle (2 cm);

\foreach \point in {0b, 1b, 2b, 3b, 4b, 5b, 6b, 7b, 8b, 9b}
\fill [vertex] (\point) circle (2pt);

\end{tikzpicture}
\]

More generally, we declare that the annihilation operator $a_{p,i}$ closes off the points $-2i$ and $-2i+1$, and relabels points: the points $-2i+2, -2i+3, \ldots$ retain their numbering but the points $-2i-1, -2i-2, \ldots$ have their label increased by $2$. Again we might see a closed curve as a result.

When $i=0$, both $a_{p,i}^*$ and $a_{p,i}$ perform operations near the basepoint labelled $0$. When $i$ increases by $1$, those operations occur $2$ spots anticlockwise.

We will also declare creation operators $a_{q,i}^\dagger$ and annihilation operators $a_{q,i}$. These have a similar effect but on the other side of the diagram. We declare $a_{q,i}^\dagger$ inserts a new chord between the points labelled $2i-1$ and $2i$; and we declare that $a_{q,i}$ closes off the points $2i-1$ and $2i$. So $a_{q,0}$ and $a_{q,0}^\dagger$ perform operations near the basepoint, and when $i$ increases by $1$, those operations occur $2$ spots clockwise.

\[
\begin{tabular}{|c|c|}
\hline
\begin{tikzpicture}[
scale=0.9, 
suture/.style={thick, draw=red},
boundary/.style={ultra thick},
vertex/.style={draw=red, fill=red}]

\draw (-4,0) node {$a_{p,i}^*$};

\coordinate [label = left:{$-2i+1$}] (a) at (144:2);
\coordinate [label = left:{$-2i$}] (b) at (216:2);
\filldraw[fill=black!10!white, draw=none] (a) -- (-1,0.8) -- (-1,-0.8) -- (b) arc (216:144:2);
\draw[suture] (a) -- (-1,0.8);
\draw[suture] (b) -- (-1,-0.8);
\draw [boundary](-1,1.732) arc (120:240:2);

\draw (0,0) node {$=$};

\coordinate [label = left:{$-2i+1$}] (c) at ($ (144:2) + (4,0) $);
\coordinate [label = left:{$-2i$}] (d) at ($ (168:2) + (4,0) $);
\coordinate [label = left:{$-2i-1$}] (e) at ($ (192:2) + (4,0) $);
\coordinate [label = left:{$-2i-2$}] (f) at ($ (216:2) + (4,0) $);
\filldraw[fill=black!10!white, draw=none] (c) -- (3,0.8) -- (3,-0.8) -- (f) arc (216:192:2) to [bend right=90] (d) arc (168:144:2);
\draw [suture] (c) -- (3,0.8);
\draw [suture] (d) to [bend left=90] (e);
\draw [suture] (f) -- (3,-0.8);
\draw [boundary](3,1.732) arc (120:240:2);

\foreach \point in {a, b, c, d, e, f}
\fill [vertex] (\point) circle (2pt);
\end{tikzpicture}
&
\begin{tikzpicture}[
scale=0.9, 
suture/.style={thick, draw=red},
boundary/.style={ultra thick},
vertex/.style={draw=red, fill=red}]

\draw (-4,0) node {$a_{p,i}$};

\coordinate [label = left:{$-2i+2$}] (a) at (144:2);
\coordinate [label = left:{$-2i+1$}] (b) at (168:2);
\coordinate [label = left:{$-2i$}] (c) at (192:2);
\coordinate [label = left:{$-2i-1$}] (d) at (216:2);
\filldraw[fill=black!10!white, draw=none] (a) -- (-1,1) -- (-1,1.732) arc (120:144:2);
\filldraw[fill=black!10!white, draw=none] (b) -- (-1,0.4) -- (-1,-0.4) -- (c) arc (192:168:2);
\filldraw[fill=black!10!white, draw=none] (d) -- (-1,-1) -- (-1,-1.732) arc (240:216:2);
\draw[suture] (a) -- (-1,1);
\draw[suture] (b) -- (-1,0.4);
\draw[suture] (c) -- (-1,-0.4);
\draw[suture] (d) -- (-1,-1);
\draw [boundary](-1,1.732) arc (120:240:2);

\draw (0,0) node {$=$};

\coordinate [label = left:{$-2i+2$}] (e) at ($ (144:2) + (4,0) $);
\coordinate [label = left:{$-2i+1$}] (f) at ($ (216:2) + (4,0) $);
\filldraw[fill=black!10!white, draw=none] (e) -- (3,1) -- (3,1.732) arc (120:144:2);
\filldraw[fill=black!10!white, draw=none] (f) -- (3,-1) -- (3,-1.732) arc (240:216:2);
\filldraw[fill=black!10!white, draw=none] (3,0.4) .. controls (2.5,0.4) and (2.4,0.5) .. (2.4,0) .. controls (2.4,-0.5) and (2.5,-0.4) .. (3,-0.4);
\draw [suture] (e) -- (3,1);
\draw [suture] (f) -- (3,-1);
\draw [suture] (3,0.4) .. controls (2.5,0.4) and (2.4,0.5) .. (2.4,0) .. controls (2.4,-0.5) and (2.5,-0.4) .. (3,-0.4);
\draw [boundary](3,1.732) arc (120:240:2);

\foreach \point in {a, b, c, d, e, f}
\fill [vertex] (\point) circle (2pt);
\end{tikzpicture}
\\
\hline
\begin{tikzpicture}[
scale=0.9, 
suture/.style={thick, draw=red},
boundary/.style={ultra thick},
vertex/.style={draw=red, fill=red}]

\draw (-3,0) node {$a_{q,i}^\dagger$};

\coordinate [label = right:{$2i-1$}] (a) at ($ (36:2) + (-3,0) $);
\coordinate [label = right:{$2i$}] (b) at ($ (-36:2) + (-3,0) $);
\filldraw[fill=black!10!white, draw=none] (a) -- (-2,0.8) -- (-2,1.732) arc (60:36:2);
\filldraw[fill=black!10!white, draw=none] (b) -- (-2,-0.8) -- (-2,-1.732) arc (-60:-36:2);
\draw[suture] (a) -- (-2,0.8);
\draw[suture] (b) -- (-2,-0.8);
\draw [boundary](-2,1.732) arc (60:-60:2);

\draw (1,0) node {$=$};

\coordinate [label = right:{$2i-1$}] (c) at ($ (36:2) + (1,0) $);
\coordinate [label = right:{$2i$}] (d) at ($ (12:2) + (1,0) $);
\coordinate [label = right:{$2i+1$}] (e) at ($ (-12:2) + (1,0) $);
\coordinate [label = right:{$2i+2$}] (f) at ($ (-36:2) + (1,0) $);
\filldraw[fill=black!10!white, draw=none] (c) -- (2,0.8) -- (2,1.732);
\filldraw[fill=black!10!white, draw=none] (f) -- (2,-0.8) -- (2,-1.732);
\filldraw[fill=black!10!white, draw=none] (d) to [bend right=90] (e); 
\draw [suture] (c) -- (2,0.8);
\draw [suture] (d) to [bend right=90] (e);
\draw [suture] (f) -- (2,-0.8);
\draw [boundary](2,1.732) arc (60:-60:2);

\foreach \point in {a, b, c, d, e, f}
\fill [vertex] (\point) circle (2pt);
\end{tikzpicture}
&
\begin{tikzpicture}[
scale=0.9, 
suture/.style={thick, draw=red},
boundary/.style={ultra thick},
vertex/.style={draw=red, fill=red}]

\draw (-3,0) node {$a_{q,i}$};

\coordinate [label = right:{$2i-2$}] (a) at ($ (36:2) + (-3,0) $);
\coordinate [label = right:{$2i-1$}] (b) at ($ (12:2) + (-3,0) $);
\coordinate [label = right:{$2i$}] (c) at ($ (-12:2) + (-3,0) $);
\coordinate [label = right:{$2i+1$}] (d) at ($ (-36:2) + (-3,0) $);
\filldraw[fill=black!10!white, draw=none] (-2,1) -- (a) arc (36:12:2) -- (-2,0.4) -- cycle;
\filldraw[fill=black!10!white, draw=none] (-2,-0.4) -- (c) arc (-12:-36:2) -- (-2,-1) -- cycle;
\draw[suture] (a) -- (-2,1);
\draw[suture] (b) -- (-2,0.4);
\draw[suture] (c) -- (-2,-0.4);
\draw[suture] (d) -- (-2,-1);
\draw [boundary](-2,1.732) arc (60:-60:2);

\draw (1,0) node {$=$};

\coordinate [label = right:{$2i-2$}] (e) at ($ (36:2) + (1,0) $);
\coordinate [label = right:{$2i-1$}] (f) at ($ (-36:2) + (1,0) $);
\filldraw[fill=black!10!white, draw=none] (2,0.4) .. controls (2.5,0.4) and (2.6,0.5) .. (2.6,0) .. controls (2.6,-0.5) and (2.5,-0.4) .. (2,-0.4) -- (2,-1) -- (f) arc (-36:36:2) -- (2,1) -- cycle;
\draw [suture] (e) -- (2,1);
\draw [suture] (f) -- (2,-1);
\draw [suture] (2,0.4) .. controls (2.5,0.4) and (2.6,0.5) .. (2.6,0) .. controls (2.6,-0.5) and (2.5,-0.4) .. (2,-0.4);
\draw [boundary](2,1.732) arc (60:-60:2);

\foreach \point in {a, b, c, d, e, f}
\fill [vertex] (\point) circle (2pt);
\end{tikzpicture}
\\ 
\hline
\end{tabular}
\]

Note that the $p$-creation operator always creates a new white region, and the $q$-creation operator always creates a new black region. Also, the $p$-annihilation operator always closes off a black region, while the $q$-annihilation operator always closes off a white region. The similarity to pawn colours is not coincidental.

\subsection{Diagrams of chessboards}

We'll call the simplest possible chord diagram, with one chord, \emph{the vacuum} $\Gamma_\emptyset$.

\begin{center}

\begin{tikzpicture}[
scale=0.9, 
suture/.style={thick, draw=red},
boundary/.style={ultra thick},
vertex/.style={draw=red, fill=red}]

\coordinate [label = above:{$0$}] (0) at (90:1);
\coordinate [label = below:{$1$}] (1) at (-90:1);
\filldraw[fill=black!10!white, draw=none] (0) arc (90:-90:1) -- cycle;
\draw [boundary] (0,0) circle (1 cm);
\draw [suture] (0) -- (1);
\foreach \point in {0, 1}
\fill [vertex] (\point) circle (2pt);

\end{tikzpicture}

\end{center}

Now for any word $w$ in the letters $p$ and $q$, we can apply a corresponding sequence of initial creation operators $a_{p,0}^*$ and $a_{q,0}^\dagger$ to the vacuum chord diagram $\Gamma_\emptyset$ to get a chord diagram $\Gamma_w$ for the word $w$.

For example, for the word $w = qpqq$ we obtain
\[
\begin{tikzpicture}[
scale=0.9, 
suture/.style={thick, draw=red},
boundary/.style={ultra thick},
vertex/.style={draw=red, fill=red}]
\coordinate [label = above:{$0$}] (0) at (90:1);
\coordinate [label = below:{$1$}] (1) at (-90:1);
\filldraw[fill=black!10!white, draw=none] (0) arc (90:-90:1) -- cycle;
\draw [boundary] (0,0) circle (1 cm);
\draw [suture] (0) -- (1);
\draw (-7,0) node {$\Gamma_{qpqq} \quad = \quad a_{q,0}^\dagger a_{p,0}^* a_{q,0}^\dagger a_{q,0}^\dagger \quad \Gamma_\emptyset \quad = \quad a_{q,0}^\dagger a_{p,0}^* a_{q,0}^\dagger a_{q,0}^\dagger$};
\foreach \point in {0, 1}
\fill [vertex] (\point) circle (2pt);
\end{tikzpicture}
\]
\[
\begin{tikzpicture}[
scale=0.9, 
suture/.style={thick, draw=red},
boundary/.style={ultra thick},
vertex/.style={draw=red, fill=red}]
\coordinate [label = above:{$0$}] (0) at (90:1);
\coordinate [label = right:{$1$}] (1) at (0:1);
\coordinate [label = below:{$2$}] (2) at (-90:1);
\coordinate [label = left:{$-1$}] (3) at (180:1);
\filldraw[fill=black!10!white, draw=none] (0) to [bend right=45] (1) arc (0:90:1);
\filldraw[fill=black!10!white, draw=none] (2) to [bend right=45] (3) arc (-180:-90:1);
\draw [boundary] (0,0) circle (1 cm);
\draw [suture] (0) to [bend right=45] (1);
\draw [suture] (2) to [bend right=45] (3);
\draw (-4,0) node {$= a_{q,0}^\dagger a_{p,0}^* a_{q,0}^\dagger $};
\foreach \point in {0,1,2,3}
\fill [vertex] (\point) circle (2pt);
\end{tikzpicture}
\begin{tikzpicture}[
scale=0.9, 
suture/.style={thick, draw=red},
boundary/.style={ultra thick},
vertex/.style={draw=red, fill=red}]
\coordinate [label = above:{$0$}] (0) at (90:1);
\coordinate [label = above right:{$1$}] (1) at (30:1);
\coordinate [label = below right:{$2$}] (2) at (-30:1);
\coordinate [label = below:{$3$}] (3) at (-90:1);
\coordinate [label = below left:{$-2$}] (4) at (-150:1);
\coordinate [label = above left:{$-1$}] (5) at (150:1);
\filldraw[fill=black!10!white, draw=none] (0) to [bend right=90] (1) arc (30:90:1);
\filldraw[fill=black!10!white, draw=none] (2) to [bend right=90] (3) arc (-90:0:1);
\filldraw[fill=black!10!white, draw=none] (4) to [bend right=90] (5) arc (150:210:1);
\draw [boundary] (0,0) circle (1 cm);
\draw [suture] (0) to [bend right=90] (1);
\draw [suture] (2) to [bend right=90] (3);
\draw [suture] (4) to [bend right=90] (5);
\draw (-4,0) node {$\quad = \quad a_{q,0}^\dagger a_{p,0}^*$};
\foreach \point in {0,1,2,3,4,5}
\fill [vertex] (\point) circle (2pt);
\end{tikzpicture}
\]
\[
\begin{tikzpicture}[
scale=0.9, 
suture/.style={thick, draw=red},
boundary/.style={ultra thick},
vertex/.style={draw=red, fill=red}]
\coordinate [label = above:{$0$}] (0) at (90:1.5);
\coordinate [label = above right:{$1$}] (1) at (45:1.5);
\coordinate [label = right:{$2$}] (2) at (0:1.5);
\coordinate [label = below right:{$3$}] (3) at (-45:1.5);
\coordinate [label = below:{$4$}] (4) at (-90:1.5);
\coordinate [label = below left:{$-3$}] (5) at (-135:1.5);
\coordinate [label = left:{$-2$}] (6) at (-180:1.5);
\coordinate [label = above left:{$-1$}] (7) at (-225:1.5);
\filldraw[fill=black!10!white, draw=none] (0) arc (90:45:1.5) to [bend left=15] (6) arc (180:135:1.5) to [bend right=90] (0);
\filldraw[fill=black!10!white, draw=none] (2) to [bend right=90] (3) arc (-45:0:1.5);
\filldraw[fill=black!10!white, draw=none] (4) to [bend right=90] (5) arc (-135:-90:1.5);
\draw [boundary] (0,0) circle (1.5 cm);
\draw [suture] (0) to [bend left=90] (7);
\draw [suture] (6) to [bend right=15] (1);
\draw [suture] (2) to [bend right=90] (3);
\draw [suture] (4) to [bend right=90] (5);
\draw (-3,0) node {$= \quad a_{q,0}^\dagger$};
\foreach \point in {0,1,2,3,4,5,6,7}
\fill [vertex] (\point) circle (2pt);
\end{tikzpicture}
\begin{tikzpicture}[
scale=0.9, 
suture/.style={thick, draw=red},
boundary/.style={ultra thick},
vertex/.style={draw=red, fill=red}]
\coordinate [label = above:{$0$}] (0) at (90:2);
\coordinate [label = above right:{$1$}] (1) at (54:2);
\coordinate [label = above right:{$2$}] (2) at (18:2);
\coordinate [label = right:{$3$}] (3) at (-18:2);
\coordinate [label = below right:{$4$}] (4) at (-54:2);
\coordinate [label = below right:{$5$}] (5) at (-90:2);
\coordinate [label = below:{$-4$}] (6) at (-126:2);
\coordinate [label = below left:{$-3$}] (7) at (-162:2);
\coordinate [label = below left:{$-2$}] (8) at (162:2);
\coordinate [label = left:{$-1$}] (9) at (126:2);
\filldraw[fill=black!10!white, draw=none] (0) to [bend right=90] (1) arc (54:90:2);
\filldraw[fill=black!10!white, draw=none] (9) to [bend right=30] (2) arc (18:-18:2) -- (8) arc (162:126:2);
\filldraw[fill=black!10!white, draw=none] (4) to [bend right=90] (5) arc (-90:-54:2);
\filldraw[fill=black!10!white, draw=none] (6) to [bend right=90] (7) arc (-162:-126:2);
\draw[suture] (9) to [bend right=30] (2);
\draw[suture] (0) to [bend right=90] (1);
\draw[suture] (3) -- (8);
\draw[suture] (4) to [bend right=90] (5);
\draw[suture] (6) to [bend right=90] (7);
\draw (-4,0) node {$\quad = \quad$};
\draw [boundary] (0,0) circle (2 cm);
\foreach \point in {0, 1, 2, 3, 4, 5, 6, 7, 8, 9}
\fill [vertex] (\point) circle (2pt);
\end{tikzpicture}
\]

Now a word $w$ in $p$ and $q$ corresponds to a chessboard, where each $p$ stands for a pawn and each $q$ stands for an anti-pawn. So we can say that the chessboard $w$ has chord diagram $\Gamma_w$.

\[
\begin{tikzpicture}[
scale=0.9, 
suture/.style={thick, draw=red},
boundary/.style={ultra thick},
vertex/.style={draw=red, fill=red}]
\coordinate [label = above:{$0$}] (0) at (90:2);
\coordinate [label = above right:{$1$}] (1) at (54:2);
\coordinate [label = above right:{$2$}] (2) at (18:2);
\coordinate [label = right:{$3$}] (3) at (-18:2);
\coordinate [label = below right:{$4$}] (4) at (-54:2);
\coordinate [label = below right:{$5$}] (5) at (-90:2);
\coordinate [label = below:{$-4$}] (6) at (-126:2);
\coordinate [label = below left:{$-3$}] (7) at (-162:2);
\coordinate [label = below left:{$-2$}] (8) at (162:2);
\coordinate [label = left:{$-1$}] (9) at (126:2);
\filldraw[fill=black!10!white, draw=none] (0) to [bend right=90] (1) arc (54:90:2);
\filldraw[fill=black!10!white, draw=none] (9) to [bend right=30] (2) arc (18:-18:2) -- (8) arc (162:126:2);
\filldraw[fill=black!10!white, draw=none] (4) to [bend right=90] (5) arc (-90:-54:2);
\filldraw[fill=black!10!white, draw=none] (6) to [bend right=90] (7) arc (-162:-126:2);
\draw[suture] (9) to [bend right=30] (2);
\draw[suture] (0) to [bend right=90] (1);
\draw[suture] (3) -- (8);
\draw[suture] (4) to [bend right=90] (5);
\draw[suture] (6) to [bend right=90] (7);
\draw (-8,0) node {
$qpqq \quad 
\leftrightarrow \quad 
\begin{tabularx}{0.3\textwidth}{|X|X|X|X|}
\hline
\BlackPawnOnWhite & \WhitePawnOnWhite & \BlackPawnOnWhite & \BlackPawnOnWhite \\
\hline 
\end{tabularx}
\quad
\leftrightarrow \quad$};
\draw [boundary] (0,0) circle (2 cm);
\foreach \point in {0, 1, 2, 3, 4, 5, 6, 7, 8, 9}
\fill [vertex] (\point) circle (2pt);
\end{tikzpicture}
\]

It turns out that the creation and annihilation perators $a_{p,i}, a_{p,i}^*, a_{q,i}, a_{q,i}^\dagger$ act on chessboards and chord diagrams coherently.

\begin{prop}
For any chessboard/word $w$,
\[
\Gamma_{a_{p,i}^* w} = a_{p,i}^* \Gamma_w.
\]
In other words, the following diagram commutes.
\[
\begin{tikzpicture}
\draw (0,0) node {$w$};
\draw (3,0) node {$\Gamma_w$};
\draw (0,-2) node {$a_{p,i}^* w$};
\draw (3,-2) node {$\Gamma_{a_{p,i}^* w}$};
\draw [->,decorate, decoration={snake,amplitude=.4mm,segment length=2mm,post length=1mm}] (0.5,0) -- (2.5,0)
	node [above, align=center, midway] {Draw\\diagram};
\draw [->,decorate, decoration={snake,amplitude=.4mm,segment length=2mm,post length=1mm}] (0.5,-2) -- (2.5,-2)
	node [above, align=center, midway] {Draw\\diagram};
\draw [->] (0,-0.4) -- (0,-1.6)
	node [left, align=center, midway] {Create\\pawn};
\draw [->] (3,-0.4) -- (3,-1.6) node [right, align=center, midway] {Create\\chord};
\end{tikzpicture}
\]
\end{prop}

A similar result holds for operators $a_{q,i}^\dagger$. And also for annihilation operators $a_{p,i}$, $a_{q,i}$, provided we take any diagram with a closed loop to be zero. Annihilating a pawn in the wrong place gives an error ``universe not found''; annihilating a chord in the wrong place, so that the result is not a chord diagram, gives an error ``chord diagram not found''.

\subsection{Chord diagrams of chessboards as ski slopes}

In fact there's a quicker way to draw the diagram of a chessboard, which also illustrates why the above proposition is true. But we need to consider a different sport: skiing.

We imagine we have a ski slope, where as usual we ski from top to bottom. It is rectangular in shape, with a starting line at the top, a finishing line at the bottom, and a left and right side along the slope.

We consider a slalom run. There are obstacles, which we can imagine as poles in the ground. As you go down the slope, you have to round all the obstacles, from top to bottom; and after rounding each one you have to return to the centre of the course. There are some (white) poles placed on the left, and some (black) on the right of the slope. We'll write $p$ for a pole on the left, and $q$ for a pole on the right. For the slalom course $qpqq$ the course looks as follows.
\[
\begin{tikzpicture}[
scale=0.7, 
suture/.style={thick, draw=red},
boundary/.style={ultra thick},
vertex/.style={draw=red, fill=red},
leftpole/.style={draw=black, fill=white},
rightpole/.style={draw=black, fill=black}]
\draw [boundary] (-2,0) -- (6,0) -- (6,-9) -- (-2,-9) -- cycle;
\coordinate (0) at (2,0);
\coordinate (1) at (4,-1);
\coordinate (2) at (4,-2);
\coordinate (-1) at (0,-3);
\coordinate (-2) at (0,-4);
\coordinate (3) at (4,-5);
\coordinate (4) at (4,-6);
\coordinate (5) at (4,-7);
\coordinate (6) at (4,-8);
\coordinate (7) at (2,-9);
\coordinate (p1) at (0,-3.5);
\coordinate (q1) at (4,-1.5);
\coordinate (q2) at (4,-5.5);
\coordinate (q3) at (4,-7.5);
\draw [suture] (0) .. controls (2,-0.5) and (3.5,-1) .. (1) arc (90:-90:0.5) 
.. controls (3.5,-2) and (0.5,-3) .. (-1)
arc (90:270:0.5)
.. controls (0.5,-4) and (3.5,-5) .. (3)
arc (90:-90:0.5)
.. controls (3.5,-6) and (2,-6.25) .. (2,-6.5) .. controls (2,-6.75) and (3.5,-7) .. (5)
arc (90:-90:0.5)
.. controls (3.5,-8) and (2,-8.5) .. (7);
\foreach \point in {p1}
\fill [leftpole] (\point) circle (5pt);
\foreach \point in {q1, q2, q3}
\fill [rightpole] (\point) circle (5pt);
\end{tikzpicture}
\]
Now, if we just consider that part of the course inside the obstacles, we get an interesting chord diagram. In fact, it is just the chord diagram $\Gamma_{qpqq}$.
\[
\begin{tikzpicture}[
scale=0.7, 
suture/.style={thick, draw=red},
boundary/.style={ultra thick},
vertex/.style={draw=red, fill=red},
leftpole/.style={draw=black, fill=white},
rightpole/.style={draw=black, fill=black}]
\coordinate [label = above:{$0$}] (0) at (2,0);
\coordinate [label = right:{$1$}] (1) at (4,-1);
\coordinate [label = right:{$2$}] (2) at (4,-2);
\coordinate [label = left:{$-1$}] (-1) at (0,-3);
\coordinate [label = left:{$-2$}] (-2) at (0,-4);
\coordinate [label = right:{$3$}] (3) at (4,-5);
\coordinate [label = right:{$4$}] (4) at (4,-6);
\coordinate [label = right:{$5$}] (5) at (4,-7);
\coordinate [label = right:{$6$}] (6) at (4,-8);
\coordinate [label = below:{$7$}] (7) at (2,-9);
\coordinate (p1) at (0,-3.5);
\coordinate (q1) at (4,-1.5);
\coordinate (q2) at (4,-5.5);
\coordinate (q3) at (4,-7.5);
\filldraw[fill=black!10!white, draw=none] (0) .. controls (2,-0.5) and (3.5,-1) .. (1) -- (4,0) -- cycle;
\filldraw[fill=black!10!white, draw=none] (2) .. controls (3.5,-2) and (0.5,-3) .. (-1) -- (-2) .. controls (0.5,-4) and (3.5,-5) .. (3) -- cycle;
\filldraw[fill=black!10!white, draw=none] (4) .. controls (3.5,-6) and (2,-6.25) .. (2,-6.5) .. controls (2,-6.75) and (3.5,-7) .. (5);
\filldraw[fill=black!10!white, draw=none] (6) .. controls (3.5,-8) and (2,-8.5) .. (7) -- (4,-9) -- cycle;
\draw [suture] (0) .. controls (2,-0.5) and (3.5,-1) .. (1);
\draw [suture, dotted] (1) arc (90:-90:0.5);
\draw [suture] (2) .. controls (3.5,-2) and (0.5,-3) .. (-1);
\draw [suture, dotted] (-1) arc (90:270:0.5);
\draw [suture] (-2) .. controls (0.5,-4) and (3.5,-5) .. (3);
\draw [suture, dotted] (3) arc (90:-90:0.5);
\draw [suture] (4) .. controls (3.5,-6) and (2,-6.25) .. (2,-6.5) .. controls (2,-6.75) and (3.5,-7) .. (5);
\draw [suture, dotted] (5) arc (90:-90:0.5);
\draw [suture] (6) .. controls (3.5,-8) and (2,-8.5) .. (7);
\draw [boundary] (0,0) -- (4,0) -- (4,-9) -- (0,-9) -- cycle;
\draw [dotted] (0,0) -- (-2,0) -- (-2,-9) -- (0,-9);
\draw [dotted] (4,0) -- (6,0) -- (6,-9) -- (4,-9);
\foreach \point in {0, 1, 2, 3, 4, 5, 6, 7, -2, -1}
\fill [vertex] (\point) circle (2pt);
\foreach \point in {p1}
\fill [leftpole] (\point) circle (5pt);
\foreach \point in {q1, q2, q3}
\fill [rightpole] (\point) circle (5pt);
\end{tikzpicture}
\]
It's not difficult to see why this skiing algorithm gives the correct chord diagram for each word/chessboard. Having the chord diagram arranged this way, with all the chords coming from pawns / $p$'s on the left, and all the chords coming from anti-pawns / $q$'s on the right, makes it easier to see why the chord diagram operations correspond to operations on chessboards. For instance, $a_{p,i}^*$ adds an extra component to the ski run, entering from the left side, then sharply turning back to the left hand side of the slope. This is precisely what you get when you add an extra pole on the left.

\subsection{Square decomposition}

We can also note that the pawns and anti-pawns correspond to a precise decomposition of the chord diagram into squares, as shown by the green lines.

\[
\begin{tikzpicture}[
scale=0.7, 
suture/.style={thick, draw=red},
boundary/.style={ultra thick},
vertex/.style={draw=red, fill=red},
rightpole/.style={draw=black, fill=black},
leftpole/.style ={draw=black, fill=white},
decomposition/.style={thick, draw=green!50!black}]
\draw (-8,-4.5) node {
$\begin{tabularx}{0.3\textwidth}{|X|X|X|X|}
\hline
\scalebox{0.9}{\BlackPawnOnWhite} 
& \scalebox{0.9}{\WhitePawnOnWhite} & 
\scalebox{0.9}{\BlackPawnOnWhite} 
& 
\scalebox{0.9}{\BlackPawnOnWhite}
\\
\hline 
\end{tabularx}$};
\coordinate [label = below left:{$0$}] (0) at (2,0);
\coordinate [label = right:{$1$}] (1) at (4,-1);
\coordinate [label = right:{$2$}] (2) at (4,-2);
\coordinate [label = left:{$-1$}] (-1) at (0,-3);
\coordinate [label = left:{$-2$}] (-2) at (0,-4);
\coordinate [label = right:{$3$}] (3) at (4,-5);
\coordinate [label = right:{$4$}] (4) at (4,-6);
\coordinate [label = right:{$5$}] (5) at (4,-7);
\coordinate [label = right:{$6$}] (6) at (4,-8);
\coordinate [label = above left:{$7$}] (7) at (2,-9);
\coordinate (p1) at (0,-3.5);
\coordinate (q1) at (4,-1.5);
\coordinate (q2) at (4,-5.5);
\coordinate (q3) at (4,-7.5);
\filldraw[fill=black!10!white, draw=none] (0) .. controls (2,-0.5) and (3.5,-1) .. (1) -- (4,0) -- cycle;
\filldraw[fill=black!10!white, draw=none] (2) .. controls (3.5,-2) and (0.5,-3) .. (-1) -- (-2) .. controls (0.5,-4) and (3.5,-5) .. (3) -- cycle;
\filldraw[fill=black!10!white, draw=none] (4) .. controls (3.5,-6) and (2,-6.25) .. (2,-6.5) .. controls (2,-6.75) and (3.5,-7) .. (5);
\filldraw[fill=black!10!white, draw=none] (6) .. controls (3.5,-8) and (2,-8.5) .. (7) -- (4,-9) -- cycle;
\draw [suture] (0) .. controls (2,-0.5) and (3.5,-1) .. (1);
\draw [suture] (2) .. controls (3.5,-2) and (0.5,-3) .. (-1);
\draw [suture] (-2) .. controls (0.5,-4) and (3.5,-5) .. (3);
\draw [suture] (4) .. controls (3.5,-6) and (2,-6.25) .. (2,-6.5) .. controls (2,-6.75) and (3.5,-7) .. (5);
\draw [suture] (6) .. controls (3.5,-8) and (2,-8.5) .. (7);
\draw [decomposition] (0,-2.5) -- (4,-2.5);
\draw [decomposition] (0,-4.5) -- (4,-4.5);
\draw [decomposition] (0,-6.5) -- (4,-6.5);
\draw [boundary] (0,0) -- (4,0) -- (4,-9) -- (0,-9) -- cycle;
\draw (-2,-4.5) node {$\leftrightarrow$};
\foreach \point in {0, 1, 2, 3, 4, 5, 6, 7, -2, -1}
\fill [vertex] (\point) circle (2pt);
\foreach \point in {q1, q2, q3}
\fill [rightpole] (\point) circle (5pt);
\foreach \point in {p1}
\fill [leftpole] (\point) circle (5pt);
\end{tikzpicture}
\]

One can essentially read the chord diagram as a chessboard in this way: reading the ski slope from top to bottom, reads off the chessboard from left to right.

Moreover, we see that each of these squares in the chord diagram contains chords in a particular configuration, corresponding to a pawn or anti-pawn.

\[
\begin{tikzpicture}[
scale=2, 
suture/.style={thick, draw=red}, 
]
\coordinate (1tl) at (0,1);
\coordinate (1tr) at (1,1);
\coordinate (1bl) at (0,0);
\coordinate (1br) at (1,0);

\draw (-1,0.5) node {$\WhitePawnOnWhite \quad \leftrightarrow$};

\filldraw[fill=black!10!white, draw=none] (0.5,1) -- (1tr) -- (1,0.5) to [bend right=45] (0.5,0) -- (1bl) -- (0,0.5) to [bend right=45] (0.5,1);

\draw (1bl) -- (1br) -- (1tr) -- (1tl) -- cycle;
\draw [suture] (0.5,0) to [bend left=45] (1,0.5);
\draw [suture] (0,0.5) to [bend right=45] (0.5,1);
\foreach \point in {1bl, 1br, 1tl, 1tr}
\fill [black] (\point) circle (1pt);
\end{tikzpicture}
\quad
\quad
\begin{tikzpicture}[
scale=2, 
suture/.style={thick, draw=red}, 
]
\coordinate (2tl) at (3,1);
\coordinate (2tr) at (4,1);
\coordinate (2bl) at (3,0);
\coordinate (2br) at (4,0);

\draw (2,0.5) node {$\BlackPawnOnWhite \quad \leftrightarrow$};

\filldraw[fill=black!10!white, draw=none] (3.5,1) -- (2tr) -- (4,0.5) to [bend left=45] (3.5,1);
\filldraw[fill=black!10!white, draw=none] (3,0.5) to [bend left=45] (3.5,0) -- (2bl) -- cycle;

\draw (2bl) -- (2br) -- (2tr) -- (2tl) -- cycle;
\draw [suture] (3.5,1) to [bend right=45] (4,0.5);
\draw [suture] (3.5,0) to [bend right=45] (3,0.5);

\foreach \point in {2bl, 2br, 2tl, 2tr}
\fill [black] (\point) circle (1pt);

\end{tikzpicture}
\]

In this way, each pawn or anti-pawn provides one of two ways of drawing in chords into the chord diagram --- providing, in a sense, \emph{one bit of information} and suggesting relations to quantum information theory. Moreover, since we have creation and annihilation operators for pawns and anti-pawns, and they can also be regarded as bits of information, one is reminded of the ``it from bit'' idea of John Archibald Wheeler \cite{Wheeler90}. These ideas are more fully developed in \cite{Me12_itsy_bitsy}.

This decomposition of a chord diagram into pieces, each of which has curves in one of two specified configurations, is also reminiscent of statistical mechanics.

\subsection{Curves on cylinders}
\label{sec:curves_on_cylinders}

Consider the cylinder shown. Its boundary consists of discs on the top and bottom, and a vertical annulus. On the vertical annulus we have some vertical curves, drawn in red.

\begin{center}
\includegraphics[scale=0.5]{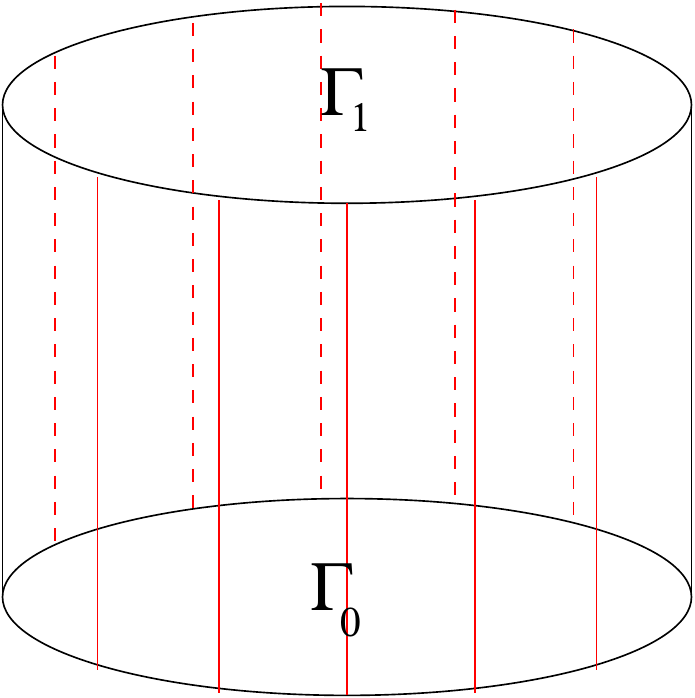}
\end{center}

We are going to draw chord diagrams $\Gamma_0$ and $\Gamma_1$ on the bottom and top discs, and then join up all the curves to obtain curves on the boundary of the cylinder, which of course is topologically a sphere.

However, when we do so, we draw the curves are arranged along the corners as shown:
\begin{center}
\includegraphics[scale=0.4]{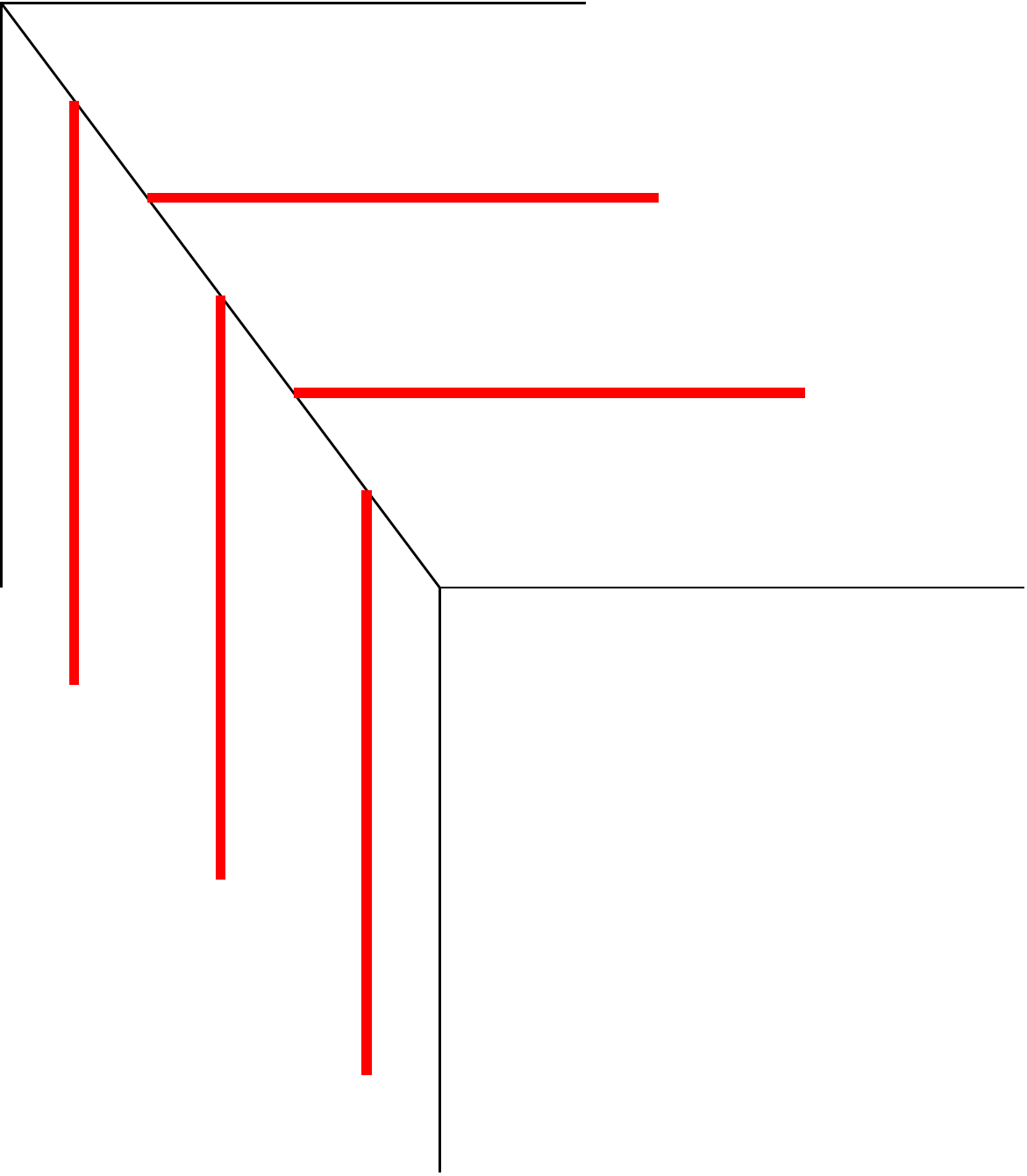}
\end{center}
When we connect up the curves, we do it in the following fashion. We could imagine that we are turning and walking along the corner, as shown on the left; or rounding the corners and the curves, as shown on the right.

\begin{center}
\includegraphics[scale=0.4]{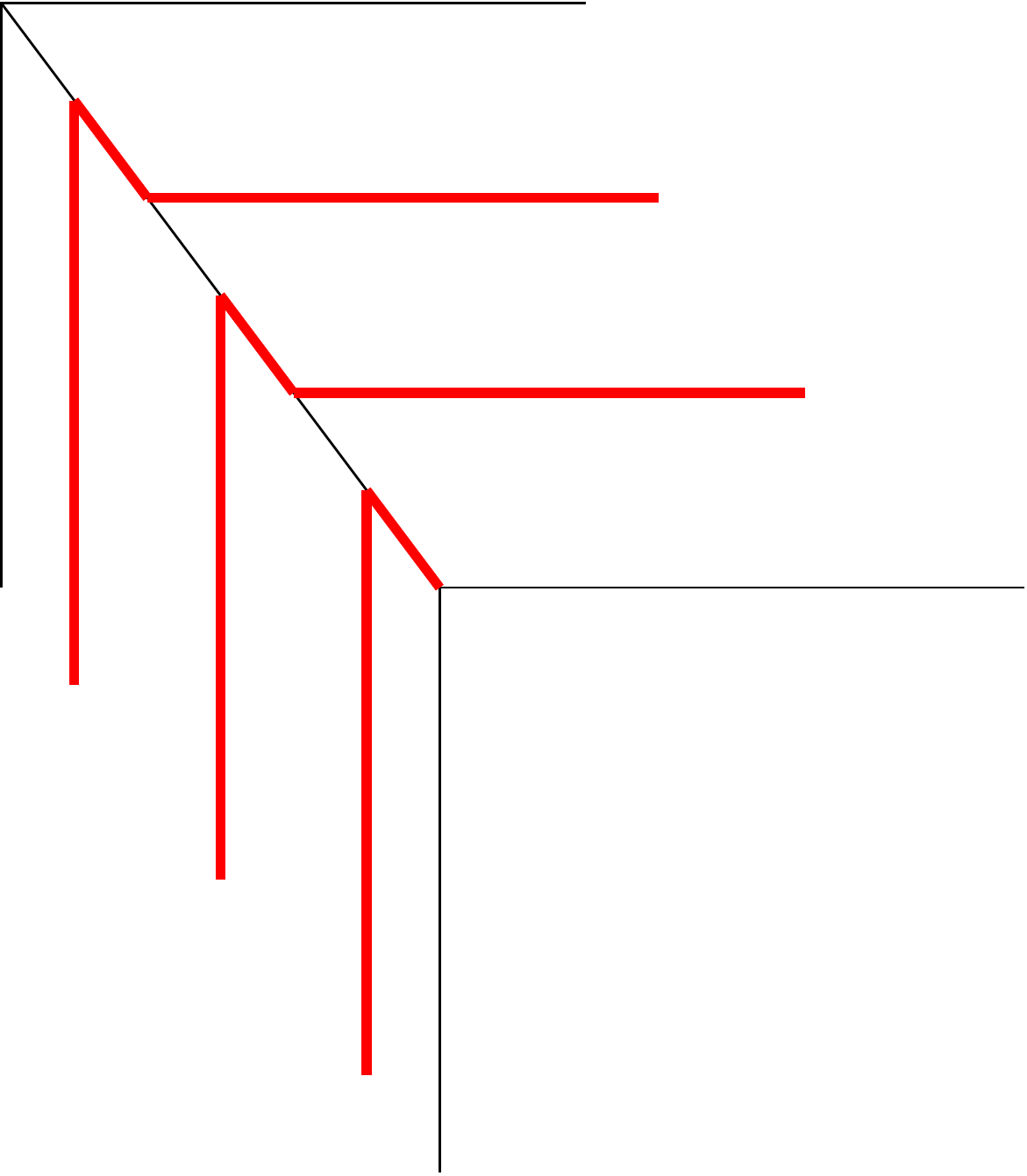}
\quad 
\includegraphics[scale=0.4]{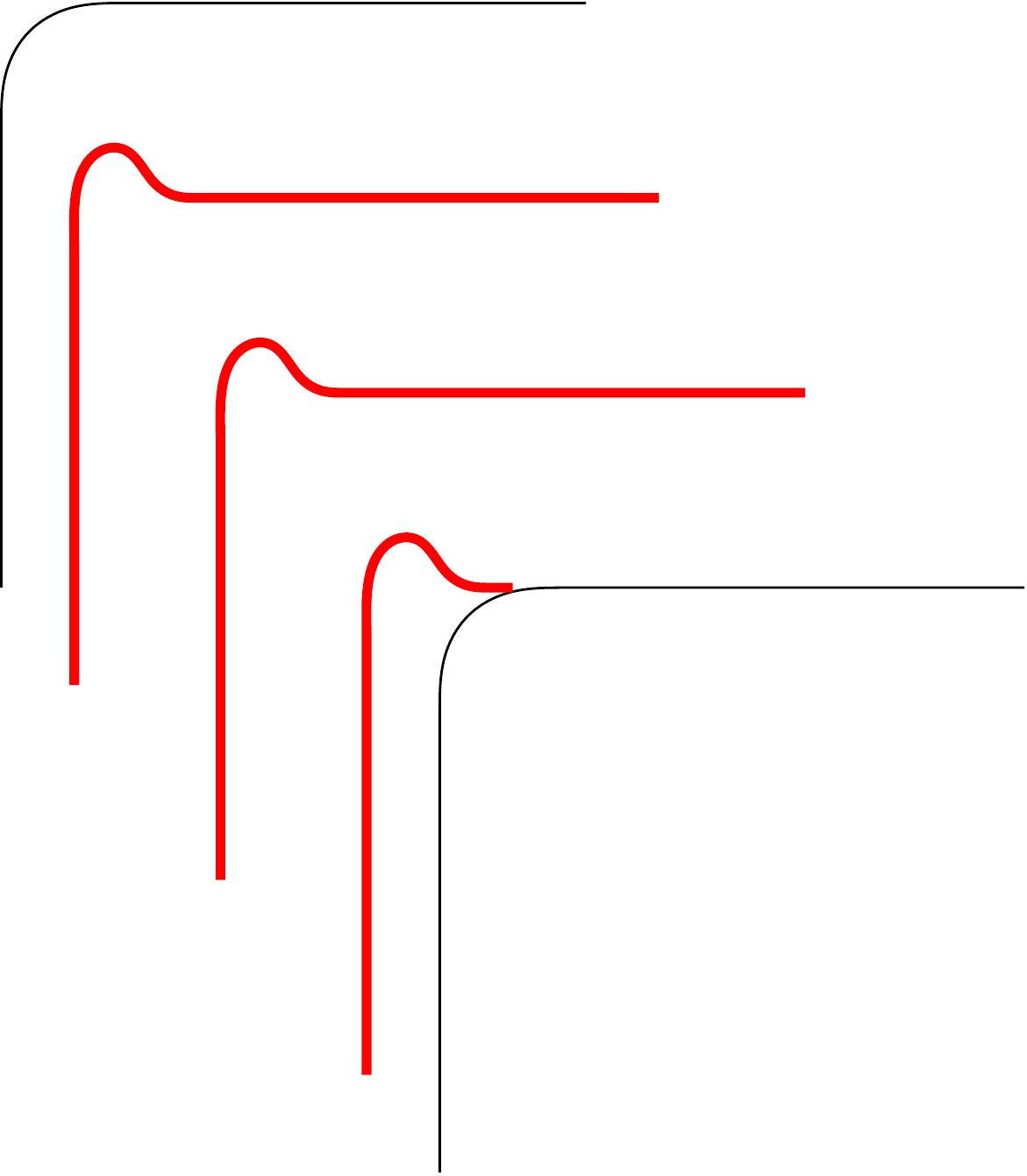}
\end{center}

Here we have an example, with a chord diagram at the top (namely $\Gamma_{qp}$, if we regard the basepoint as being at the back), and another chord diagram at the bottom (namely $\Gamma_{pq}$).
\begin{center}
\includegraphics[scale=0.5]{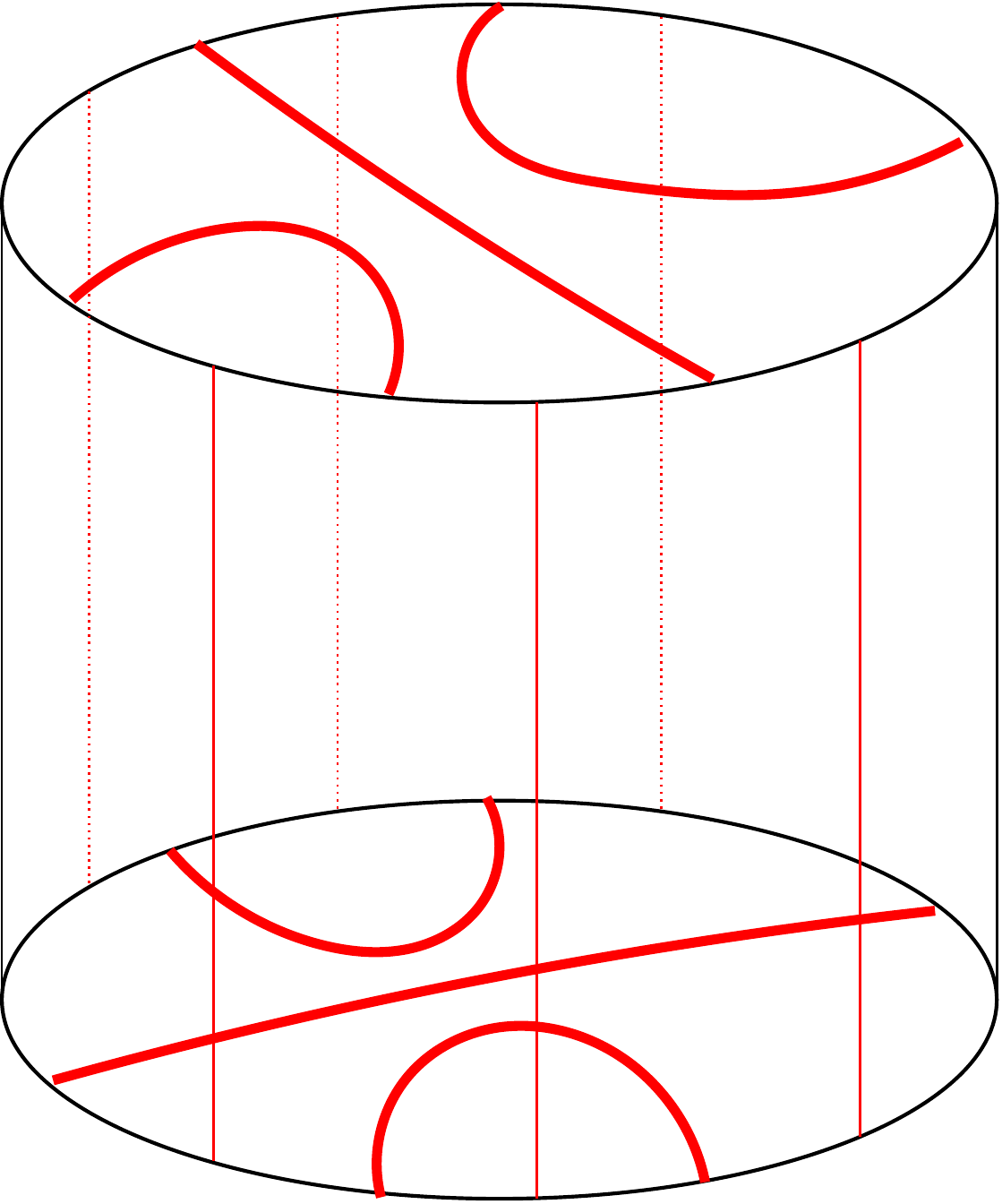}
\end{center}
If we join up the curves as we specified, we obtain the following, which you can check is a single connected curve on the cylinder/sphere. (A fun exercise is to prove that if we draw the \emph{same} chord diagram on the top and bottom of the cylinder, aligned exactly, then joining up the curves on the cylinder in this fashion always gives a single connected curve.)
\begin{center}
\includegraphics[scale=0.5]{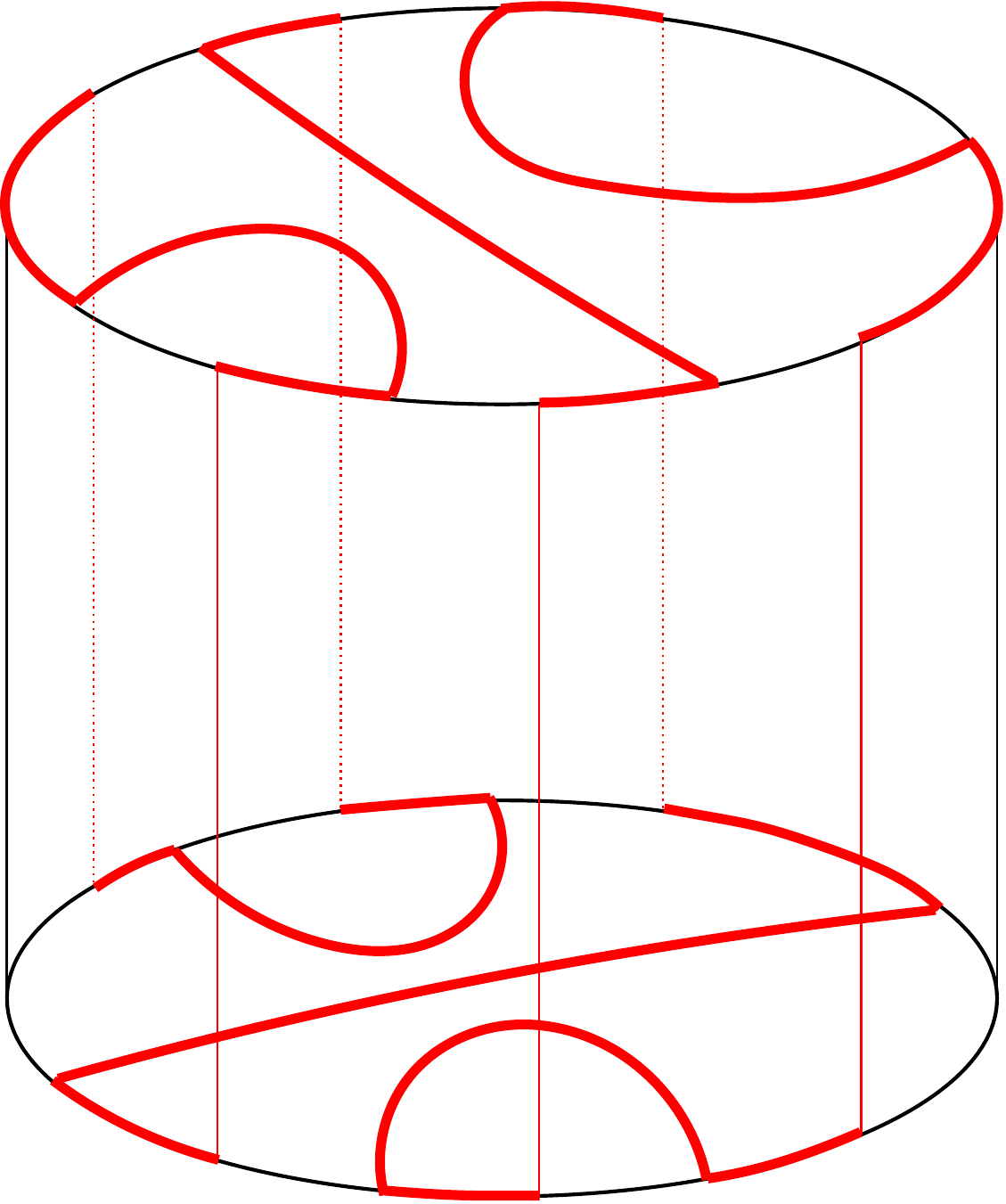}
\end{center}

\subsection{Finger moves}

Now, suppose we perform an annihilation operator $a_{p,i}$ on the chord diagram on the bottom of the cylinder. The topological effect is the same as performing the creation operator $a_{p,i}^*$ on the top of the cylinder --- the two are related by a \emph{finger move} as shown.

\begin{center}
\def\svgwidth{120pt}
\begingroup%
  \makeatletter%
  \providecommand\color[2][]{%
    \errmessage{(Inkscape) Color is used for the text in Inkscape, but the package 'color.sty' is not loaded}%
    \renewcommand\color[2][]{}%
  }%
  \providecommand\transparent[1]{%
    \errmessage{(Inkscape) Transparency is used (non-zero) for the text in Inkscape, but the package 'transparent.sty' is not loaded}%
    \renewcommand\transparent[1]{}%
  }%
  \providecommand\rotatebox[2]{#2}%
  \ifx\svgwidth\undefined%
    \setlength{\unitlength}{340.65934272bp}%
    \ifx\svgscale\undefined%
      \relax%
    \else%
      \setlength{\unitlength}{\unitlength * \real{\svgscale}}%
    \fi%
  \else%
    \setlength{\unitlength}{\svgwidth}%
  \fi%
  \global\let\svgwidth\undefined%
  \global\let\svgscale\undefined%
  \makeatother%
  \begin{picture}(1,1.22209895)%
    \put(0,0){\includegraphics[width=\unitlength]{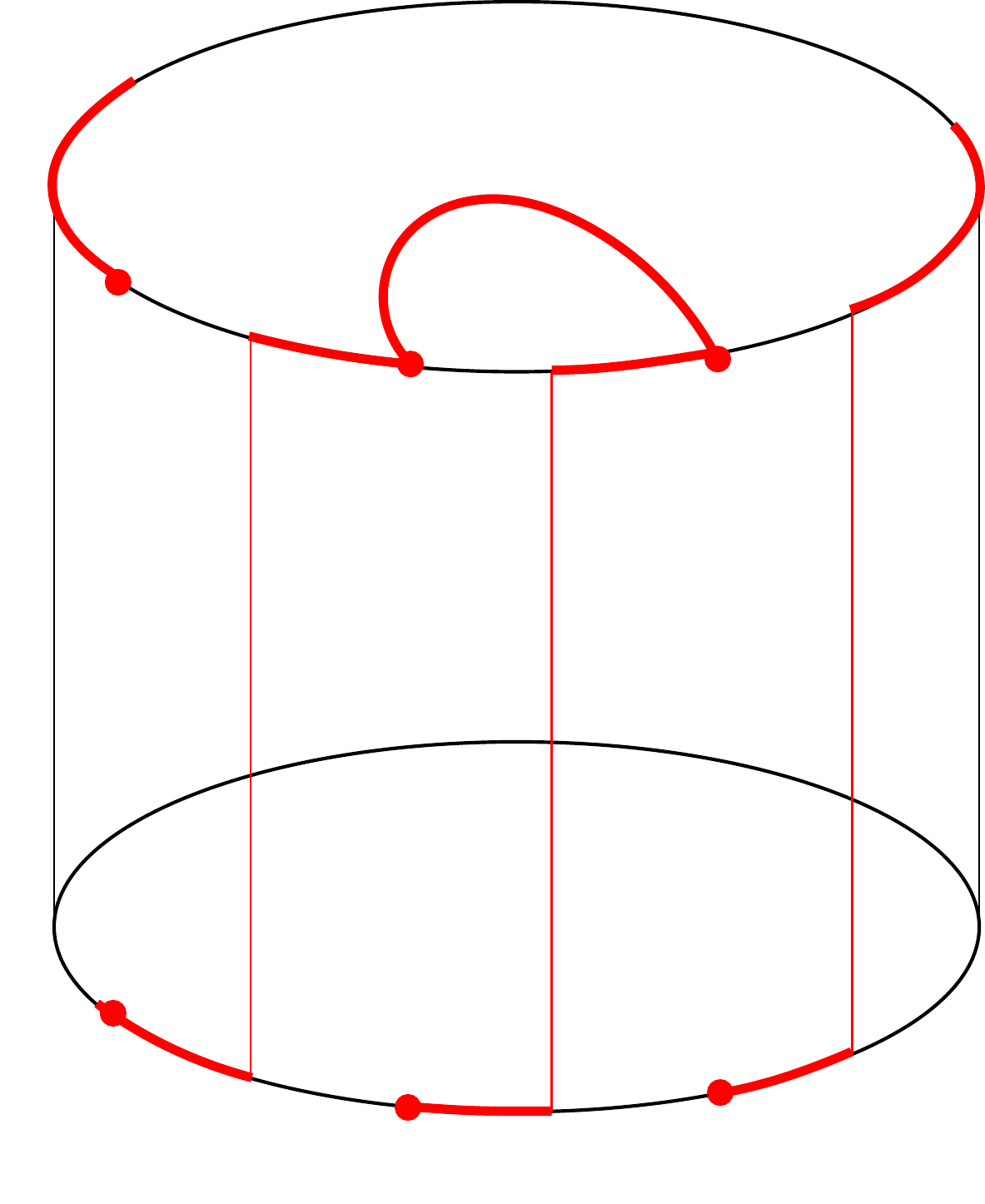}}%
    \put(0.64953066,0.76568568){\color[rgb]{0,0,0}\makebox(0,0)[lb]{\smash{$-2i-1$}}}%
    \put(0.3652226,0.7631697){\color[rgb]{0,0,0}\makebox(0,0)[lb]{\smash{$-2i$}}}%
    \put(0.36773854,0.00836954){\color[rgb]{0,0,0}\makebox(0,0)[lb]{\smash{$-2i$}}}%
    \put(0.62940267,0.01340143){\color[rgb]{0,0,0}\makebox(0,0)[lb]{\smash{$-2i-1$}}}%
    \put(0.0004025,0.08636555){\color[rgb]{0,0,0}\makebox(0,0)[lb]{\smash{$-2i+1$}}}%
    \put(-0.00211352,0.83613374){\color[rgb]{0,0,0}\makebox(0,0)[lb]{\smash{$-2i+1$}}}%
  \end{picture}%
\endgroup%

\quad
\def\svgwidth{120pt}
\begingroup%
  \makeatletter%
  \providecommand\color[2][]{%
    \errmessage{(Inkscape) Color is used for the text in Inkscape, but the package 'color.sty' is not loaded}%
    \renewcommand\color[2][]{}%
  }%
  \providecommand\transparent[1]{%
    \errmessage{(Inkscape) Transparency is used (non-zero) for the text in Inkscape, but the package 'transparent.sty' is not loaded}%
    \renewcommand\transparent[1]{}%
  }%
  \providecommand\rotatebox[2]{#2}%
  \ifx\svgwidth\undefined%
    \setlength{\unitlength}{344.94484687bp}%
    \ifx\svgscale\undefined%
      \relax%
    \else%
      \setlength{\unitlength}{\unitlength * \real{\svgscale}}%
    \fi%
  \else%
    \setlength{\unitlength}{\svgwidth}%
  \fi%
  \global\let\svgwidth\undefined%
  \global\let\svgscale\undefined%
  \makeatother%
  \begin{picture}(1,1.20691591)%
    \put(0,0){\includegraphics[width=\unitlength]{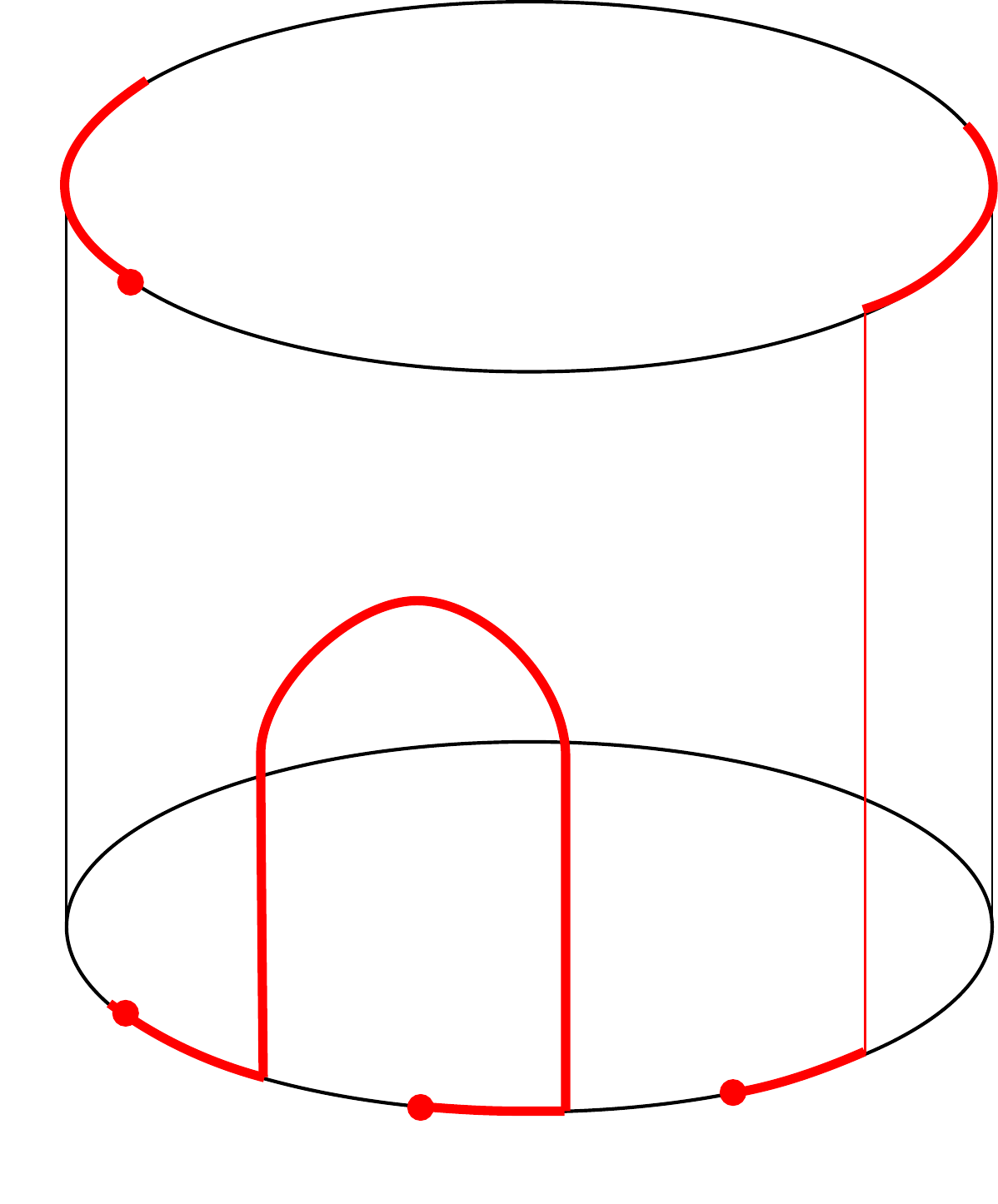}}%
    \put(0.40541048,0.00826556){\color[rgb]{0,0,0}\makebox(0,0)[lb]{\smash{$-2i$}}}%
    \put(0.6538848,0.02317397){\color[rgb]{0,0,0}\makebox(0,0)[lb]{\smash{$-2i-1$}}}%
    \put(0.01033647,0.10517049){\color[rgb]{0,0,0}\makebox(0,0)[lb]{\smash{$-2i+1$}}}%
    \put(-0.00208726,0.82823056){\color[rgb]{0,0,0}\makebox(0,0)[lb]{\smash{$-2i+1$}}}%
  \end{picture}%
\endgroup%

\quad
\def\svgwidth{120pt}
\begingroup%
  \makeatletter%
  \providecommand\color[2][]{%
    \errmessage{(Inkscape) Color is used for the text in Inkscape, but the package 'color.sty' is not loaded}%
    \renewcommand\color[2][]{}%
  }%
  \providecommand\transparent[1]{%
    \errmessage{(Inkscape) Transparency is used (non-zero) for the text in Inkscape, but the package 'transparent.sty' is not loaded}%
    \renewcommand\transparent[1]{}%
  }%
  \providecommand\rotatebox[2]{#2}%
  \ifx\svgwidth\undefined%
    \setlength{\unitlength}{324.90178046bp}%
    \ifx\svgscale\undefined%
      \relax%
    \else%
      \setlength{\unitlength}{\unitlength * \real{\svgscale}}%
    \fi%
  \else%
    \setlength{\unitlength}{\svgwidth}%
  \fi%
  \global\let\svgwidth\undefined%
  \global\let\svgscale\undefined%
  \makeatother%
  \begin{picture}(1,1.2754861)%
    \put(0,0){\includegraphics[width=\unitlength]{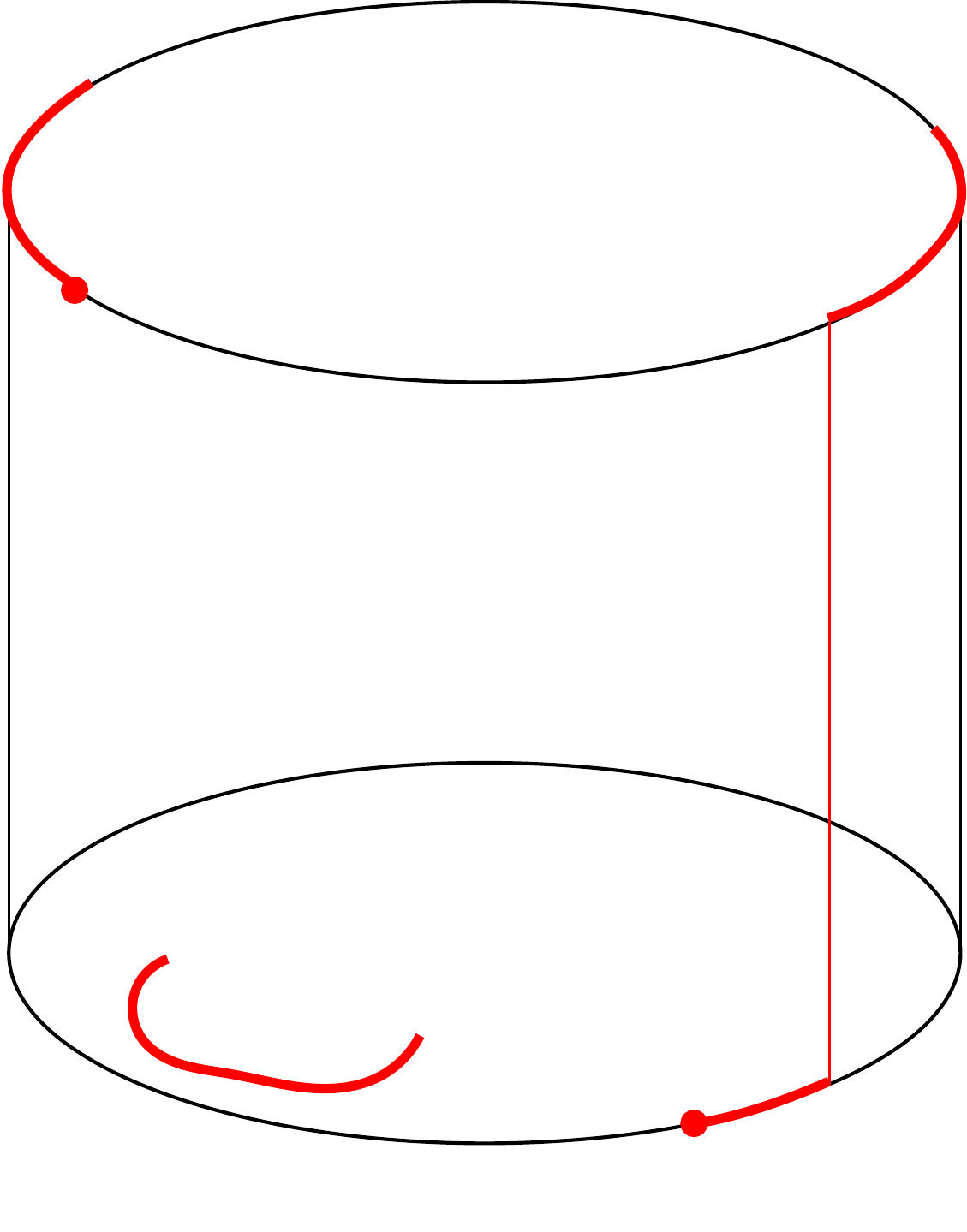}}%
    \put(0.60775854,0.00877546){\color[rgb]{0,0,0}\makebox(0,0)[lb]{\smash{$-2i+1$}}}%
    \put(-0.00221603,0.87617045){\color[rgb]{0,0,0}\makebox(0,0)[lb]{\smash{$-2i+1$}}}%
  \end{picture}%
\endgroup%

\end{center}

Similarly, if we perform the creation operator $a_{q,i}^\dagger$ on the \emph{bottom} of the cylinder, the topological effect is the same as performing the annihilation operator $a_{q,i}$ on the \emph{top}.

\begin{center}
\def\svgwidth{120pt}
\begingroup%
  \makeatletter%
  \providecommand\color[2][]{%
    \errmessage{(Inkscape) Color is used for the text in Inkscape, but the package 'color.sty' is not loaded}%
    \renewcommand\color[2][]{}%
  }%
  \providecommand\transparent[1]{%
    \errmessage{(Inkscape) Transparency is used (non-zero) for the text in Inkscape, but the package 'transparent.sty' is not loaded}%
    \renewcommand\transparent[1]{}%
  }%
  \providecommand\rotatebox[2]{#2}%
  \ifx\svgwidth\undefined%
    \setlength{\unitlength}{340.65934272bp}%
    \ifx\svgscale\undefined%
      \relax%
    \else%
      \setlength{\unitlength}{\unitlength * \real{\svgscale}}%
    \fi%
  \else%
    \setlength{\unitlength}{\svgwidth}%
  \fi%
  \global\let\svgwidth\undefined%
  \global\let\svgscale\undefined%
  \makeatother%
  \begin{picture}(1,1.22209895)%
    \put(0,0){\includegraphics[width=\unitlength]{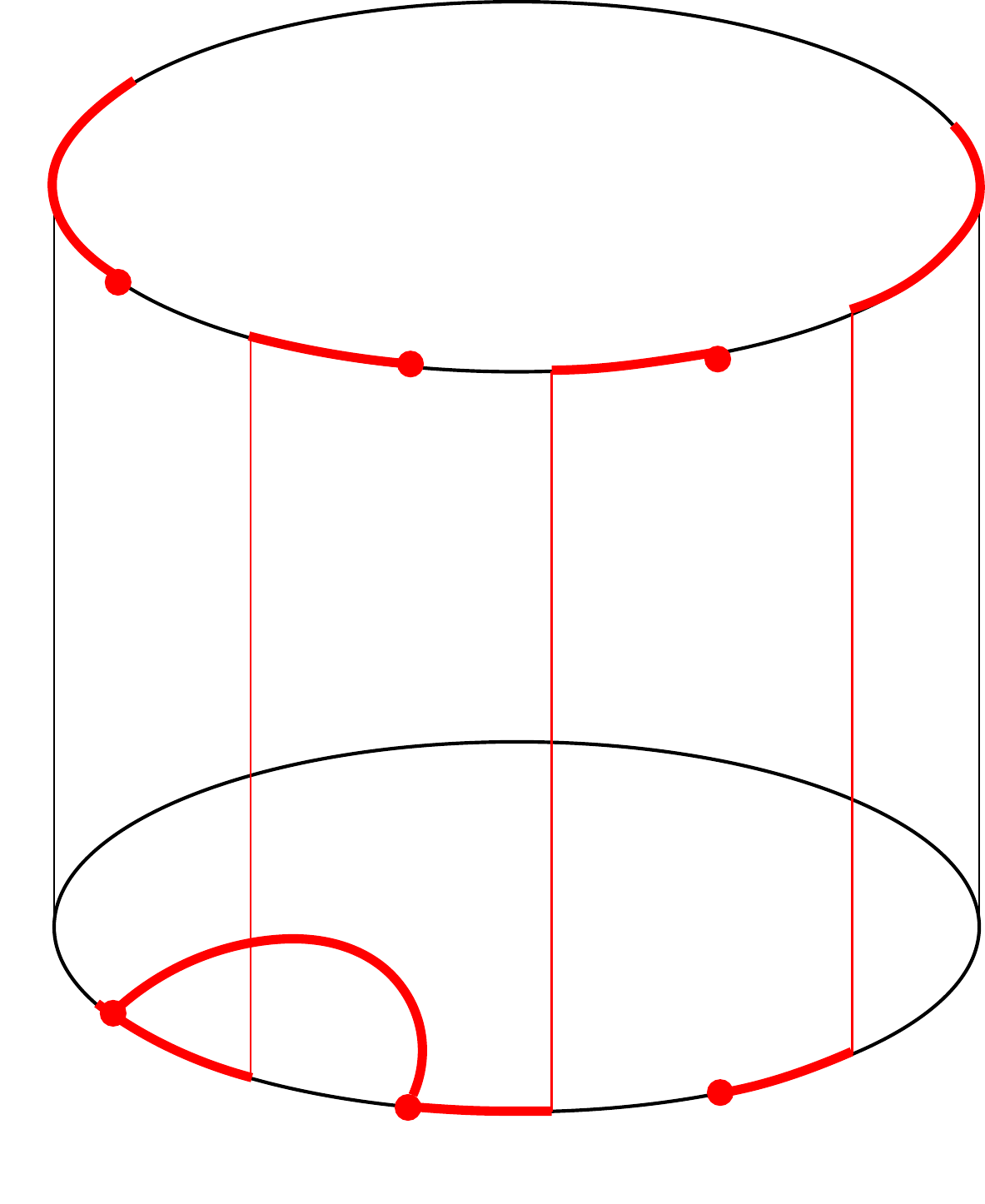}}%
    \put(0.64953066,0.76568568){\color[rgb]{0,0,0}\makebox(0,0)[lb]{\smash{$2i-1$}}}%
    \put(0.3652226,0.7631697){\color[rgb]{0,0,0}\makebox(0,0)[lb]{\smash{$2i$}}}%
    \put(0.36773854,0.00836954){\color[rgb]{0,0,0}\makebox(0,0)[lb]{\smash{$2i$}}}%
    \put(0.62940267,0.01340143){\color[rgb]{0,0,0}\makebox(0,0)[lb]{\smash{$2i-1$}}}%
    \put(0.0004025,0.08636555){\color[rgb]{0,0,0}\makebox(0,0)[lb]{\smash{$2i+1$}}}%
    \put(-0.00211352,0.83613374){\color[rgb]{0,0,0}\makebox(0,0)[lb]{\smash{$2i+1$}}}%
  \end{picture}%
\endgroup%

\quad
\def\svgwidth{120pt}
\begingroup%
  \makeatletter%
  \providecommand\color[2][]{%
    \errmessage{(Inkscape) Color is used for the text in Inkscape, but the package 'color.sty' is not loaded}%
    \renewcommand\color[2][]{}%
  }%
  \providecommand\transparent[1]{%
    \errmessage{(Inkscape) Transparency is used (non-zero) for the text in Inkscape, but the package 'transparent.sty' is not loaded}%
    \renewcommand\transparent[1]{}%
  }%
  \providecommand\rotatebox[2]{#2}%
  \ifx\svgwidth\undefined%
    \setlength{\unitlength}{340.65934272bp}%
    \ifx\svgscale\undefined%
      \relax%
    \else%
      \setlength{\unitlength}{\unitlength * \real{\svgscale}}%
    \fi%
  \else%
    \setlength{\unitlength}{\svgwidth}%
  \fi%
  \global\let\svgwidth\undefined%
  \global\let\svgscale\undefined%
  \makeatother%
  \begin{picture}(1,1.21706706)%
    \put(0,0){\includegraphics[width=\unitlength]{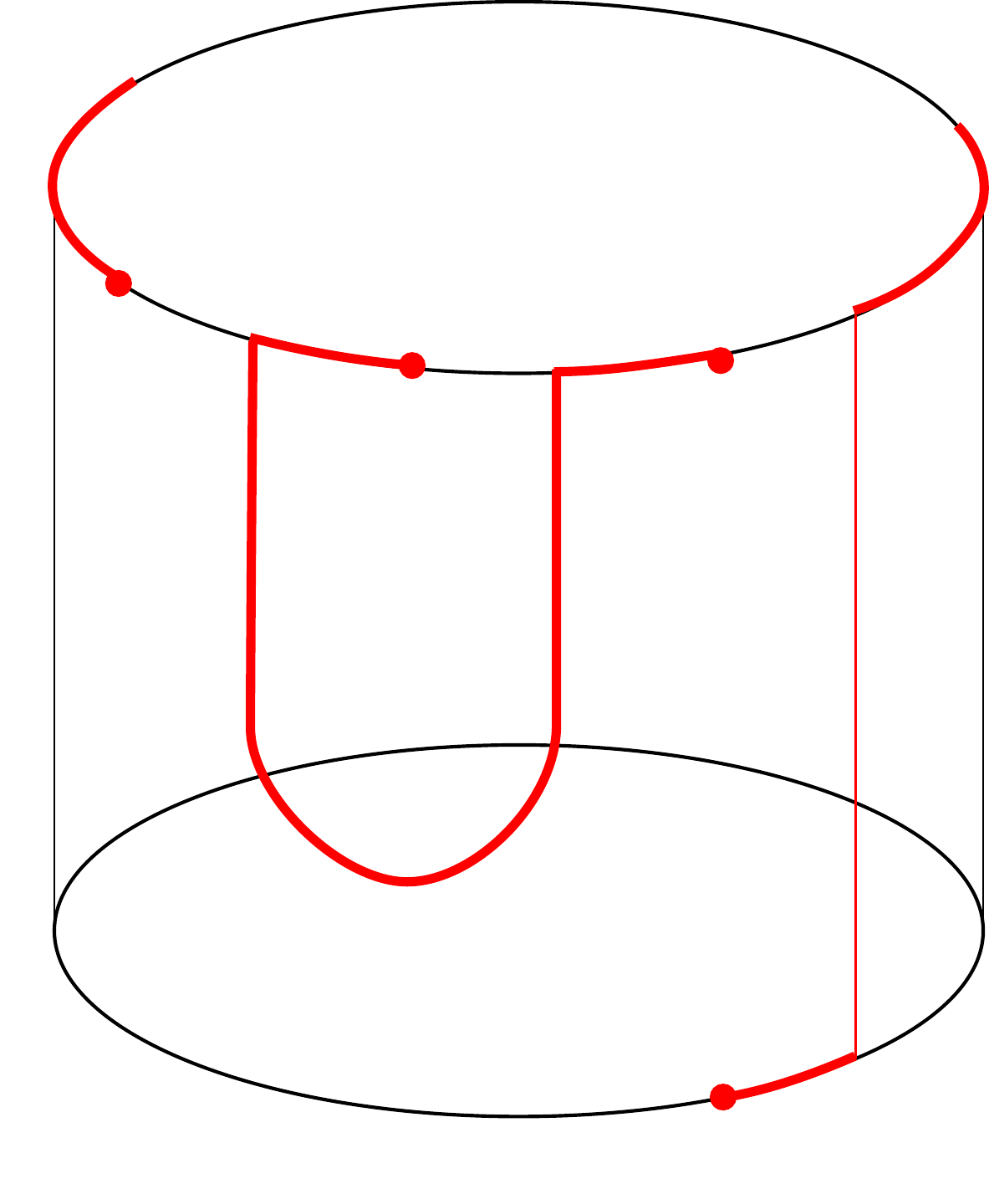}}%
    \put(0.64953066,0.76065379){\color[rgb]{0,0,0}\makebox(0,0)[lb]{\smash{$2i-1$}}}%
    \put(0.3652226,0.75813781){\color[rgb]{0,0,0}\makebox(0,0)[lb]{\smash{$2i$}}}%
    \put(0.62940267,0.00836954){\color[rgb]{0,0,0}\makebox(0,0)[lb]{\smash{$2i-1$}}}%
    \put(-0.00211352,0.83110185){\color[rgb]{0,0,0}\makebox(0,0)[lb]{\smash{$2i+1$}}}%
  \end{picture}%
\endgroup%

\quad
\def\svgwidth{120pt}
\begingroup%
  \makeatletter%
  \providecommand\color[2][]{%
    \errmessage{(Inkscape) Color is used for the text in Inkscape, but the package 'color.sty' is not loaded}%
    \renewcommand\color[2][]{}%
  }%
  \providecommand\transparent[1]{%
    \errmessage{(Inkscape) Transparency is used (non-zero) for the text in Inkscape, but the package 'transparent.sty' is not loaded}%
    \renewcommand\transparent[1]{}%
  }%
  \providecommand\rotatebox[2]{#2}%
  \ifx\svgwidth\undefined%
    \setlength{\unitlength}{340.65934272bp}%
    \ifx\svgscale\undefined%
      \relax%
    \else%
      \setlength{\unitlength}{\unitlength * \real{\svgscale}}%
    \fi%
  \else%
    \setlength{\unitlength}{\svgwidth}%
  \fi%
  \global\let\svgwidth\undefined%
  \global\let\svgscale\undefined%
  \makeatother%
  \begin{picture}(1,1.21706706)%
    \put(0,0){\includegraphics[width=\unitlength]{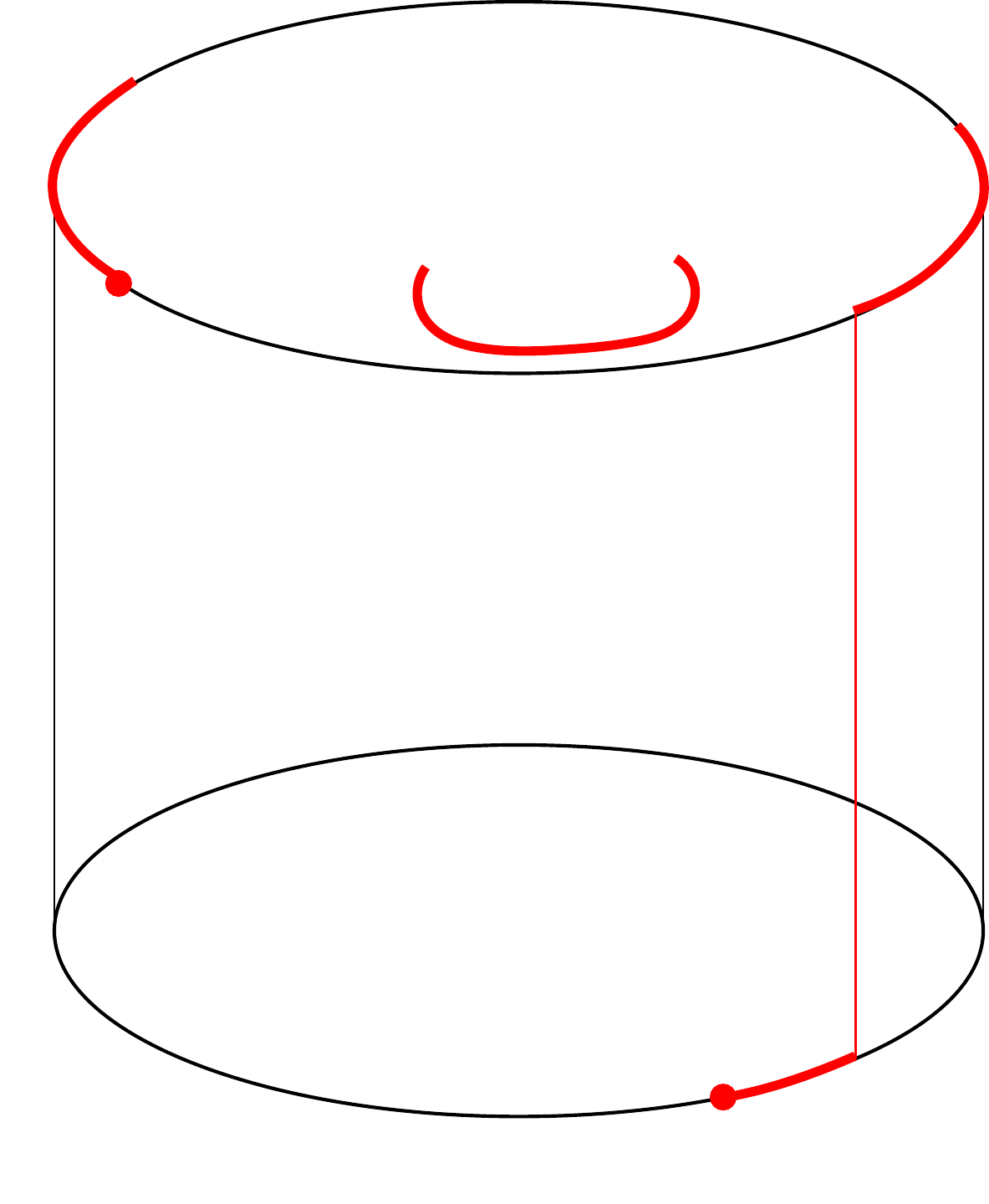}}%
    \put(0.62940267,0.00836954){\color[rgb]{0,0,0}\makebox(0,0)[lb]{\smash{$2i-1$}}}%
    \put(-0.00211352,0.83110185){\color[rgb]{0,0,0}\makebox(0,0)[lb]{\smash{$2i-1$}}}%
  \end{picture}%
\endgroup%

\end{center}

So these creation and annihilation operators, inserting and closing off chords in a chord diagram, are related in a way that is quite similar to an adjoint.

\subsection{``Inner product'' on chord diagrams}

In fact, based on the above, we will define an  ``inner product'' of two chord diagrams $\langle \Gamma_0 | \Gamma_1 \rangle$ via what happens when you insert those two diagrams into the cylinder and round the corners. The result must be a collection of curves on the sphere; it may be one connected curve, or it may be disconnected and consist of several connected curve components. 

We set
\[
\left\langle \Gamma_0 | \Gamma_1 \right\rangle = \left\{ \begin{array}{ll}
1 & \text{if the result of rounding the curves on the cylinder} \\ & \quad \text{is a single connected curve on the sphere;} \\
0 & \text{if the result is disconnected.}
\end{array} \right.
\]

Note the asymmetry: we require that $\Gamma_0$ go on the bottom, and $\Gamma_1$ on the top.

The example drawn above in section \ref{sec:curves_on_cylinders} shows that
\[
\langle \Gamma_{pq} | \Gamma_{qp} \rangle = 1.
\]
The finger move we drew above shows that
\[
\langle a_{q,i}^\dagger \Gamma_0 | \Gamma_1 \rangle = \langle \Gamma_0 | a_{q,i} \Gamma_1 \rangle.
\]
Similar finger moves show that annihilation and creation operators on chord diagrams are adjoint --- just as on chessboards.
\begin{align*}
\langle a_{p,i} \Gamma_0 | \Gamma_1 \rangle &= \langle \Gamma_0 | a_{p,i}^* \Gamma_1 \rangle \\
\langle a_{q,i+1} \Gamma_0 | \Gamma_1 \rangle &= \langle \Gamma_0 | a_{q,i}^\dagger \Gamma_1 \rangle \\
\langle a_{p,i}^* \Gamma_0 | \Gamma_1 \rangle &= \langle \Gamma_0 | a_{p,i+1} \Gamma_1 \rangle.
\end{align*}

One can prove that this bilinear form is isomorphic to the one defined on chessboards.
\begin{thm}[\cite{Me09Paper, Me10_Sutured_TQFT}]
For any two chessboards/words $w_0, w_1$,
\[
\langle w_0 | w_1 \rangle = \langle \Gamma_{w_0} | \Gamma_{w_1} \rangle.
\]
\end{thm}

\[
\begin{tikzpicture}
\draw (0,0) node {$(w_0, w_1)$};
\draw (5,0) node {$(\Gamma_{w_0}, \Gamma_{w_1})$};
\draw (0,-3) node {$\langle w_0 | w_1 \rangle$};
\draw (5,-3) node {$\langle \Gamma_{w_0} | \Gamma_{w_1} \rangle$};
\draw [->,decorate, decoration={snake,amplitude=.4mm,segment length=2mm,post length=1mm}] (1,0) -- (4,0)
	node [above, align=center, midway] {Draw\\diagrams};
\draw [<->]
(1,-3) -- (4,-3)
	node [above, align=center, midway] {$=$};
\draw [->] (0,-0.4) -- (0,-2.6)
	node [left, align=center, midway] {``Pawn dynamics''\\``inner product''\\bilinear form};
\draw [->] (5,-0.4) -- (5,-2.6) node [right, align=center, midway] {Insert into cylinder\\``inner product''\\bilinear form};
\end{tikzpicture}
\]

That is, the property of pawns on a chessboard being able to move from one setup to another, is precisely the property of curves on the cylinder joining up to give a single connected curve.

Having seen this, it is perhaps more plausible why repeated adjoints of chessboard operators might be periodic.

\subsection{Bypass discs and $120^\circ$ rotations}

Suppose we have a chord diagram $\Gamma$ on a disc $D$, and we consider a sub-disc $B \subset D$ on which the chord diagram appears as shown.
\[
\begin{tikzpicture}[
scale=0.9, 
suture/.style={thick, draw=red},
boundary/.style={ultra thick},
vertex/.style={draw=red, fill=red}]
\coordinate (0) at (90:1);
\coordinate (1) at (30:1);
\coordinate (2) at (-30:1);
\coordinate (3) at (-90:1);
\coordinate (4) at (-150:1);
\coordinate (5) at (150:1);
\filldraw[fill=black!10!white, draw=none] (0) arc (90:30:1) to [bend right=90] (2) arc (-30:-90:1) -- cycle;
\filldraw[fill=black!10!white, draw=none] (4) to [bend right=90] (5) arc (150:210:1);
\draw [boundary] (0,0) circle (1 cm);
\draw [suture] (1) to [bend right=90] (2);
\draw [suture] (0) -- (3);
\draw [suture] (4) to [bend right=90] (5);
\end{tikzpicture}
\]
Such a disc, on which the chord diagram restricts to $3$ parallel arcs, is called a \emph{bypass disc}. There are two natural ways to adjust this chord diagram, consistent with the colours, which amount to $120^\circ$ rotations. These operations are called \emph{bypass surgeries}.
\[
\begin{tikzpicture}[
scale=0.9, 
suture/.style={thick, draw=red},
boundary/.style={ultra thick},
vertex/.style={draw=red, fill=red}]
\coordinate (0a) at ($ (90:1) + (-4,0) $);
\coordinate (1a) at ($ (30:1) + (-4,0) $);
\coordinate (2a) at ($ (-30:1) + (-4,0) $);
\coordinate (3a) at ($ (-90:1) + (-4,0) $);
\coordinate (4a) at ($ (-150:1) + (-4,0) $);
\coordinate (5a) at ($ (150:1) + (-4,0) $);
\filldraw[fill=black!10!white, draw=none] (4a) arc (210:150:1) to [bend right=90] (0a) arc (90:30:1) -- cycle;
\filldraw[fill=black!10!white, draw=none] (2a) to [bend right=90] (3a) arc (-90:-30:1);
\draw [boundary] (-4,0) circle (1 cm);
\draw [suture] (5a) to [bend right=90] (0a);
\draw [suture] (4a) -- (1a);
\draw [suture] (2a) to [bend right=90] (3a);
\draw (-4,-2) node {$\Gamma'$};
\coordinate (0) at (90:1);
\coordinate (1) at (30:1);
\coordinate (2) at (-30:1);
\coordinate (3) at (-90:1);
\coordinate (4) at (-150:1);
\coordinate (5) at (150:1);
\filldraw[fill=black!10!white, draw=none] (0) arc (90:30:1) to [bend right=90] (2) arc (-30:-90:1) -- cycle;
\filldraw[fill=black!10!white, draw=none] (4) to [bend right=90] (5) arc (150:210:1);
\draw [boundary] (0,0) circle (1 cm);
\draw [suture] (1) to [bend right=90] (2);
\draw [suture] (0) -- (3);
\draw [suture] (4) to [bend right=90] (5);
\draw (0,-2) node {$\Gamma$};
\coordinate (0b) at ($ (90:1) + (4,0) $);
\coordinate (1b) at ($ (30:1) + (4,0) $);
\coordinate (2b) at ($ (-30:1) + (4,0) $);
\coordinate (3b) at ($ (-90:1) + (4,0) $);
\coordinate (4b) at ($ (-150:1) + (4,0) $);
\coordinate (5b) at ($ (150:1) + (4,0) $);
\filldraw[fill=black!10!white, draw=none] (2b) arc (-30:-90:1) to [bend right=90] (4b) arc (210:150:1) -- cycle;
\filldraw[fill=black!10!white, draw=none] (0b) to [bend right=90] (1b) arc (30:90:1);
\draw [boundary] (4,0) circle (1 cm);
\draw [suture] (3b) to [bend right=90] (4b);
\draw [suture] (2b) -- (5b);
\draw [suture] (0b) to [bend right=90] (1b);
\draw (4,-2) node {$\Gamma''$};
\end{tikzpicture}
\]
With these surgeries on $B \subset D$, we have three chord diagrams on $D$, which we denote $\Gamma, \Gamma', \Gamma''$. We consider inserting $\Gamma, \Gamma', \Gamma''$ successively into one end of a cylinder, with some other chord diagram $\Gamma_1$ on the other end. We obtain three sets of curves on the cylinder.

Below is drawn a possible arrangment of curves on the cylinder. We draw what happens inside $B$ inside a circle, and what happens outside $B$ outside the circle --- this does not necessarily correspond to the top and bottom of the cylinder.
\[
\begin{tikzpicture}[
scale=0.9, 
suture/.style={thick, draw=red},
boundary/.style={ultra thick},
vertex/.style={draw=red, fill=red}]
\coordinate (0a) at ($ (90:1) + (-4,0) $);
\coordinate (1a) at ($ (30:1) + (-4,0) $);
\coordinate (2a) at ($ (-30:1) + (-4,0) $);
\coordinate (3a) at ($ (-90:1) + (-4,0) $);
\coordinate (4a) at ($ (-150:1) + (-4,0) $);
\coordinate (5a) at ($ (150:1) + (-4,0) $);
\filldraw[fill=black!10!white, draw=none] (4a) arc (210:150:1) to [bend right=90] (0a) arc (90:30:1) -- cycle;
\filldraw[fill=black!10!white, draw=none] (2a) to [bend right=90] (3a) arc (-90:-30:1);
\draw [boundary] (-4,0) circle (1 cm);
\draw [suture] (5a) to [bend right=90] (0a);
\draw [suture] (4a) -- (1a);
\draw [suture] (2a) to [bend right=90] (3a);
\draw [suture] (0a) .. controls ($ (-4,0) + (90:2) $) and ($ (-4,0) + (1.8,1) $) .. ($ (-4,0) + (1.8,0) $) .. controls ($ (-4,0) + (1.8,-1) $) and ($ (-4,0) + (-90:2) $) .. (3a);
\draw [suture] (1a) .. controls ($ (1a) + (30:1) $) and ($ (2a) + (-30:1) $) .. (2a);
\draw [suture] (4a) .. controls ($ (4a) + (-150:1) $) and ($ (5a) + (150:1) $) .. (5a);
\draw (-4,-2) node {$\langle \Gamma' | \Gamma_1 \rangle$};
\coordinate (0) at (90:1);
\coordinate (1) at (30:1);
\coordinate (2) at (-30:1);
\coordinate (3) at (-90:1);
\coordinate (4) at (-150:1);
\coordinate (5) at (150:1);
\filldraw[fill=black!10!white, draw=none] (0) arc (90:30:1) to [bend right=90] (2) arc (-30:-90:1) -- cycle;
\filldraw[fill=black!10!white, draw=none] (4) to [bend right=90] (5) arc (150:210:1);
\draw [boundary] (0,0) circle (1 cm);
\draw [suture] (1) to [bend right=90] (2);
\draw [suture] (0) -- (3);
\draw [suture] (4) to [bend right=90] (5);
\draw [suture] (0) .. controls ($ (0,0) + (90:2) $) and ($ (0,0) + (1.8,1) $) .. ($ (0,0) + (1.8,0) $) .. controls ($ (0,0) + (1.8,-1) $) and ($ (0,0) + (-90:2) $) .. (3);
\draw [suture] (1) .. controls ($ (1) + (30:1) $) and ($ (2) + (-30:1) $) .. (2);
\draw [suture] (4) .. controls ($ (4) + (-150:1) $) and ($ (5) + (150:1) $) .. (5);
\draw (0,-2) node {$\langle \Gamma | \Gamma_1 \rangle$};
\coordinate (0b) at ($ (90:1) + (4,0) $);
\coordinate (1b) at ($ (30:1) + (4,0) $);
\coordinate (2b) at ($ (-30:1) + (4,0) $);
\coordinate (3b) at ($ (-90:1) + (4,0) $);
\coordinate (4b) at ($ (-150:1) + (4,0) $);
\coordinate (5b) at ($ (150:1) + (4,0) $);
\filldraw[fill=black!10!white, draw=none] (2b) arc (-30:-90:1) to [bend right=90] (4b) arc (210:150:1) -- cycle;
\filldraw[fill=black!10!white, draw=none] (0b) to [bend right=90] (1b) arc (30:90:1);
\draw [boundary] (4,0) circle (1 cm);
\draw [suture] (3b) to [bend right=90] (4b);
\draw [suture] (2b) -- (5b);
\draw [suture] (0b) to [bend right=90] (1b);
\draw [suture] (0b) .. controls ($ (4,0) + (90:2) $) and ($ (4,0) + (1.8,1) $) .. ($ (4,0) + (1.8,0) $) .. controls ($ (4,0) + (1.8,-1) $) and ($ (4,0) + (-90:2) $) .. (3b);
\draw [suture] (1b) .. controls ($ (1b) + (30:1) $) and ($ (2b) + (-30:1) $) .. (2b);
\draw [suture] (4b) .. controls ($ (4b) + (-150:1) $) and ($ (5b) + (150:1) $) .. (5b);
\draw (4,-2) node {$\langle \Gamma'' | \Gamma_1 \rangle$};
\end{tikzpicture}
\]
We find here that, of the three sets of curves on the cylinder, $2$ of them are connected. The key observation is that \emph{whatever chord diagrams we have for $\Gamma$ and $\Gamma_1$, out of the $3$ sets of curves obtained on the cylinder, the number of sets of curves which is connected is even.}

\begin{prop}
Let $\Gamma, \Gamma_1$ be chord diagrams on the disc $D$ with the same number of chords. Then, with $\Gamma', \Gamma''$ obtained from $\Gamma$ as above,
\[
\langle \Gamma | \Gamma_1 \rangle  +  \langle \Gamma' | \Gamma_1 \rangle + \langle \Gamma'' | \Gamma_1 \rangle = 0.
\]
In other words,
\[
\langle \quad \Gamma + \Gamma' + \Gamma'' \quad | \quad \Gamma_1 \quad  \rangle = 0.
\]
\end{prop}

\begin{proof}
The arrangement is either as in the above diagram, giving $1+0+1 = 0$ mod $2$, or more degenerate, giving $0+0+0=0$. See below.
\[
\begin{tikzpicture}[
scale=0.9, 
suture/.style={thick, draw=red},
boundary/.style={ultra thick},
vertex/.style={draw=red, fill=red}]
\coordinate (0a) at ($ (90:1) + (-4,0) $);
\coordinate (1a) at ($ (30:1) + (-4,0) $);
\coordinate (2a) at ($ (-30:1) + (-4,0) $);
\coordinate (3a) at ($ (-90:1) + (-4,0) $);
\coordinate (4a) at ($ (-150:1) + (-4,0) $);
\coordinate (5a) at ($ (150:1) + (-4,0) $);
\filldraw[fill=black!10!white, draw=none] (4a) arc (210:150:1) to [bend right=90] (0a) arc (90:30:1) -- cycle;
\filldraw[fill=black!10!white, draw=none] (2a) to [bend right=90] (3a) arc (-90:-30:1);
\draw [boundary] (-4,0) circle (1 cm);
\draw [suture] (5a) to [bend right=90] (0a);
\draw [suture] (4a) -- (1a);
\draw [suture] (2a) to [bend right=90] (3a);
\draw [suture] (0a) .. controls ($ (-4,0) + (90:2) $) and ($ (-4,0) + (1.8,1) $) .. ($ (-4,0) + (1.8,0) $) .. controls ($ (-4,0) + (1.8,-1) $) and ($ (-4,0) + (-90:2) $) .. (3a);
\draw [suture] (1a) .. controls ($ (1a) + (30:1) $) and ($ (2a) + (-30:1) $) .. (2a);
\draw [suture] (4a) .. controls ($ (4a) + (-150:1) $) and ($ (5a) + (150:1) $) .. (5a);
\draw (-4,-2) node {$1$};
\draw (-2,-2) node {$+$};
\coordinate (0) at (90:1);
\coordinate (1) at (30:1);
\coordinate (2) at (-30:1);
\coordinate (3) at (-90:1);
\coordinate (4) at (-150:1);
\coordinate (5) at (150:1);
\filldraw[fill=black!10!white, draw=none] (0) arc (90:30:1) to [bend right=90] (2) arc (-30:-90:1) -- cycle;
\filldraw[fill=black!10!white, draw=none] (4) to [bend right=90] (5) arc (150:210:1);
\draw [boundary] (0,0) circle (1 cm);
\draw [suture] (1) to [bend right=90] (2);
\draw [suture] (0) -- (3);
\draw [suture] (4) to [bend right=90] (5);
\draw [suture] (0) .. controls ($ (0,0) + (90:2) $) and ($ (0,0) + (1.8,1) $) .. ($ (0,0) + (1.8,0) $) .. controls ($ (0,0) + (1.8,-1) $) and ($ (0,0) + (-90:2) $) .. (3);
\draw [suture] (1) .. controls ($ (1) + (30:1) $) and ($ (2) + (-30:1) $) .. (2);
\draw [suture] (4) .. controls ($ (4) + (-150:1) $) and ($ (5) + (150:1) $) .. (5);
\draw (0,-2) node {$0$};
\draw (2,-2) node {$+$};
\coordinate (0b) at ($ (90:1) + (4,0) $);
\coordinate (1b) at ($ (30:1) + (4,0) $);
\coordinate (2b) at ($ (-30:1) + (4,0) $);
\coordinate (3b) at ($ (-90:1) + (4,0) $);
\coordinate (4b) at ($ (-150:1) + (4,0) $);
\coordinate (5b) at ($ (150:1) + (4,0) $);
\filldraw[fill=black!10!white, draw=none] (2b) arc (-30:-90:1) to [bend right=90] (4b) arc (210:150:1) -- cycle;
\filldraw[fill=black!10!white, draw=none] (0b) to [bend right=90] (1b) arc (30:90:1);
\draw [boundary] (4,0) circle (1 cm);
\draw [suture] (3b) to [bend right=90] (4b);
\draw [suture] (2b) -- (5b);
\draw [suture] (0b) to [bend right=90] (1b);
\draw [suture] (0b) .. controls ($ (4,0) + (90:2) $) and ($ (4,0) + (1.8,1) $) .. ($ (4,0) + (1.8,0) $) .. controls ($ (4,0) + (1.8,-1) $) and ($ (4,0) + (-90:2) $) .. (3b);
\draw [suture] (1b) .. controls ($ (1b) + (30:1) $) and ($ (2b) + (-30:1) $) .. (2b);
\draw [suture] (4b) .. controls ($ (4b) + (-150:1) $) and ($ (5b) + (150:1) $) .. (5b);
\draw (4,-2) node {$1$};
\draw (6,-2) node {$=$};
\draw (8,-2) node {$0$};
\end{tikzpicture}
\]
\[
\begin{tikzpicture}[
scale=0.9, 
suture/.style={thick, draw=red},
boundary/.style={ultra thick},
vertex/.style={draw=red, fill=red}]
\coordinate (0a) at ($ (90:1) + (-4,0) $);
\coordinate (1a) at ($ (30:1) + (-4,0) $);
\coordinate (2a) at ($ (-30:1) + (-4,0) $);
\coordinate (3a) at ($ (-90:1) + (-4,0) $);
\coordinate (4a) at ($ (-150:1) + (-4,0) $);
\coordinate (5a) at ($ (150:1) + (-4,0) $);
\filldraw[fill=black!10!white, draw=none] (4a) arc (210:150:1) to [bend right=90] (0a) arc (90:30:1) -- cycle;
\filldraw[fill=black!10!white, draw=none] (2a) to [bend right=90] (3a) arc (-90:-30:1);
\draw [boundary] (-4,0) circle (1 cm);
\draw [suture] (5a) to [bend right=90] (0a);
\draw [suture] (4a) -- (1a);
\draw [suture] (2a) to [bend right=90] (3a);
\draw [suture] (0a) .. controls ($ (0a) + (90:1) $) and ($ (1a) + (30:1) $) .. (1a);
\draw [suture] (2a) .. controls ($ (2a) + (-30:1) $) and ($ (3a) + (-90:1) $) .. (3a);
\draw [suture] (4a) .. controls ($ (4a) + (-150:1) $) and ($ (5a) + (150:1) $) .. (5a);
\draw (-4,-2) node {$0$};
\draw (-2,-2) node {$+$};
\coordinate (0) at (90:1);
\coordinate (1) at (30:1);
\coordinate (2) at (-30:1);
\coordinate (3) at (-90:1);
\coordinate (4) at (-150:1);
\coordinate (5) at (150:1);
\filldraw[fill=black!10!white, draw=none] (0) arc (90:30:1) to [bend right=90] (2) arc (-30:-90:1) -- cycle;
\filldraw[fill=black!10!white, draw=none] (4) to [bend right=90] (5) arc (150:210:1);
\draw [boundary] (0,0) circle (1 cm);
\draw [suture] (1) to [bend right=90] (2);
\draw [suture] (0) -- (3);
\draw [suture] (4) to [bend right=90] (5);
\draw [suture] (0) .. controls ($ (0) + (90:1) $) and ($ (1) + (30:1) $) .. (1);
\draw [suture] (2) .. controls ($ (2) + (-30:1) $) and ($ (3) + (-90:1) $) .. (3);
\draw [suture] (4) .. controls ($ (4) + (-150:1) $) and ($ (5) + (150:1) $) .. (5);
\draw (0,-2) node {$0$};
\draw (2,-2) node {$+$};
\coordinate (0b) at ($ (90:1) + (4,0) $);
\coordinate (1b) at ($ (30:1) + (4,0) $);
\coordinate (2b) at ($ (-30:1) + (4,0) $);
\coordinate (3b) at ($ (-90:1) + (4,0) $);
\coordinate (4b) at ($ (-150:1) + (4,0) $);
\coordinate (5b) at ($ (150:1) + (4,0) $);
\filldraw[fill=black!10!white, draw=none] (2b) arc (-30:-90:1) to [bend right=90] (4b) arc (210:150:1) -- cycle;
\filldraw[fill=black!10!white, draw=none] (0b) to [bend right=90] (1b) arc (30:90:1);
\draw [boundary] (4,0) circle (1 cm);
\draw [suture] (3b) to [bend right=90] (4b);
\draw [suture] (2b) -- (5b);
\draw [suture] (0b) to [bend right=90] (1b);
\draw [suture] (0b) .. controls ($ (0b) + (90:1) $) and ($ (1b) + (30:1) $) .. (1b);
\draw [suture] (2b) .. controls ($ (2b) + (-30:1) $) and ($ (3b) + (-90:1) $) .. (3b);
\draw [suture] (4b) .. controls ($ (4b) + (-150:1) $) and ($ (5b) + (150:1) $) .. (5b);
\draw (4,-2) node {$0$};
\draw (6,-2) node {$=$};
\draw (8,-2) node {$0$};
\end{tikzpicture}
\]
\end{proof}

\subsection{A vector space of chord diagrams}

We have seem how some chord diagrams can be drawn from chessboards. But this is a very small class of chord diagrams; there are many more chord diagrams than chessboard / ski slope chord diagrams.

On the other hand, our creation and annihilation operators are defined on any chord diagram, not just chessboard ones. It's s easy to draw in extra chords, or close off chords, on any chord diagram, not just those drawn from chessboards. The adjoint relations from finger moves also work generally, and the bilinear form / ``inner product'' defined from insertion into a cylinder also works for any chord diagram.

We can relate general chord diagrams, to chessboard chord diagrams, with the following observations.

If our ``inner product'' is supposed to be nondegenerate, then based on our previous theorem, we should set
\[
\Gamma + \Gamma' + \Gamma'' = 0
\]
for any triple of chord diagrams $\Gamma, \Gamma', \Gamma''$ obtained by byass surgeries.

This relation is called the \emph{bypass relation} and can be written schematically as

\[
\begin{tikzpicture}[
scale=0.9, 
suture/.style={thick, draw=red},
boundary/.style={ultra thick},
vertex/.style={draw=red, fill=red}]
\coordinate (0a) at ($ (90:1) + (-4,0) $);
\coordinate (1a) at ($ (30:1) + (-4,0) $);
\coordinate (2a) at ($ (-30:1) + (-4,0) $);
\coordinate (3a) at ($ (-90:1) + (-4,0) $);
\coordinate (4a) at ($ (-150:1) + (-4,0) $);
\coordinate (5a) at ($ (150:1) + (-4,0) $);
\filldraw[fill=black!10!white, draw=none] (4a) arc (210:150:1) to [bend right=90] (0a) arc (90:30:1) -- cycle;
\filldraw[fill=black!10!white, draw=none] (2a) to [bend right=90] (3a) arc (-90:-30:1);
\draw [boundary] (-4,0) circle (1 cm);
\draw [suture] (5a) to [bend right=90] (0a);
\draw [suture] (4a) -- (1a);
\draw [suture] (2a) to [bend right=90] (3a);
\draw (-2,0) node {$+$};
\coordinate (0) at (90:1);
\coordinate (1) at (30:1);
\coordinate (2) at (-30:1);
\coordinate (3) at (-90:1);
\coordinate (4) at (-150:1);
\coordinate (5) at (150:1);
\filldraw[fill=black!10!white, draw=none] (0) arc (90:30:1) to [bend right=90] (2) arc (-30:-90:1) -- cycle;
\filldraw[fill=black!10!white, draw=none] (4) to [bend right=90] (5) arc (150:210:1);
\draw [boundary] (0,0) circle (1 cm);
\draw [suture] (1) to [bend right=90] (2);
\draw [suture] (0) -- (3);
\draw [suture] (4) to [bend right=90] (5);
\draw (2,0) node {$+$};
\coordinate (0b) at ($ (90:1) + (4,0) $);
\coordinate (1b) at ($ (30:1) + (4,0) $);
\coordinate (2b) at ($ (-30:1) + (4,0) $);
\coordinate (3b) at ($ (-90:1) + (4,0) $);
\coordinate (4b) at ($ (-150:1) + (4,0) $);
\coordinate (5b) at ($ (150:1) + (4,0) $);
\filldraw[fill=black!10!white, draw=none] (2b) arc (-30:-90:1) to [bend right=90] (4b) arc (210:150:1) -- cycle;
\filldraw[fill=black!10!white, draw=none] (0b) to [bend right=90] (1b) arc (30:90:1);
\draw [boundary] (4,0) circle (1 cm);
\draw [suture] (3b) to [bend right=90] (4b);
\draw [suture] (2b) -- (5b);
\draw [suture] (0b) to [bend right=90] (1b);
\draw (6,0) node {$=$};
\draw (8,0) node {$0$};
\end{tikzpicture}
\]

We therefore define a vector space $V_n$ as follows: $V_n$ is generated over $\Z_2$ by all chord diagrams with $n$ chords, subject to the bypass relation, i.e. the relation that any three chord diagrams related by bypass surgeries sum to zero.
\[
V_n = \frac{ \Z_2 \langle \text{Chord diagrams with $n$ chords} \rangle }{ \text{Bypass relation} }
\]
Before taking a quotient, the vector space is freely generated by chord diagrams, and its dimension is equal to the number of chord diagrams with $n$ chords, which is $C_n$, the $n$'th Catalan number.

After the quotient, something interesting happens. The dimension is reduced to $2^{n-1}$ and a basis is rather familiar.
\begin{thm}[\cite{Me09Paper}]
The $\Z_2$ vector space $V_n$ has dimension $2^{n-1}$ and the diagrams from chessboards of $n-1$ squares form a basis.
\end{thm}

Indeed, there are $2^{n-1}$ configurations of pawns on a chessboard with $n-1$ squares, and these give the chessboard chord diagrams with $n$ chords. (Each square, again, contributing one bit of information.) Alternatively, this is the $\Z_2$ vector space freely generated by words in $p$ and $q$.

\begin{cor}
\[
V_n \cong \left( p \Z_2 \oplus q \Z_2 \right)^{\otimes (n-1)}.
\]
\end{cor}

There is much more structure in this vector space, encoding various combinatorial and representation-theoretic properties of chord diagrams. For instance, $V_n$ has $2^{2^{n-1}}$ elements; the $C_n$ of these which correspond to chord diagrams are distributed in a combinatorially interesting way.

To give an idea of how chord diagrams decompose into chessboard diagrams, we give a computation below.

\[
\begin{tikzpicture}[
scale=0.5, 
suture/.style={thick, draw=red},
boundary/.style={ultra thick},
vertex/.style={draw=red, fill=red},
filling/.style={fill=black!10!white, draw=none},
a/.style={xshift=6 cm},
b/.style={xshift=12 cm},
c/.style={xshift=6cm, yshift=-6 cm},
d/.style={xshift=12 cm, yshift=-6 cm},
e/.style={xshift=18 cm, yshift=-6 cm},
f/.style={xshift=24 cm, yshift=-6 cm}]
\coordinate (0) at (2,0);
\coordinate (1) at (4,-1);
\coordinate (-1) at (0,-1);
\coordinate (2) at (4,-2);
\coordinate (-2) at (0,-2);
\coordinate (3) at (4,-3);
\coordinate (-3) at (0,-3);
\coordinate (4) at (4,-4);
\coordinate (-4) at (0,-4);
\coordinate (5) at (2,-5);
\filldraw[filling] (0) -- (4,0) -- (1) arc (90:270:0.5) -- (3) to [bend left=15] (0);
\filldraw[filling] (-1) .. controls (0.5,-1) and (3.5, -4) .. (4) -- (4,-5) -- (5) to [bend right=15] (-2) -- cycle;
\filldraw[filling] (-3) -- (-4) arc (-90:90:0.5);
\draw [suture] (0) to [bend right=15] (3);
\draw [suture] (1) arc (90:270:0.5);
\draw [suture] (-1) .. controls (0.5,-1) and (3.5,-4) .. (4);
\draw [suture] (5) to [bend right=15] (-2);
\draw [suture] (-3) arc (90:-90:0.5);
\draw (0.5,-0.5) -- (3.75,-0.5) -- (3.75, -1.5) -- (0.5,-1.5) -- cycle;
\draw [boundary] (0,0) -- (4,0) -- (4,-5) -- (0,-5) -- cycle;
\foreach \point in {0, 1, 2, 3, 4, 5, -1,-2,-3,-4}
\fill [vertex] (\point) circle (2pt);
\draw (5,-2.5) node {$=$};
\coordinate (0a) at (8,0);
\coordinate (1a) at (10,-1);
\coordinate (-1a) at (6,-1);
\coordinate (2a) at (10,-2);
\coordinate (-2a) at (6,-2);
\coordinate (3a) at (10,-3);
\coordinate (-3a) at (6,-3);
\coordinate (4a) at (10,-4);
\coordinate (-4a) at (6,-4);
\coordinate (5a) at (8,-5);
\filldraw [filling, a] (0a) -- (4,0) -- (1a) arc (90:270:1.5) -- (4,-5) -- (5a) to [bend right=15] (-2a) -- (-1a) to [bend right=15] (0a);
\filldraw [filling, a] (2a) arc (90:270:0.5) -- cycle;
\filldraw [filling, a] (-3a) arc (90:-90:0.5) -- cycle;
\draw [suture, a] (0a) to [bend left=15] (-1a);
\draw [suture, a] (1a) arc (90:270:1.5);
\draw [suture, a] (2a) arc (90:270:0.5);
\draw [suture, a] (-2a) to [bend left=15] (5a);
\draw [suture, a] (-3a) arc (90:-90:0.5);
\draw [a] (0.25,-2.5) -- (3,-2.5) -- (3,-3.5) -- (0.25,-3.5) -- cycle;
\draw [boundary, a] (0,0) -- (4,0) -- (4,-5) -- (0,-5) -- cycle;
\foreach \point in {0a, 1a, 2a, 3a, 4a, 5a, -1a,-2a,-3a,-4a}
\fill [vertex] (\point) circle (2pt);
\draw (11,-2.5) node {$+$};
\coordinate (0b) at (14,0);
\coordinate (1b) at (16,-1);
\coordinate (-1b) at (12,-1);
\coordinate (2b) at (16,-2);
\coordinate (-2b) at (12,-2);
\coordinate (3b) at (16,-3);
\coordinate (-3b) at (12,-3);
\coordinate (4b) at (16,-4);
\coordinate (-4b) at (12,-4);
\coordinate (5b) at (14,-5);
\filldraw [filling, b] (0b) -- (4,0) -- (1b) to [bend left=15] (0b);
\filldraw [filling, b] (-1b) .. controls (0.5,-1) and (3.5,-2) .. (2b) -- (3b) arc (90:270:0.5) -- (4,-5) -- (5b) to [bend right=15] (-2b) -- cycle;
\filldraw [filling, b] (-3b) arc (90:-90:0.5) -- cycle;
\draw [suture, b] (0b) to [bend right=15] (1b);
\draw [suture, b] (-1b) .. controls (0.5,-1) and (3.5,-2) .. (2b);
\draw [suture, b] (-2b) to [bend left=15] (5b);
\draw [suture, b] (3b) arc (90:270:0.5);
\draw [suture, b] (-3b) arc (90:-90:0.5);
\draw [b] (0.25,-3.25) -- (3.75,-3.25) -- (3.75,-3.75) -- (0.25,-3.75) -- cycle;
\draw [boundary, b] (0,0) -- (4,0) -- (4,-5) -- (0,-5) -- cycle;
\foreach \point in {0b, 1b, 2b, 3b, 4b, 5b, -1b,-2b,-3b,-4b}
\fill [vertex] (\point) circle (2pt);
\draw (5,-8.5) node {$=$};
\coordinate (0c) at (8,-6);
\coordinate (1c) at (10,-7);
\coordinate (-1c) at (6,-7);
\coordinate (2c) at (10,-8);
\coordinate (-2c) at (6,-8);
\coordinate (3c) at (10,-9);
\coordinate (-3c) at (6,-9);
\coordinate (4c) at (10,-10);
\coordinate (-4c) at (6,-10);
\coordinate (5c) at (8,-11);
\filldraw [filling, c] (0c) -- (4,0) -- (1c) .. controls (3.5,-1) and (0.5,-4) .. (-4c) -- (-3c) arc (-90:90:0.5) -- (-1c) to [bend right=15] (0c);
\filldraw [filling, c] (2c) arc (90:270:0.5) -- cycle;
\filldraw [filling, c] (4c) -- (4,-5) -- (5c) to [bend left=15] (4c);
\draw [suture, c] (0c) to [bend left=15] (-1c);
\draw [suture, c] (2c) arc (90:270:0.5);
\draw [suture, c] (-2c) arc (90:-90:0.5);
\draw [suture, c] (1c) .. controls (3.5,-1) and (0.5,-4) .. (-4c);
\draw [suture, c] (4c) to [bend right=15] (5c);
\draw [boundary, c] (0,0) -- (4,0) -- (4,-5) -- (0,-5) -- cycle;
\foreach \point in {0c, 1c, 2c, 3c, 4c, 5c, -1c,-2c,-3c,-4c}
\fill [vertex] (\point) circle (2pt);
\draw (11,-8.5) node {$+$};
\coordinate (0d) at (14,-6);
\coordinate (1d) at (16,-7);
\coordinate (-1d) at (12,-7);
\coordinate (2d) at (16,-8);
\coordinate (-2d) at (12,-8);
\coordinate (3d) at (16,-9);
\coordinate (-3d) at (12,-9);
\coordinate (4d) at (16,-10);
\coordinate (-4d) at (12,-10);
\coordinate (5d) at (14,-11);
\filldraw [filling, d] (0d) -- (4,0) -- (1d) .. controls (3.5,-1) and (0.5,-2) .. (-2d) -- (-1d) to [bend right=15] (0d);
\filldraw [filling, d] (2d) arc (90:270:0.5) -- cycle;
\filldraw [filling, d] (-3d) .. controls (0.5,-3) and (3.5,-4) .. (4d) -- (4,-5) -- (5d) to [bend right=15] (-4d) -- cycle;
\draw [suture, d] (0d) to [bend left=15] (-1d);
\draw [suture, d] (2d) arc (90:270:0.5);
\draw [suture, d] (1d) .. controls (3.5,-1) and (0.5,-2) .. (-2d);
\draw [suture, d] (-3d) .. controls (0.5,-3) and (3.5,-4) .. (4d);
\draw [suture, d] (-4d) to [bend left=15] (5d);
\draw [boundary, d] (0,0) -- (4,0) -- (4,-5) -- (0,-5) -- cycle;
\foreach \point in {0d, 1d, 2d, 3d, 4d, 5d, -1d,-2d,-3d,-4d}
\fill [vertex] (\point) circle (2pt);
\draw (17,-8.5) node {$+$};
\coordinate (0e) at (20,-6);
\coordinate (1e) at (22,-7);
\coordinate (-1e) at (18,-7);
\coordinate (2e) at (22,-8);
\coordinate (-2e) at (18,-8);
\coordinate (3e) at (22,-9);
\coordinate (-3e) at (18,-9);
\coordinate (4e) at (22,-10);
\coordinate (-4e) at (18,-10);
\coordinate (5e) at (20,-11);
\filldraw [filling, e] (0e) -- (4,0) -- (1e) to [bend left=15] (0e);
\filldraw [filling, e] (-1e) .. controls (0.5,-1) and (3.5,-2) .. (2e) -- (3e) .. controls (3.5,-3) and (0.5,-4) .. (-4e) -- (-3e) arc (-90:90:0.5) -- (-1e);
\filldraw [filling, e] (4e) to [bend right=15] (5e) -- (4,-5) -- cycle;
\draw [suture, e] (0e) to [bend right=15] (1e);
\draw [suture, e] (-1e) .. controls (0.5,-1) and (3.5,-2) .. (2e);
\draw [suture, e] (-2e) arc (90:-90:0.5);
\draw [suture, e] (3e) .. controls (3.5,-3) and (0.5,-4) .. (-4e);
\draw [suture, e] (4e) to [bend right=15] (5e);
\draw [boundary, e] (0,0) -- (4,0) -- (4,-5) -- (0,-5) -- cycle;
\foreach \point in {0e, 1e, 2e, 3e, 4e, 5e, -1e,-2e,-3e,-4e}
\fill [vertex] (\point) circle (2pt);
\draw (23,-8.5) node {$+$};
\coordinate (0f) at (26,-6);
\coordinate (1f) at (28,-7);
\coordinate (-1f) at (24,-7);
\coordinate (2f) at (28,-8);
\coordinate (-2f) at (24,-8);
\coordinate (3f) at (28,-9);
\coordinate (-3f) at (24,-9);
\coordinate (4f) at (28,-10);
\coordinate (-4f) at (24,-10);
\coordinate (5f) at (26,-11);
\filldraw [filling, f] (0f) -- (4,0) -- (1f) to [bend left=15] (0f);
\filldraw [filling, f] (-1f) .. controls (0.5,-1) and (3.5,-2) .. (2f) -- (3f) .. controls (3.5,-3) and (0.5,-2) .. (-2f) -- cycle;
\filldraw [filling, f] (-3f) .. controls (0.5,-3) and (3.5,-4) .. (4f) -- (4,-5) -- (5f) to [bend right=15] (-4f) -- cycle;
\draw [suture, f] (0f) to [bend right=15] (1f);
\draw [suture, f] (-1f) .. controls (0.5,-1) and (3.5,-2) .. (2f);
\draw [suture, f] (-2f) .. controls (0.5,-2) and (3.5,-3) .. (3f);
\draw [suture, f] (-3f) .. controls (0.5,-3) and (3.5,-4) .. (4f);
\draw [suture, f] (-4f) to [bend left=15] (5f);
\draw [boundary, f] (0,0) -- (4,0) -- (4,-5) -- (0,-5) -- cycle;
\foreach \point in {0f, 1f, 2f, 3f, 4f, 5f, -1f, -2f, -3f, -4f}
\fill [vertex] (\point) circle (2pt);
\draw (5, -13) node {$=$};
\draw (8,-13) node {$\Gamma_{ppqq}$};
\draw (11,-13) node {$+$};
\draw (14,-13) node {$\Gamma_{pqqp}$};
\draw (17,-13) node {$+$};
\draw (20,-13) node {$\Gamma_{qppq}$};
\draw (23,-13) node {$+$};
\draw (26,-13) node {$\Gamma_{qpqp}$};
\draw (5,-15) node {$\cong$};
\draw (8,-15) node {$ppqq$};
\draw (11,-15) node {$+$};
\draw (14,-15) node {$pqqp$};
\draw (17,-15) node {$+$};
\draw (20,-15) node {$qppq$};
\draw (23,-15) node {$+$};
\draw (26,-15) node {$qpqp$};
\end{tikzpicture}
\]

We should point out that this is not the only interesting basis of $V_n$. In \cite{Me12_itsy_bitsy} we defined a large class of bases, one from any \emph{quadrangulation} of a surface.

\section{Contact topology}

As it turns out, all these curves on surfaces and stuffing into cylinders and $120^\circ$ rotations describe 3-dimensional contact topology.

\subsection{Chord diagrams and contact structures}

First of all, \emph{a chord diagram $\Gamma$ on a disc $D$ describes a contact structure $\xi_\Gamma$ on $D \times I$}. This contact structure $\xi_\Gamma$ consists of planes which are, roughly (and inaccurately) speaking,
\begin{itemize}
\item
Tangent to $D$ around the boundary of $D$
\item
``Perpendicular'' to $D$ precisely along $\Gamma$.\footnote{This statement makes no sense: a contact manifold has no metric, and no notion of perpendicularity! Rather, we should say that the vector field in the $I$ direction lies in $\xi$.}
\end{itemize}

In the simple example below of a disc with two chords, as we proceed from one side to the other, crossing both chords, the contact planes rotate through a full $360^\circ$. A rough (and inaccurate) interpretation of the contact condition of non-integrability is that if we follow a curve $C$ tangent to the contact structure $\xi$, dotted in the diagram, then the contact planes always rotate in the same direction about $C$. 

\begin{center}
\includegraphics[scale=1]{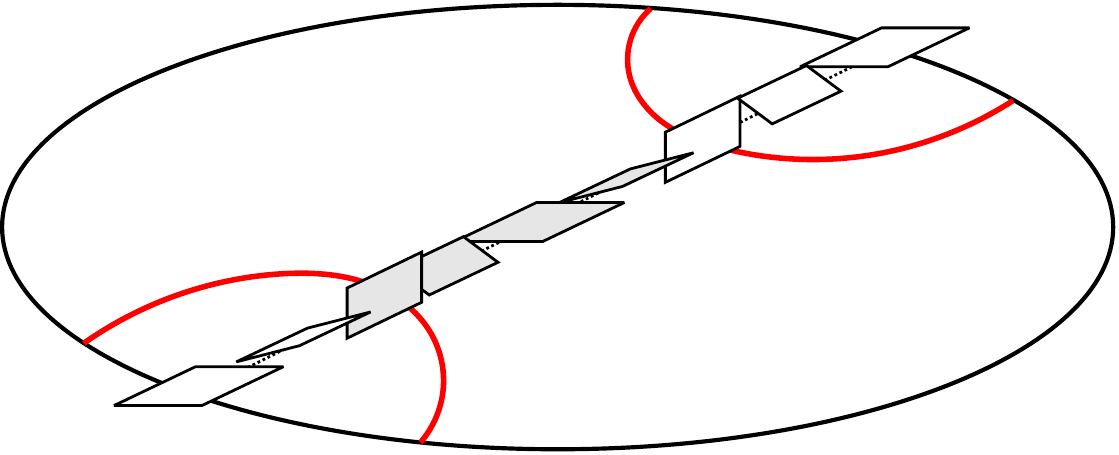}
\end{center}

If we colour one side of the contact planes white and the other side black (grey) then the black and white regions in a chord diagram correspond to which side of the contact planes are visible from above.
\begin{center}
\includegraphics[scale=1]{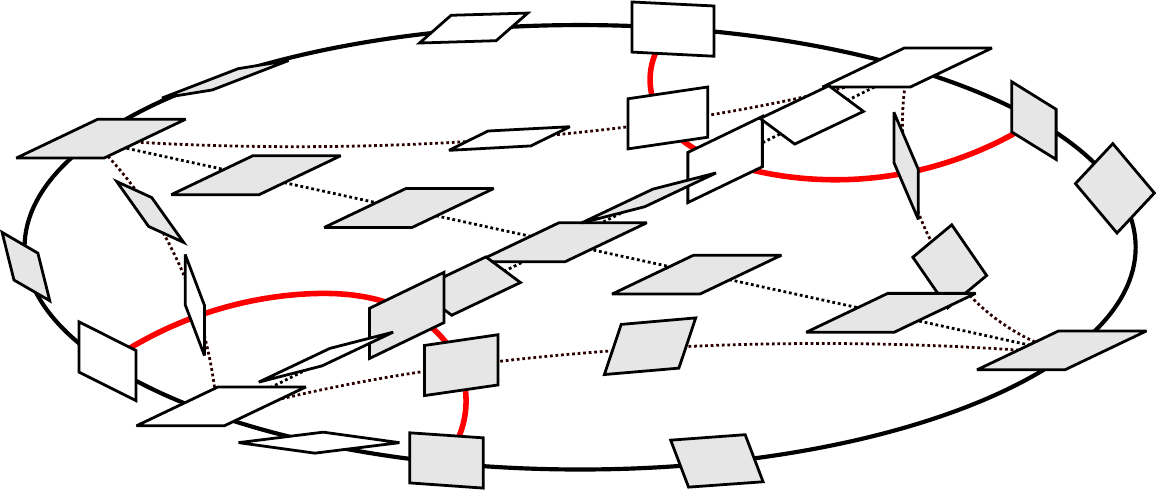}
\end{center}

It turns out that, in a sense which can be made precise, there is only one ``way up to isotopy'' to draw contact planes consistent with the chords. The lines of intersection of the contact planes with the disc $D$ trace out a foliation called the \emph{characteristic foliation}, some curves of which are dotted in the diagrams. The characteristic foliation determines the germ of a contact structure near $D$. The characteristic foliation is transverse to the chords $\Gamma$ (known as a \emph{dividing set} in contact geometry) and in fact can be directed by a vector field to exponentially dilate an area form on each component of $D \backslash \Gamma$ exiting through $\Gamma$. All such foliations compatible with $\Gamma$ give isotopic germs of contact structures. This is all part of the theory of \emph{convex surfaces} developed by Giroux \cite{Gi91} in 1991.

\subsection{Contact structures and overtwisted discs}

Given a 3-manifold $M$, it's an interesting and difficult question to find all the isotopy classes of contact structures on $M$.

Eliashberg in \cite{ElOT} showed that there are fundamentally two types of contact structures, called \emph{overtwisted} and \emph{tight}. An overtwisted contact structure is one that contains an object called an \emph{overtwisted disc}; a tight contact structure is one that does not.

An \emph{overtwisted disc} is a neighbourhood of a disc with a contact structure as shown. It corresponds to seeing a \emph{contractible closed curve} in a chord diagram / dividing set on a disc.

\begin{center}
\includegraphics[scale=1]{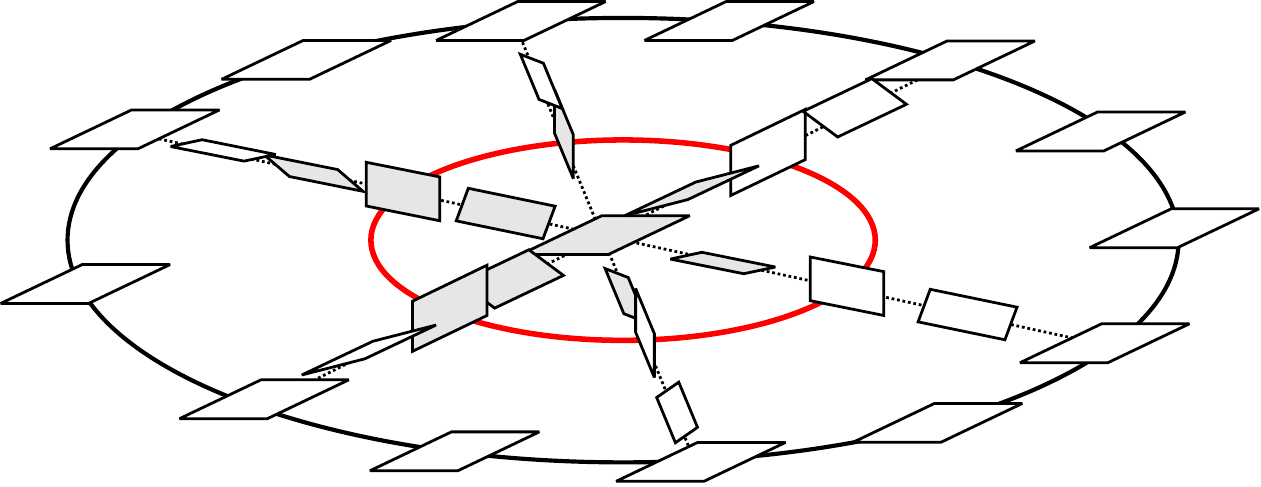}
\end{center}

Eliashberg in 1989 \cite{ElOT} proved that the classification of overtwisted contact structures on a 3-manifold $M$ is equivalent to the homotopy classification of plane fields on $M$. In this case the contact geometry reduces to homotopy theory (which is well understood) and contact geometry offers nothing new. So contact topologists tend to regard an overtwisted disc as ``spoiling'' a contact structure or rendering it trivial. It is quite surprising that the presence of one particular disc in a contact 3-manifold has such global topological consequences.

The \emph{tight} contact structures on $M$, on the other hand, offer important information about the topology of $M$. Not every 3-manifold has a tight contact structure, as Etnyre and Honda proved in 1999 \cite{EH99}; and the number of tight contact structures on a $3$-manifold depends intricately on the topology of $M$ (see, e.g., \cite{Gi00, GiBundles, Hon00I, Hon00II}).

This makes it quite reasonable that we regard a diagram on a disc, with a red closed curve, as ``trivial'', and counting it as $0$, such as in this figure seen previously:
\[
\begin{tikzpicture}[
scale=0.9, 
suture/.style={thick, draw=red},
boundary/.style={ultra thick},
vertex/.style={draw=red, fill=red}]

\coordinate [label = above:{$0$}] (0b) at ($ (90:2) + (8,0) $);
\coordinate [label = above right:{$1$}] (1b) at ($ (54:2) + (8,0) $);
\coordinate [label = above right:{$2$}] (2b) at ($ (18:2) + (8,0) $);
\coordinate [label = right:{$3$}] (3b) at ($ (-18:2) + (8,0) $);
\coordinate [label = below right:{$4$}] (4b) at ($ (-54:2) + (8,0) $);
\coordinate [label = below right:{$5$}] (5b) at ($ (-90:2) + (8,0) $);
\coordinate [label = below:{$-4$}] (6b) at ($ (-126:2) + (8,0) $);
\coordinate [label = below left:{$-3$}] (7b) at ($ (-162:2) + (8,0) $);
\coordinate [label = below left:{$-2$}] (8b) at ($ (162:2) + (8,0) $);
\coordinate [label = left:{$-1$}] (9b) at ($ (126:2) + (8,0) $);

\filldraw[fill=black!10!white, draw=none] (2b) arc (18:-18:2) -- (8b) arc (162:126:2) to [bend right=90] (0b) arc (90:54:2) to [bend right=90] (2b);
\filldraw[fill=black!10!white, draw=none] (4b) to [bend right=90] (5b) arc (-90:-54:2);
\filldraw[fill=black!10!white, draw=none] (6b) to [bend right=90] (7b) arc (-162:-126:2);
\filldraw[fill=black!10!white, draw=red] ($ (36:1.8) + (8,0) $) circle (0.1);
\draw[suture] (0b) to [bend left=90] (9b);
\draw[suture] (1b) to [bend right=90] (2b);
\draw[suture] (3b) -- (8b);
\draw[suture] (4b) to [bend right=90] (5b);
\draw[suture] (6b) to [bend right=90] (7b);

\draw [boundary] (8,0) circle (2 cm);

\foreach \point in {0b, 1b, 2b, 3b, 4b, 5b, 6b, 7b, 8b, 9b}
\fill [vertex] (\point) circle (2pt);

\end{tikzpicture}
\]

\subsection{Contact structures on spheres and balls}

Contact structures in the neighbourhood of a sphere $S^2$ are again given by dividing sets on $S^2$, which we continue to draw in red. It turns out that there is essentially only one tight contact structure in the neighbourhood of an $S^2$, and that is given by a \emph{single} connected curve. Any dividing set with more than one curve gives an overtwisted contact structure.

\begin{center}
\includegraphics[scale=0.3]{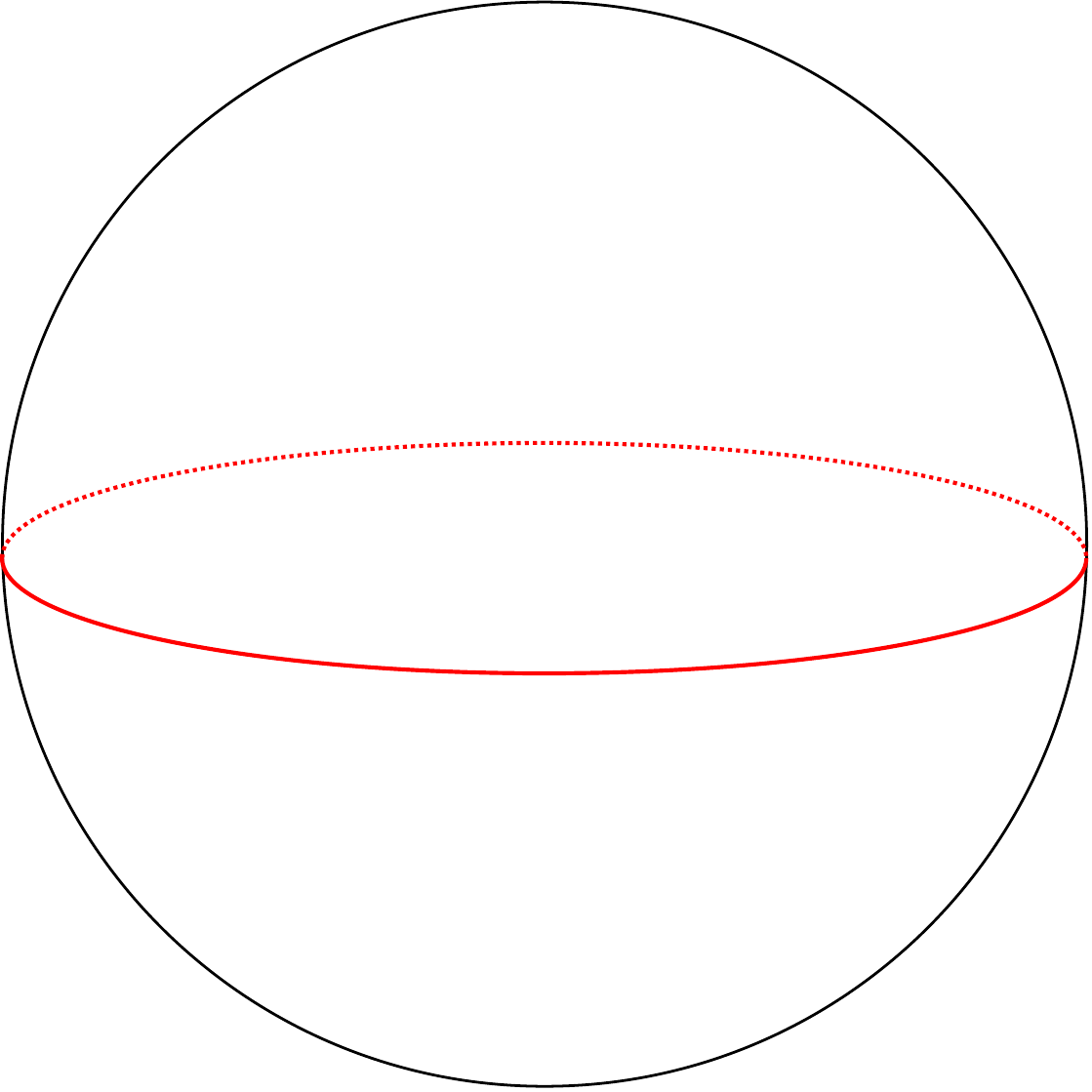}
\quad
\includegraphics[scale=0.3]{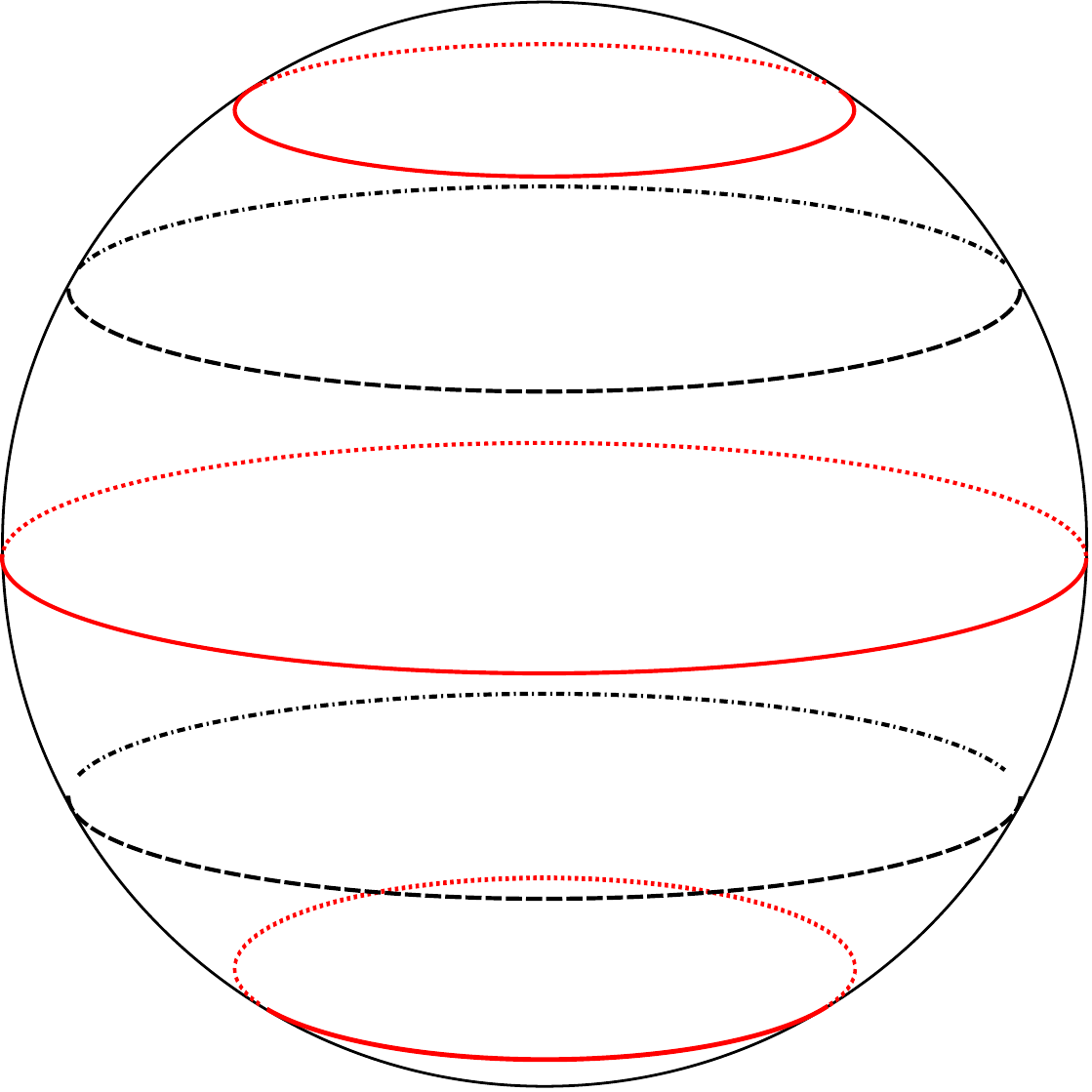}
\end{center}

The sphere on the left has a tight contact neighbourhood. The sphere on the right has an overtwisted contact neighbourhood, and  boundaries of two overtwisted discs are drawn in dashed black.

Now consider each sphere $S^2$ as the boundary of a ball $B^3$. Contact structures on balls were classified by Eliashberg in 1992 \cite{ElMartinet}. Putting aside the overtwisted structures, there is essentially only one \emph{tight} contact structure on the ball, up to isotopy. Given a dividing set on the boundary sphere which is connected, there is an essentially unique way to fill it in.

\subsection{Contact corners}

When considering a surface $S$ with boundary $\partial S$ in a contact 3-manifold $(M, \xi)$, we often require the contact planes to be tangent to the boundary $\partial S$. Above, all our diagrams of discs with contact planes have been drawn in this way.

When two surfaces meet along their boundary, forming a \emph{corner}, then the dividing sets do not intersect along the corner --- with our rough interpretation of dividing set as ``where contact planes are perpendicular to the surface'', this would mean that the contact planes are perpendicular to both surfaces meeting at the corner, as well as tangent to the corner, which is impossible.

Rather, the contact planes rotate around the corner curve, and the the dividing sets \emph{interleave} as shown.

\begin{center}
\includegraphics[scale=0.6]{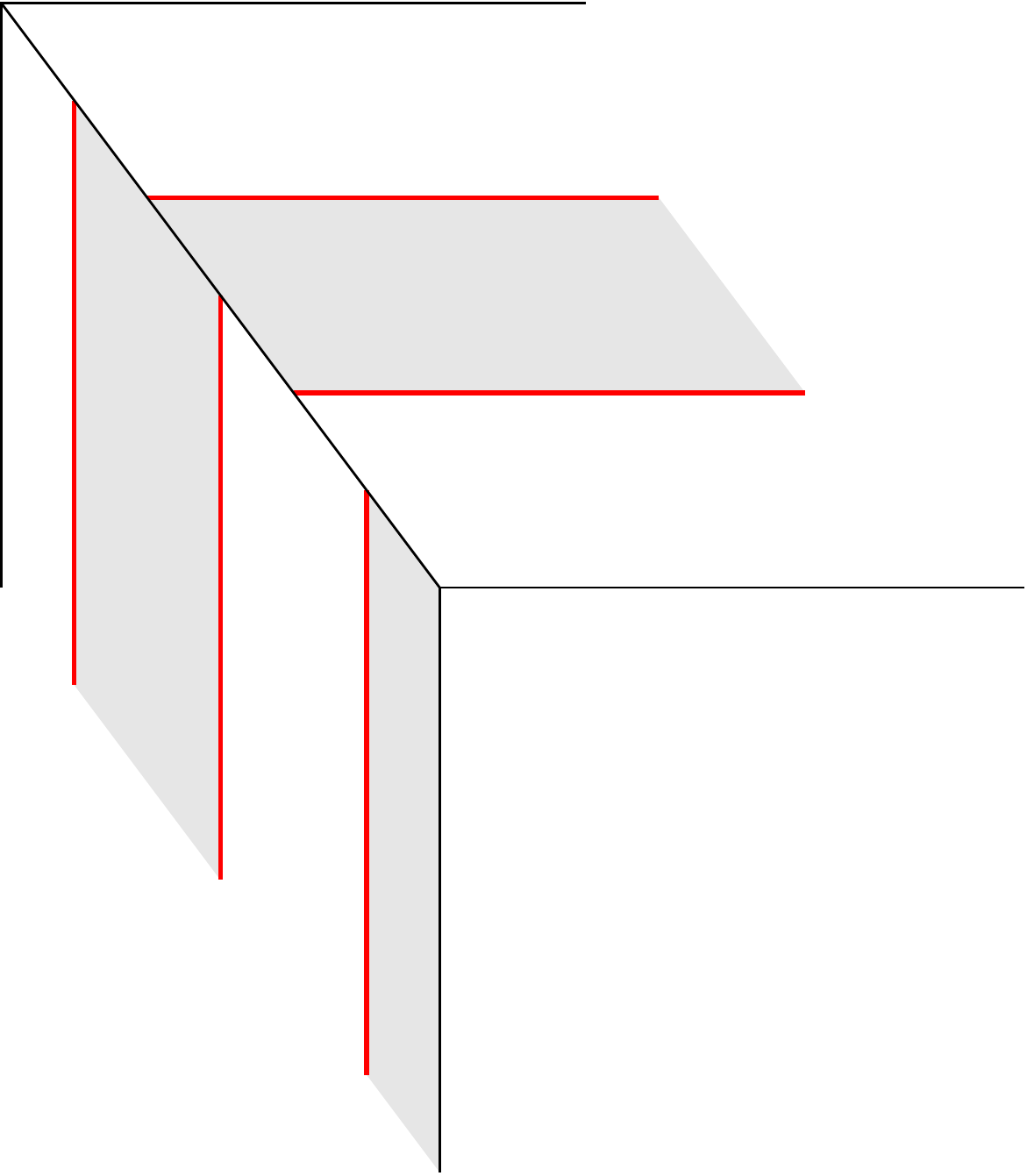}
\quad
\includegraphics[scale=0.6]{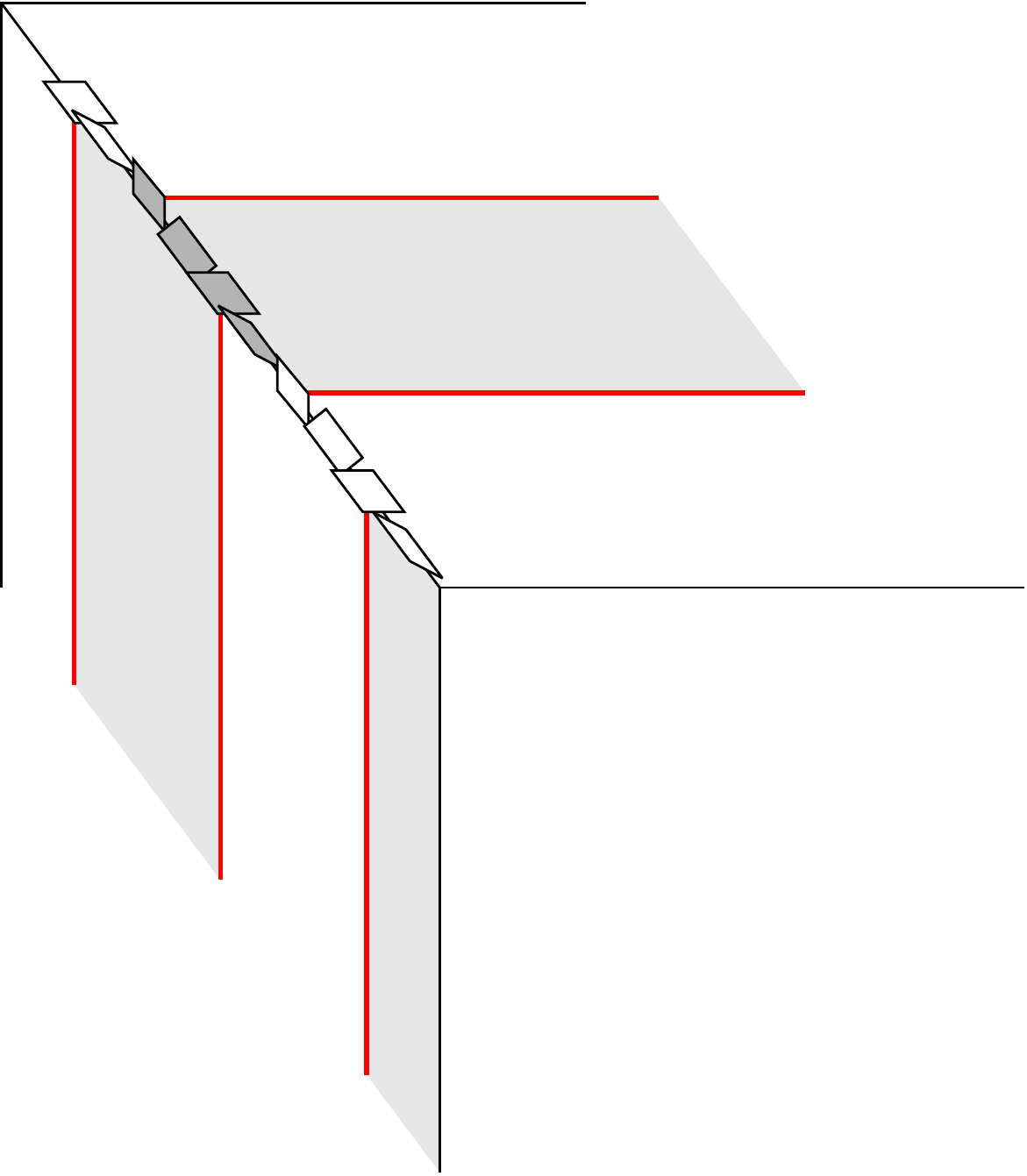}
\end{center}

It's not too difficult to believe, then, that if we \emph{round} the corners, then dividing curves behave as we proposed earlier.

\begin{center}
\includegraphics[scale=0.6]{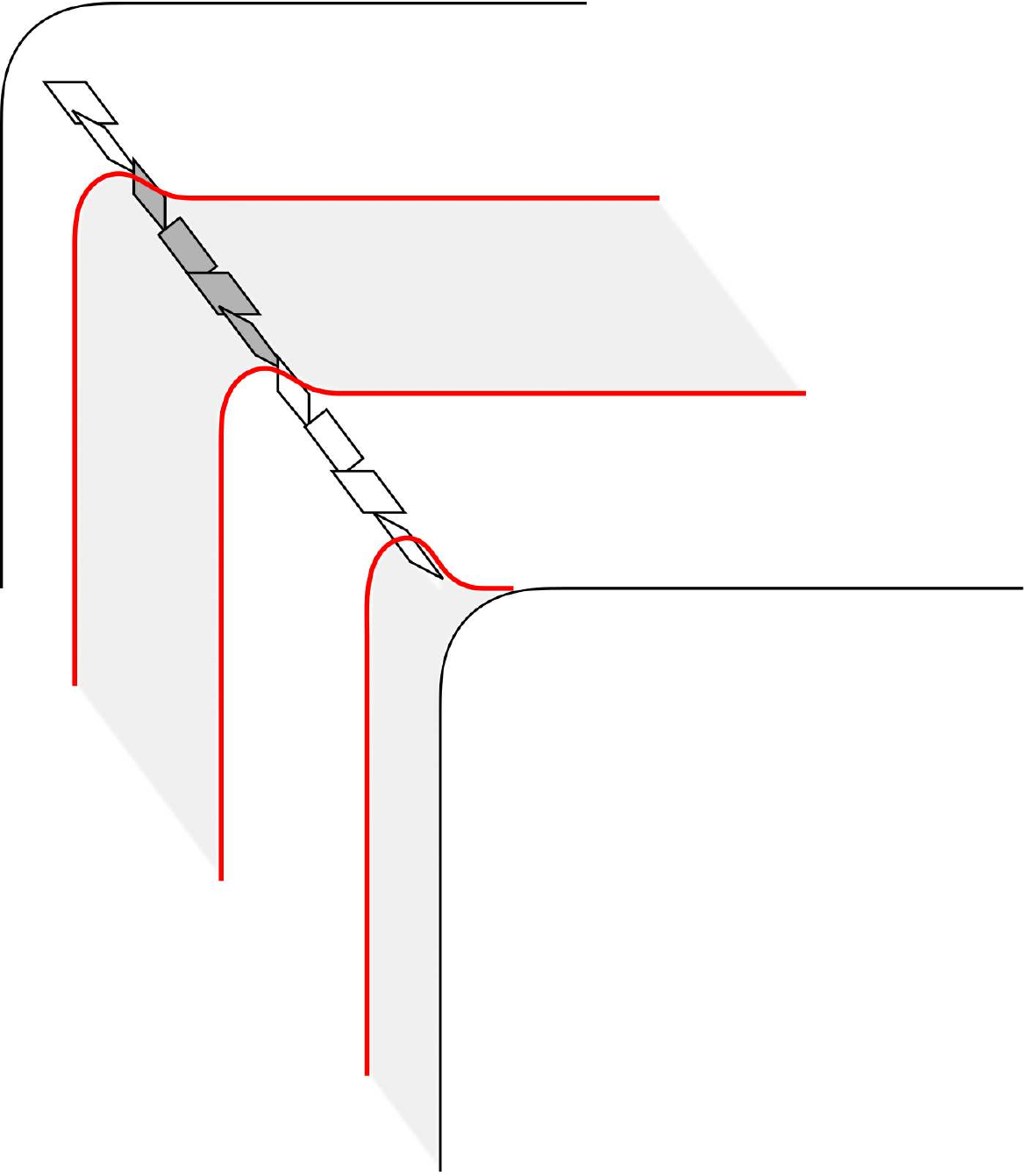}
\end{center}

It follows then that the ``inner product'' on chord diagrams, involving insertion into a cylinder, has a simple contact geometry interpretation.

\begin{prop}
Let $\Gamma_0, \Gamma_1$ be chord diagrams. The following are equivalent:
\begin{enumerate}
\item
$\langle \Gamma_0 | \Gamma_1 \rangle = 1$.
\item
The solid cylinder with dividing set $\Gamma_0$ on the bottom and $\Gamma_1$ on the top has a tight contact structure.
\end{enumerate}
\qed
\end{prop}

The ``finger moves'' we saw earlier for curves on the cylinder now correspond simply to isotopy of contact structures near the sphere, and in the ball.

\subsection{Bypasses}

Suppose we start from a surface $S$ with a dividing set on it --- such as a disc with a chord diagram. This describes a contact structure near $S$; effectively, on $S \times I$. We could then try to build up a contact manifold by gluing on more pieces on top of $S \times I$, and gluing them up.

It turns out that any contact 3-manifold can be built up in this way, using only fundamental building blocks called \emph{bypasses}. Drawn in terms of chord diagrams and cylinders, a bypass is actually something we saw earlier in computing $\langle \Gamma_{pq} \; | \; \Gamma_{qp} \rangle$. It is drawn below.

\begin{center}
\includegraphics[scale=0.5]{stack_2-eps-converted-to.pdf}
\end{center}

A bypass is nothing but a particular contact 3-ball. Ware consider the cylinder as bounding a 3-ball; from Eliashberg's theorem we know that the contact structure on the boundary of the cylinder extends uniquely throughout the ball to a tight contact structure.

Thus, the elementary building blocks of contact topology are, in a certain sense, precisely the elementary pawn moves
\[
\begin{array}{c}
pq \\
\begin{tabularx}{0.15\textwidth}{|X|X|}
\hline
\WhitePawnOnWhite & \BlackPawnOnWhite \\
\hline
\end{tabularx}
\end{array}
\quad \rightarrow \quad
\begin{array}{c}
qp \\
\begin{tabularx}{0.15\textwidth}{|X|X|}
\hline
\BlackPawnOnWhite & \WhitePawnOnWhite \\
\hline
\end{tabularx}
\end{array}
\]

However, beware! If you stack two bypasses directly on top of each other, you will find an overtwisted disc. We might say that the shortest step in contact topology is half way to oblivion.

\begin{center}
\includegraphics[scale=0.4]{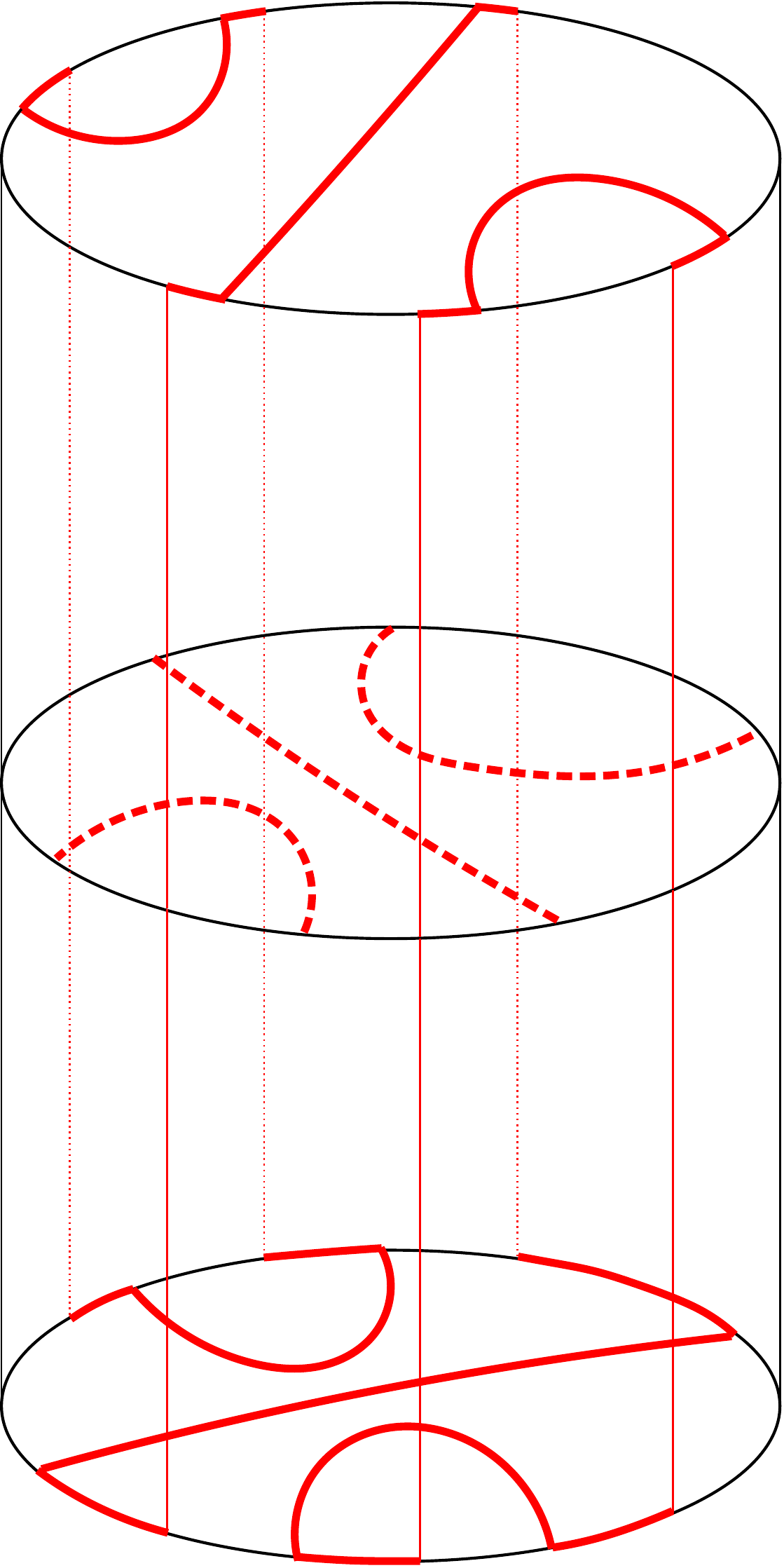}
\end{center}

To build a tight contact structure, you'll need to place bypasses in more sophisticated locations than just on top of each other.

We see that addition of a bypass on top of a disc performs a $120^\circ$ rotation in the chord diagram. This is precisely the operation described earlier and we defined the vector space $V_n$ by setting triples of bypass-related chord diagrams to sum to zero.

In fact, our creation and annihilation operators, which added or closed off curves in a chord diagram, can also be seen as building onto an existing contact structure.

\subsection{The contact category}

Honda has introduced the concept of the \emph{contact category} \cite{HonCat}. 
For a disc, the contact category $\mathcal{C}_n$ consists of objects and morphisms as follows.
\begin{enumerate}
\item
The \emph{objects} are chord diagrams with $n$ chords, i.e. contact structures near discs $D$.
\item
The \emph{morphisms} $\Gamma_0 \To \Gamma_1$ are contact structures on the cylinder $D \times I$, with $\Gamma_0$ on the bottom and $\Gamma_1$ on the top.
\item
\emph{Composition of morphisms} is given by stacking cylinders on top of each other.
\end{enumerate}

Honda showed that this category has many of the properties of a \emph{triangulated category}. In particular bypass triples behave very much like exact triangles. If $\Gamma, \Gamma', \Gamma''$ are successively obtained from each other by adding bypasses, then the composition of any two is overtwisted, hence $0$.
\[
\begin{tikzpicture}[
scale=0.9, 
suture/.style={thick, draw=red},
boundary/.style={ultra thick},
vertex/.style={draw=red, fill=red}]
\coordinate (0a) at ($ (90:1) + (-150:3) $);
\coordinate (1a) at ($ (30:1) + (-150:3) $);
\coordinate (2a) at ($ (-30:1) + (-150:3) $);
\coordinate (3a) at ($ (-90:1) + (-150:3) $);
\coordinate (4a) at ($ (-150:1) + (-150:3) $);
\coordinate (5a) at ($ (150:1) + (-150:3) $);
\filldraw[fill=black!10!white, draw=none] (4a) arc (210:150:1) to [bend right=90] (0a) arc (90:30:1) -- cycle;
\filldraw[fill=black!10!white, draw=none] (2a) to [bend right=90] (3a) arc (-90:-30:1);
\draw [boundary] (-150:3) circle (1 cm);
\draw [suture] (5a) to [bend right=90] (0a);
\draw [suture] (4a) -- (1a);
\draw [suture] (2a) to [bend right=90] (3a);
\draw (-150:1.5) node {$\Gamma'$};
\coordinate (0) at ($ (90:1) + (90:3) $);
\coordinate (1) at ($ (30:1) + (90:3) $);
\coordinate (2) at ($ (-30:1)  + (90:3) $);
\coordinate (3) at ($ (-90:1) + (90:3) $);
\coordinate (4) at ($ (-150:1) + (90:3) $);
\coordinate (5) at ($ (150:1) + (90:3) $);
\filldraw[fill=black!10!white, draw=none] (0) arc (90:30:1) to [bend right=90] (2) arc (-30:-90:1) -- cycle;
\filldraw[fill=black!10!white, draw=none] (4) to [bend right=90] (5) arc (150:210:1);
\draw [boundary] (90:3) circle (1 cm);
\draw [suture] (1) to [bend right=90] (2);
\draw [suture] (0) -- (3);
\draw [suture] (4) to [bend right=90] (5);
\draw (90:1.5) node {$\Gamma$};
\coordinate (0b) at ($ (90:1) + (-30:3) $);
\coordinate (1b) at ($ (30:1) + (-30:3) $);
\coordinate (2b) at ($ (-30:1) + (-30:3) $);
\coordinate (3b) at ($ (-90:1) + (-30:3) $);
\coordinate (4b) at ($ (-150:1) + (-30:3) $);
\coordinate (5b) at ($ (150:1) + (-30:3) $);
\filldraw[fill=black!10!white, draw=none] (2b) arc (-30:-90:1) to [bend right=90] (4b) arc (210:150:1) -- cycle;
\filldraw[fill=black!10!white, draw=none] (0b) to [bend right=90] (1b) arc (30:90:1);
\draw [boundary] (-30:3) circle (1 cm);
\draw [suture] (3b) to [bend right=90] (4b);
\draw [suture] (2b) -- (5b);
\draw [suture] (0b) to [bend right=90] (1b);
\draw (-30:1.5) node {$\Gamma''$};
\draw [->] (115:1.5) -- (-165:1.5);
\draw [->] (-135:1.5) -- (-45:1.5);
\draw [->] (-15:1.5) -- (75:1.5);
\end{tikzpicture}
\]

But we have seen here that chord diagrams themselves can be described by pawns and chessboards --- and the dynamics of pawns on chessboards is itself essentially a \emph{partial order}, which is a type of category. Our work in \cite{Me09Paper} defined a \emph{contact 2-category} out of this geometric structure.

In fact, looking now at the vector space $V_n$ we defined, we have essentially
\[
V_n = \frac{\Z_2 \langle \text{Chord diagrams with $n$ chords} \rangle }{ \text{Bypass relation} } = \frac{\Z_2 \langle \text{Objects of $\mathcal{C}_n$} \rangle }{ \text{Exact triangles sum to zero} }
\]
and thus $V_n$ is the \emph{Grothendieck group} of $\mathcal{C}_n$.

\subsection{Combinatorics of contact geometry}

To summarise, we've now seen various ways in which the combinatorics of chords on discs and cylinders, quantum pawn dynamics, and contact geometry are connected:
\begin{enumerate}
\item
Any chord diagram describes a contact structure near a disc.
\item
The ``inner product'' $\langle \cdot | \cdot \rangle$ can be interpreted as pawn dynamics on a chessboard, insertion of chord diagrams into a cylinder, or the existence of tight contact structures in solid cylinders.
\item
Creation and annihilation operators can be defined as pawn/antipawn creation/annihilation, or the insertion/closure of curves in a chord diagram, or building a contact structure.
\item
The adjoint property of creation and annihilation can be interepreted via the combinatorics of chessboards, finger moves on cylinders, or isotopy of contact structures on the sphere.
\item
A diagram on a disc with a closed curve describes an overtwisted contact structure, hence is ``trivial'' in terms of contact geometry. Similarly for curves on a cylinder with more than one component.
\item
Bypass triples of chord diagrams correspond to: building blocks of contact 3-manifolds; triples which must sum to zero for the ``insertion into cylinder inner product'' $\langle \cdot | \cdot \rangle$ to be nondegenerate; and exact triangles in the contact category.
\item
Chessboards correspond to a special class of chord diagrams, which look like slalom ski slopes, forming a basis for $V_n$, the Grothendieck group of the contact category of the disc.
\end{enumerate}

Much of this structure is similar to \emph{topological quantum field theories} (TQFT's). In particular, we have assigned vector spaces $V_n$ to discs, which are 2-dimensional. To cylinders, which we can think of as $(2+1)$-dimensional, or as the ``time evolution'' of a disc, we have associated an element of $\Z_2$, or a morphism in the contact category. And we have the ``inner product'' which describes not so much a probability amplitude, but a \emph{possibility amplitude} for a tight contact structure to evolve from one disc to another. We can call this structure \emph{contact TQFT} and summarise much of what we have said as
\[
\text{Contact TQFT} \quad \cong \quad \text{QPD}.
\]

\section{Holomorphic invariants}

In fact, all of this structure actually arose from some of the holomorphic invariants mentioned earlier.

\subsection{Sutured Floer homology}

Sutured Floer homology is a variant of Heegaard Floer homology. A sutured 3-manifold $(M, \Gamma)$ is (roughly) a 3-manifold $M$ with boundary, and some curves $\Gamma$ on the boundary $\partial M$, which divide $\partial M$ into positive and negative regions. Sutured 3-manifolds were studied by Gabai in the 1980s in the context of foliation theory \cite{Gabai83}, but also describe contact 3-manifolds with boundary --- indeed, the diagrams we have drawn with discs, red curves, and complementary regions coloured black and white are sutures. There are close relationships between foliation theory, sutured manifolds and contact topology \cite{HKM02, Eliashberg_Thurston}.

Sutured Floer homology, again speaking roughly and imprecisely, is defined as follows. A sutured 3-manifold\footnote{Strictly speaking, a \emph{balanced} sutured 3-manifold.} $(M, \Gamma)$ has a Heegaard decomposition, consisting of a Heegaard surface with boundary $\Sigma$, and two sets of curves $\alpha_1, \ldots, \alpha_k$ and $\beta_1, \ldots, \beta_k$ on $\Sigma$. The manifold can be recovered from the decomposition by gluing to $\Sigma \times [0,1]$ discs along $\alpha_i \times \{0\}$ and $\beta_i \times \{1\}$. The sutures lie along $\partial \Sigma \times [0,1]$, the negative part of the boundary is $\Sigma \times \{0\}$ surgered along the $\alpha_i \times \{0\}$, and the positive part of the boundary is $\Sigma \times \{1\}$ surgered along the $\beta_i \times \{1\}$. We then consider holomorphic curves 
\[
u \; : \; S \To \Sigma \times I \times \R
\]
where $S$ is a Riemann surface with boundary and with negative boundary punctures $p_1, \ldots, p_k$ and positive boundary punctures $q_1, \ldots, q_k$, satisfying conditions including the following (among others)\footnote{There are also some more technical conditions including conditions on the complex structure, non-constant preojections of $u$, and finite energy. We just highlight the fact that the boundary conditions come from 3-manifold topology. See \cite{Lip} and \cite{Ju06}.}:
\begin{enumerate}
\item
$u$ sends $\partial S$ to 
\[
\left( \left( \cup_i \alpha_i \right) \times \{1\} \times \R \right) \cup \left( \left( \cup_i \beta_i \right) \times \{0\} \times \R \right).
\]
\item
$u$ asymptotically sends each $p_i$ to $-\infty$ in the $\R$ coordinate, and each $q_i$ to $+\infty$.
\item
Each $u^{-1}(\alpha_i \times \{1\} \times \R)$ and $u^{-1} (\beta_i \times \{0\} \times \R)$ consists of precisely one segment of $\partial S \backslash \{p_1, \ldots, p_k, q_1, \ldots, q_k\}$.
\end{enumerate}

The point I want to make here is that, whatever the precise boundary conditions for holomorphic curves, they are defined in terms of the Heegaard decomposition. At $+\infty$, the curve $u(S)$ (or rather, its projection to $\Sigma$) runs through several intersections of $\alpha$ and $\beta$ curves. Indeed, at $+\infty$ we obtain a set of points $z_1, \ldots, z_k \in \Sigma$, such that
\[
z_1 \in \alpha_1 \cap \beta_{\sigma(1)}, \quad
z_2 \in \alpha_2 \cap \beta_{\sigma(2)}, \quad
\ldots, \quad
z_k \in \alpha_k \cap \beta_{\sigma(k)}
\]
for some permutation $\sigma$. We obtain a similar set of intersection points at $-\infty$.


We define a chain complex $\widehat{CF}$ generated over $\Z_2$ by sets of such intersection points,
\[
{\bf x} = \{z_1, \ldots, z_k\}.
\]
If we fix ${\bf x}$ and ${\bf y}$ to be two sets of intersection points, then we consider the moduli space $\mathcal{M}({\bf x}, {\bf y})$ with boundary conditions given by ${\bf x}$ at $+\infty$ and ${\bf y}$ at $-\infty$. The dimension of this moduli space can be given in terms of the topology of ${\bf x}$ and ${\bf y}$ (an \emph{index formula}). It follows, with some serious analysis, that we can define a differential to count index-1 curves
\[
\partial {\bf x} = \sum_{\dim \mathcal{M}({\bf x}, {\bf y}) = 1} \# \widehat{\mathcal{M}}({\bf x}, {\bf y}) \cdot {\bf y}
\]
and $\partial^2 = 0$. (Here $\widehat{\mathcal{M}}$ is the quotient of the moduli space $\mathcal{M}$ by the action of translation along $\R$; so $\widehat{\mathcal{M}}$ consists of finitely many points.) Then we may take the \emph{homology} and Ozsv\'{a}th--Szab\'{o} (for closed manifolds \cite{OS04Closed, OS04Prop}) and Juhasz (for sutured manifolds \cite{Ju06}) showed that this homology is \emph{independent of the choice of Heegaard decomposition} (and independent of other technical choices made along the way). That is, the resulting homology is a sutured manifold invariant which we may call $SFH(M, \Gamma)$.

Sutured Floer homology also has the nice property that any contact structure $\xi$ on $(M, \Gamma)$ gives a \emph{contact element} $c(\xi)$ in $SFH(M, \Gamma)$ \cite{HKM09}.\footnote{This element is only defined up to sign; however if we take $\Z_2$ coefficients, no ambiguity arises.}

It turns out that our subject matter in this note has been sutured Floer homology. In particular, $SFH$ of solid tori $D^2 \times S^1$. If we let $F_n$ consist of $2n$ points on the boundary $\partial D^2$, so that $F_n \times S^1$ forms a set of sutures on $D^2 \times S^1$, then we have the following.
\begin{thm}[\cite{Me09Paper}]
\[
V_n \cong SFH(D^2 \times S^1, F_n \times S^1).
\]
Moreover, any chord diagram $\Gamma$ in $V_n$ corresponds to a tight contact structure $\xi_\Gamma$ on $(D^2 \times S^1, F_n \times S^1)$, unique up to isotopy, and the isomorphism takes $\Gamma \mapsto c(\xi_\Gamma)$.
\end{thm}

Therefore, all of the structure we have been discussing, is contained in sutured Floer homology of one of the simplest sutured 3-manifolds, namely a solid torus.

\subsection{Embedded contact homology and string homology}

Another holomorphic invariant is \emph{embedded contact homology}, defined by Hutchings \cite{Hutchings02}. Given a 3-manifold $M$, we consider a contact structure $\xi$ and a contact form $\alpha$. From the 1-form $\alpha$ there is a natural vector field called the \emph{Reeb vector field} $R$, which is defined by 
\[
d\alpha \left( R, \cdot \right) = 0, \quad \alpha(R) = 1.
\]
Embedded contact homology counts holomorphic curves in the \emph{symplectization} $M \times \R$ with symplectic form $d(e^t \alpha)$, where $t$ is the coordinate on $\R$. The most important condition prescribed on holomorphic curves is that the curves must approach \emph{Reeb orbits}, i.e. orbits of $R$, as $t \to \pm \infty$. Roughly, given certain collections of Reeb orbits
\[
{\bf \gamma^+} = \{ \gamma_1^+, \ldots, \gamma_n^+ \} 
\quad \text{and} \quad
{\bf \gamma^-} = \{ \gamma_1^-, \ldots, \gamma_n^- \},
\]
(some of which might be repeated and some of which might be covered multiple times), we consider the moduli space $\mathcal{M}({\bf \gamma^+}, {\bf \gamma^-})$ of \emph{embedded} holomorphic curves which approach ${\bf \gamma^+}$ at $+\infty$ and ${\bf \gamma^-}$ at $-\infty$. Again there is an index formula giving the dimension of the moduli space in terms of the contact geometry of the $\gamma_i^\pm$ and again, after a lot of work, it is possible to define a differential
\[
\partial {\bf \gamma^+} = \sum_{\dim \mathcal{M}({\bf \gamma^+}, {\bf \gamma^-}) = 1} \# \widehat{\mathcal{M}}({\bf \gamma^+}, {\bf \gamma^-}) \cdot {\bf \gamma^-}
\]
where $\partial^2 = 0$. The homology of this complex is \emph{embedded contact homology}, $ECH$. It's possible to define $ECH$ also for sutured manifolds \cite{CGHH}. 

It has recently been shown that embedded contact homology is isomorphic to Heegaard Floer homology (see e.g. \cite{CGH10}). Note that $ECH$, even though it is defined in terms of a contact form, turns out not to depend on the contact structure at all, but is a smooth manifold invariant.

This deep isomorphism between $ECH$ and $HF$ implies that we should be able to obtain all of the combinatorial and algebraic structure we have discussed above, from embedded contact homology. Based on work of Cieliebak--Latschev \cite{CL}, as discussed in \cite{Mathews_Schoenfeld12_string}, one is led to consider the following ideas.

Take a disc $D$ with $2n$ points $F$ marked on the boundary, as we have seen with chord diagrams. But now consider sets of curves on $D$ with boundary $F$, which are not necessarily chord diagrams --- the curves may intersect. We call these \emph{string diagrams}.

We define a $\Z_2$ vector space freely generated by string diagrams, up to homotopy relative to boundary. On this vector space, there is a \emph{differential} $\partial$ defined by \emph{resolving crossings}. That is, given a string diagram $s$, $\partial s$ is the sum of string diagrams, each obtained by resolving one of the crossings of $s$ as shown.

\begin{center}
\begin{tikzpicture}[scale=1, string/.style={thick, draw=red, -to}]
\draw [string] (-1,0) -- (1,0);
\draw [string] (0,-1) -- (0,1);
\draw [shorten >=1mm, -to, decorate, decoration={snake,amplitude=.4mm, segment length = 2mm, pre=moveto, pre length = 1mm, post length = 2mm}]
(1.5,0) -- (2.5,0);
\draw [string] (3,0) -- (3.7,0) to [bend right=45] (4,0.3) -- (4,1);
\draw [string] (4,-1) -- (4,-0.3) to [bend left=45] (4.3,0) -- (5,0);
\end{tikzpicture}
\end{center}

It turns out that $\partial^2 = 0$, and the homology is something with which we are by now familiar.
\begin{thm}[\cite{Mathews_Schoenfeld12_string}]
\[
HS \cong V_n \cong \frac{\Z_2 \langle \text{chord diagrams on $(D^2, F)$} \rangle }{ \text{Bypass relation} }
\]
\end{thm}

The ``reason'' for this relation --- which is far from a proof --- is the following picture.

\begin{center}
\begin{tikzpicture}[
scale=1.2, 
string/.style={thick, draw=red, postaction={nomorepostaction, decorate, decoration={markings, mark=at position 0.5 with {\arrow{>}}}}}]

\draw (-6,0) circle (1 cm);
\draw (3,0) circle (1 cm); 	
\draw (0,0) circle (1 cm);
\draw (-3,0) circle (1 cm);

\draw [string] ($ (-6,0) + (-90:1) $) to [bend right=45] ($ (-6,0) + (90:1) $);
\draw [string] ($ (-6,0) + (30:1) $) to [bend right=45] ($ (-6,0) + (210:1) $);
\draw [string] ($ (-6,0) + (150:1) $) to [bend right=45] ($ (-6,0) + (-30:1) $);

\draw [string] (30:1) arc (120:240:0.57735);
\draw [string] (0,-1) -- (0,1);
\draw [string] (150:1) arc (60:-60:0.57735);

\draw [string] (3,0) ++ (150:1) arc (-120:0:0.57735);
\draw [string] (3,0) ++ (30:1) -- ($ (3,0) + (210:1) $);
\draw [string] (3,0) ++ (-90:1) arc (180:60:0.57735);

\draw [string] (-3,0) ++ (30:1) arc (-60:-180:0.57735);
\draw [string] ($ (-3,0) + (150:1) $) -- ($ (-3,0) +  (-30:1) $);
\draw [string] (-3,-1) arc (0:120:0.57735);

\draw (-7.5,0) node {$\partial$};
\draw (-4.5,0) node {$=$};
\draw (-1.5,0) node {$+$};
\draw (1.5,0) node {$+$};

\end{tikzpicture}
\end{center}

This idea of taking vector spaces generated by topological classes of curves, and resolving their crossings, is closer to the sorts of objects studied in \emph{string topology} (e.g. \cite{CS}) and as such we can say that contact topology and sutured Floer homology are expressed as a \emph{string homology}. Further details and generalisations of this result are given in \cite{Mathews_Schoenfeld12_string}.

\addcontentsline{toc}{section}{References}

\small

\bibliography{danbib}
\bibliographystyle{amsplain}

\end{document}